\documentclass[]{interact}

\usepackage{lipsum}
\usepackage{graphicx}
\usepackage{epstopdf}
\usepackage{amsfonts,amssymb,amsmath}
\usepackage{algorithmic,algorithm}
\usepackage{setspace}
\usepackage{comment}
\usepackage{cite}
\usepackage{pdflscape}
\usepackage{color}
\usepackage{graphicx}
\usepackage{multirow}

\newtheorem{assumption}{Assumption}

\usepackage{enumitem}

\allowdisplaybreaks
\newcommand{\norm}[1]{\left\lVert#1\right\rVert}
\ifpdf
  \DeclareGraphicsExtensions{.eps,.pdf,.png,.jpg}
\else
  \DeclareGraphicsExtensions{.eps}
\fi

\usepackage{pgfplots,tikz}
\pgfplotsset{compat=1.3}

\usetikzlibrary{calc}

\newcommand{\TheTitle}{\vspace{-1.2cm}Distributed optimization based on Gradient-tracking Revisited: enhancing convergence rate via surrogation } 

\title{{\TheTitle}\thanks{Submitted to the editors May 7, 2019; Revised October 2020.\newline  \indent  The authors are with the School of   Industrial Engineering, Purdue University, West-Lafayette, IN, USA. Emails: \email{$<$sun578,adaneshm,gscutari$>$@purdue.edu}.
\funding{This work  has been supported by the USA National Science Foundation under Grants  CIF 1719205, CIF 1564044, CMMI 1832688; and  the Army Research Office under Grant W911NF1810238.}}}

\author{
  Dianne Doe\thanks{Imagination Corp., Chicago, IL
    (\email{ddoe@imag.com}, \url{http://www.imag.com/\string~ddoe/}).}
  \and
  Paul T. Frank\thanks{Department of Applied Mathematics, Fictional
    University, Boise, ID (\email{ptfrank@fictional.edu},
    \email{jesmith@fictional.edu}).}
  \and
  Jane E. Smith\footnotemark[3]
}

\usepackage{amsopn}

\usepackage{amsmath}
\DeclareMathOperator*{\argmin}{argmin}

\renewcommand*{\theassumption}{\Alph{assumption}}

\usepackage{multirow}
\usepackage{pifont}

\newcommand{\wavg}[2]{\bar{\mathbf{#1}}_{\boldsymbol{\phi}}^{#2}}

\newcommand{\var}[2]{\mathbf{#1}_{\boldsymbol{\phi},\bot}^{#2}}
\newcommand{\Deltax}{\mathbf{d}}
\newcommand{\Deltaxi}[1]{\mathbf{d}_{#1}}
\newcommand{\optgap}{p_{\boldsymbol{\phi}}}

\newcommand{\bx}{\mathbf{x}}

\newcommand{\pgrad}{\nabla\mathbf{f}}

\newcommand{\hbx}{\widehat{\mathbf{x}}}

\newcommand{\hbW}{\widehat{\mathbf{W}}}

\newcommand{\bu}{\mathbf{u}}
\newcommand{\bn}{\mathbf{n}}

\newcommand{\by}{\mathbf{y}}
\newcommand{\bz}{\mathbf{z}}
\newcommand{\bA}{\mathbf{A}}

\newcommand{\bb}{\mathbf{b}}
\newcommand{\bH}{\mathbf{H}}
\newcommand{\bW}{\mathbf{W}}

\newcommand{\bJ}{\mathbf{J}}
\newcommand{\bI}{\mathbf{I}}
\newcommand{\1}{\mathbf{1}}
\newcommand{\0}{\mathbf{0}}

\newcommand{\bphi}{\boldsymbol{\phi}}

\newcommand{\bdelta}{\boldsymbol{\delta}}

\newcommand{\tf}{\widetilde{f}}
\newcommand{\tF}{\widetilde{F}}
\newcommand{\tH}{\widetilde{\mathbf{H}}}
\newcommand{\tL}{\widetilde{L}}
\newcommand{\tU}{\widetilde{U}}

\newcommand{\KK}{\mathcal{K}}
\newcommand{\GG}{\mathcal{G}}
\newcommand{\EE}{\mathcal{E}}
\newcommand{\PP}{\mathcal{P}}
\newcommand{\RR}{\mathcal{R}}

\newcommand{\VV}{\mathcal{V}}

\newcommand{\real}{\mathbb{R}}
\renewcommand{\natural}{\mathbb{N}}

\newcommand{\tmu}{\widetilde{\mu}}

\newcommand{\seqnorm}[1]{ #1^K(z)}

\DeclareMathOperator{\mx}{mx}
\DeclareMathOperator{\mn}{mn}

\makeatletter

\usepackage{makecell}
\usepackage[makeroom]{cancel}
\usepackage{marginnote}
\usepackage{hyperref}

\usepackage{epstopdf}
\usepackage{subfigure}

\usepackage[numbers,sort&compress]{natbib}
\bibpunct[, ]{[}{]}{,}{n}{,}{,}

\theoremstyle{plain}
\newtheorem{theorem}{Theorem}[section]
\newtheorem{lemma}[theorem]{Lemma}
\newtheorem{corollary}[theorem]{Corollary}
\newtheorem{proposition}[theorem]{Proposition}

\theoremstyle{definition}
\newtheorem{definition}[theorem]{Definition}

\begin{document}
	
	\articletype{~}
	
	\title{Distributed Optimization Based on Gradient-tracking Revisited: Enhancing Convergence Rate via Surrogation}
	
	\author{
		\name{Ying Sun
			\thanks		{
				\indent The authors are with the School of   Industrial Engineering, Purdue University, West-Lafayette, IN, USA.  Emails: $<$sun578,adaneshm,gscutari$>$@purdue.edu.
				\newline
				\indent \textbf{Funding:} This work  has been supported by the USA National Science Foundation under Grants  CIF 1719205, CIF 1564044, CMMI 1832688; and  the Army Research Office under Grant W911NF1810238..
			}, 
			Amir Daneshmand, and Gesualdo Scutari
		}
	}
	\maketitle

\begin{abstract}{ 
We study distributed  
 multiagent optimization over (directed, time-varying) graphs. We consider the   minimization of $F+G$ subject to   convex constraints, where $F$ is the smooth strongly convex   sum of   the agent's losses and $G$ is a nonsmooth convex function.   We build on the SONATA algorithm:  the algorithm employs the use of surrogate objective functions    in the agents' subproblems (going  thus  beyond  linearization, such as proximal-gradient) coupled    with a   perturbed (push-sum)
consensus mechanism that aims to track locally the gradient of $F$. 
 SONATA achieves precision $\epsilon>0$ on the objective value in $\mathcal{O}(\kappa_g \log(1/\epsilon))$ gradient computations at each node and  $\tilde{\mathcal{O}}\big(\kappa_g (1-\rho)^{-1/2} \log(1/\epsilon)\big)$ communication steps, where $\kappa_g$ is the condition number of   $F$ and $\rho$ characterizes the connectivity of the network. This is the first linear rate result for distributed composite optimization; it   also  improves on   existing (non-accelerated) schemes just minimizing $F$, whose rate  depends  on much larger quantities than  $\kappa_g$ (e.g., the worst-case condition number among the agents). When considering  in particular empirical risk minimization problems with statistically similar data across the agents,   SONATA 
 employing high-order surrogates   achieves precision $\epsilon>0$ in $\mathcal{O}\big((\beta/\mu) \log(1/\epsilon)\big)$ iterations   and  $\tilde{\mathcal{O}}\big((\beta/\mu) (1-\rho)^{-1/2} \log(1/\epsilon)\big)$ communication steps, where $\beta$ measures the degree of similarity of the agents' losses and $\mu$ is the strong convexity constant of $F$.   Therefore, when  $\beta/\mu < \kappa_g$, the use of high-order surrogates yields provably faster rates than what achievable by first-order models; this is without exchanging any Hessian matrix over the network.}
\vspace{-0.2cm}
\end{abstract}

\begin{keywords}
Distributed optimization, gradient tracking, linear rate, machine learning,   statistical similarity, surrogate functions.
\newline 
\newline
Submitted on May 2019; Revised on Oct. 2020
\end{keywords}

\section{Introduction} \label{sec:intro} We study distributed    optimization over networks in the form: 
\vspace{-0.4cm}
\begin{equation}\label{eq:P}
\begin{array}{cl}
\underset{\mathbf{x}}{\min} & U(\bx)\triangleq\underbrace{\frac{1}{m}\sum_{i=1}^{m}f_{i}(\bx)}_{F(\bx)}+G(\bx)\vspace{-0.2cm}\\
\text{s.t.} & \bx\in\KK,
\end{array}\tag{P}\vspace{-0.2cm}
\end{equation} where   $f_i:\mathbb{R}^d\to \mathbb{R}$ is the loss function of agent $i$, 
 assumed to be smooth and  convex   while $F$ is strongly convex on $\mathcal{K}$;    $G:\mathbb{R}^d\to \mathbb{R}$ is      a   nonsmooth convex function on $\mathcal{K}$;  and $\mathcal{K}\subseteq \mathbb{R}^d$ represents the set of common  convex constraints. Each $f_i$ is known to the associated agent only. 
Agents are connected through a communication network, modeled as a  graph, possibly  directed and/or time-varying. The goal is to   
cooperatively solve  \eqref{eq:P} by exchanging  information only with their immediate neighbors.

Distributed  optimization in the form \eqref{eq:P}  has found a wide range of applications in several areas, including network information processing, telecommunications,    multi-agent control, and machine learning. {  An instance of particular interest to this work  is the distributed Empirical Risk Minimization (ERM) whereby the goal is to minimize the average loss over some dataset, distributed across the nodes of the network (cf. Sec.~\ref{sec_related}). Letting $\mathcal{D}^{(i)}=\{\mathbf{z}_1^{(i)},\ldots ,\mathbf{z}_n^{(i)}\}$ the dataset of $n$ examples available at node $i$'s side, the local empirical loss reads   $f_i(\bx)=1/n\sum_{j=1}^n f(\bx; \mathbf{z}_j^{(i)})$, where $f(\bx; \mathbf{z}_j^{(i)})$  measures the fit between the parameter $\bx$ and the  sample  $\mathbf{z}_j^{(i)}$. 
Data sets are usually large and high-dimensional, which makes routing local data to other agents (let alone to a centralized node) infeasible or highly inefficients. Given the cost of communications (especially if compared with the speed of local processing),  the challenge in such a network setting is designing communication efficient distributed algorithms.}   

  {  Motivated by the  aforementioned  applications, our focus  pertains to such a design in two possible settings (one being a special case of the other) \cite{Arjevani-ShamirNIPS15}:  {\bf 1)} The scenario where  no significant relationship can
be assumed among the local functions $f_i$--this is  what the   literature of distributed optimization has extensively studied, and will be refereed to as  the  {\it unrelated} setting---and {\bf 2)} the case where the $f_i$'s are {\it  related}, e.g., because they reflect statistical similarity in the data residing at different nodes. For instance, in the   distributed ERM problem above, when data are i.i.d.   among machines, one can show that quantities such as the gradients and Hessian matrices of the local functions differ only by $\beta=\mathcal{O}(1/\sqrt{n})$, due to concentrations of measure effects   \cite{DANE,DISCO}--we will refer to this as {\it $\beta$-related} setting (cf. Sec.~\ref{sec_related}). 
 If properly exploited in the algorithmic design, such similarity   can  speed up the
optimization/learning process over general purpose optimization algorithms. 

 \noindent  \paragraph*{\bf Centralized algorithms
 } Problem~\eqref{eq:P} in the two settings above has been extensively studied in the centralized environment, including   star-networks wherein there is  a master node   connected to all the other workers. Our interest is in the following (non-accelerated) algorithms: 

 {\it  1)  Unrelated setting:} \eqref{eq:P} can be solved on star-networks  employing the standard   proximal gradient method:   
 to reach precision $\epsilon>0$  on the objective value, one needs $\mathcal{O}\big(\kappa_g \log(1/\epsilon)\big)$ iterations (which is also the number of   communication rounds between the master and the workers), where $\kappa_g$ is the condition number of $F$. 
 
{\it  2)  $\beta$-related setting:} When the agents' functions $f_i$ are sufficiently similar,  a linear rate proportional  to $\kappa_g$ may be  highly suboptimal. For instance, in the  extreme case where all $f_i$'s are identical ($\beta=0$), the number of iterations/communications to an $\epsilon>0$ solution  would remain the same as for $\beta=\mathcal{O}(L)$.  In fact, when $1+\beta/\mu < \kappa_g$, faster rates can be   obtained exploiting the similarity of the  $f_i$'s. Specifically,    \cite{DANE} proposed DANE: a mirror-descent type algorithm over star-networks, where each worker $i$ replaces  the quadratic term in its local proximal-gradient update   with the Bregman divergence of the reference function  $f_i+\beta/2\|\bullet\|^2$;  and  the master averages the solutions  of the workers. DANE  is applicable to \eqref{eq:P} with  $G=0$: For  {\it quadratic} losses, it achieves an $\epsilon$-solution 
in $\mathcal{O}\big((\beta/\mu)^2\cdot \log(1/\epsilon)\big)$ iterations/communications (it is assumed $\beta/\mu\geq 1)$  
 while no improvement is proved over the proximal gradient   if the $f_i$'s are not quadratic. More recently, \cite{Fan2020}  proposed   
 CEASE,  which achieves  DANE's rate      for  \eqref{eq:P} with 
    $G\neq 0$ and    nonquadratic losses. Using recent results in \cite{LuFreundNesterov2020}, it is not difficult to check that the mirror-descent algorithm implemented at the master (thus  without averaging workers' iterates)   with the Bregman divergence of    $f_1+\beta/2\|\bullet\|^2$ ($f_1$ is the local function  at the master) achieves  an $\epsilon>0$ solution    in  $\widetilde{\mathcal{O}}\big(\beta/\mu \cdot\log(1/\epsilon)\big)$ iterations/communications,   improving thus  on DANE/CEASE's rates.

A natural question is whether similar results--in particular the dependence of the rate on global optimization parameters as obtained on star-networks in the unrelated and $\beta$-related settings--are achievable over general network topologies, possibly time-varying and directed. 
The literature  of distributed algorithms over  general network topologies--albeit vast--do not provide a satisfactory answer, leaving a gap between  rate results over star networks and what has been certified over general graphs--see  Sec~\ref{sec:related_works} for a   review of the state of the art. In a nutshell, {\bf (i)} there are no distributed schemes   provably achieving linear rate for \eqref{eq:P} with  $G\neq 0$ and/or constraints (cf. Table~\ref{convex_table}). Furthermore,    even considering  the unconstrained  minimization of $F$ (i.e., $G\!=\!0$ and $\mathcal{K}\!=\!\mathbb{R}^d$),  {\bf (ii)}  linear convergence is certified at a rate    depending on much larger quantities than the global condition number $\kappa_g$--see Table \ref{tab:optim_linear_rate}; and {\bf (iii)} when $1+\beta/\mu< \kappa_g$ ($\beta$-related setting), no rate improvement is provably achieved  by   existing distributed algorithms. These are much more pessimistic rate dependencies than what achieved over star-topologies. The goal of this paper is to close exactly this gap. }

{  
\subsection{Major contributions} Our major results are summarized next.   \begin{enumerate}
	\item  We provide the first  linear convergence rate analysis of a distributed algorithm, SONATA    (Successive cONvex Approximation algorithm over Time-varying digrAphs), applicable to the   {\it composite, constrained} formulation \eqref{eq:P} over (time-varying, directed)  graphs.     SONATA  was earlier proposed in the companion paper \cite{SONATA-companion}  for nonconvex problems.  It combines the use of surrogate functions in the agents' subproblems with a  perturbed (push-sum) consensus mechanism that aims at locally  tracking  the gradient of $F$.   
	Surrogate functions replace  the more classical first order approximation of the local $f_i$'s, which  is the omnipresent   choice in current distributed algorithms, offering   the potential to  better suit the geometry of the problem. For instance, (approximate) Newton-type subproblems or mirror descent-type updates  naturally  fit our surrogate models; they are the  key enabler of provably faster rates in the $\beta$-related setting. We comment SONATA's rates below (cf.  Table~\ref{Table-SONATA}).

	\item {\bf Unrelated setting (Table~\ref{Table-SONATA}):} When the network is sufficiently connected or it has a star-topology, SONATA reaches  an $\epsilon$-solution on the objective value in $\mathcal{O}\big(\kappa_g\log(1/\epsilon)\big)$ iterations/communications, which matches the  rate   of  the centralized proximal-gradient algorithm. For arbitrary network connectivity, the same iteration complexity is achieved at the cost of $\mathcal{O}((1-\rho)^{-1/2})$ rounds of communications per iteration (employing Chebishev acceleration), where  $\rho\in [0,1)$  is the  second  largest eigenvalue modulus of the mixing matrix.  Our rates improve on those of existing   distributed algorithms which    show  a much more pessimistic dependence on the optimization parameters and are proved under more restrictive assumptions
	--contrast Table~\ref{tab:optim_linear_rate} with Table~\ref{Table-SONATA}. Linear rates over time-varying digraphs are reported in   Table~\ref{Table-SONATA-TV} (cf. Sec.~\ref{sec:linear-rate-TV}).
	  
  \item  {\bf $\beta$-related setting (Table~\ref{Table-SONATA}):}  When the agents' functions are sufficiently similar  (specifically, $1+\beta/\mu<\kappa_g$), the use of a mirror descent-type surrogate   over linearization of the $f_i$'s   provably yields    faster rates,  
   at higher computation costs.  This improves on the rate of existing distributed algorithms, which are oblivious of function similarity (cf. Table~\ref{tab:optim_linear_rate}).  Notice that this is achieved without exchanging any Hessian matrix over the network but leveraging function  homogeneity via surrogation. 
When customized over star-topologies, SONATA's rates improve on DANE/CEASE's ones too. 
  

 \end{enumerate}

}
 \renewcommand{\arraystretch}{1.3}
\begin{table}[t!]  
\centering   
\resizebox{\textwidth}{!}{
	\begin{tabular} {l c | cccccccc}		
				\hline	
		{\bf   Algorithms}	& 						&    \cite{qu2016harnessing,shi2015extra,ling2015dlm,mokhtari2015dqm,shi2014linear,6926737,li2017decentralized,Jakovetic:da, Maros:2018vo}
		&   
		\cite{8267245,xi2015linear, Pu:ir, zeng2015extrapush}		
			 
			& \cite{Maros:2019ia,nedic2017geometrically,nedich2016achieving,yuan2018exact_p2, Khan-TV-AB18}
			
			&
				{\bf  SONATA}\\
			\hline 
			\multirow{3}{*}{\begin{tabular}[c]{@{}l@{}}\textbf{Problem:} \end{tabular} } &  \textbf{$F$ (smooth)} 	&	each $f_i$ scvx		&		each	$f_i$ scvx 
			&  $F$ scvx &  $F$ scvx
			\\	
			\cline{2-6}
			& \textbf{$G$ (nonsmooth)}   		&  																				&	 	   																		 	   					 & & \checkmark								
						\\	
			\cline{2-6}
			& \textbf{constraints $\mathcal{K}$}   		&	 														&	 	  												
			& & \checkmark															
			\\
			\hline 
			\multirow{2}{*}{\begin{tabular}[c]{@{}l@{}}\textbf{Network:} \end{tabular}  } &  \textbf{time-varying} 				&	 	only \cite{Maros:2018vo} 																	 	&   & 			only \cite{nedich2016achieving, Khan-TV-AB18}															  	&			 	\checkmark							
			\\	
			\cline{2-6}
			& \textbf{digraph}   		&			 																	&	\checkmark		 			 			  & only \cite{nedich2016achieving, Khan-TV-AB18}&	\checkmark																																 
			\\
			\hline 
	\end{tabular}
	}\medskip
	\centering 
	\caption{Existing linearly convergent distributed  algorithms. SONATA is the only scheme achieving linear rate in the presence of     $G$ in \eqref{eq:P} or constraints.  The explicit expression of the rates of the above nonaccelerated schemes (for which is available)  is reported in Table~\ref{tab:optim_linear_rate}.}
	\label{convex_table} \vspace{-.7cm}
\end{table}

\renewcommand{\arraystretch}{1.7}
\begin{table}[h!]\label{tab:optim_linear_rate}
	\centering
	 \resizebox{0.72\textwidth}{!}{
	\begin{tabular}{c|c|c}
		\hline
		 \small 
		{\bf Algorithm} & \small {\bf Problem}  &\small  {\bf Linear rate:} $\small \mathcal{O}\big(\delta\,\log(1/\epsilon)\big)$
		\\
		\hline \hline
	\small 	 EXTRA~\cite{shi2015extra} & \small $F$  &  \small $\delta=
	\mathcal{O}\big(\frac{\kappa_{\ell}^{2}}{1-\rho}\big)$,\quad ${\kappa_\ell}=\frac{L_{\mx}}{\mu_{\mn}}$  
		\\
		\hline
	\small 	 DIGing~\cite{nedich2016achieving,nedic2017geometrically}   & $F$       & \small $\delta=\frac{\hat{\kappa}^{1.5}}{(1-\rho)^2}$,\quad \small $\hat{\kappa}\triangleq \frac{L_{\mx}}{({1}/{m})\sum_i\mu_i}$  
		\\
		\hline
		 Harnessing~\cite{qu2016harnessing}   & $F$     & \small $\small\delta=\frac{\kappa_{\ell}^2}{(1-\rho)^2}$ 
		 \\
		\hline
	\small 	 NIDS~\cite{li2017decentralized}, \small ABC \cite{XU_at_al20}    & $F$    & \small $\delta=\max\big\{\kappa_{\ell},\, \frac{1}{1-\rho}\big\}$   
		\\
		\hline
		\small  Exact Diffusion~\cite{yuan2018exact_p2}    & \small $F$    & 
		\small 	$\delta=\frac{\bar{\kappa}^2}{1-\rho}$, \quad \small $\bar{\kappa}\triangleq \frac{L_{\mx}}{\mu_{\mx}}$
		\\
		\hline
		  \parbox{4cm}{\centering \small Augmented Lagrangian 
		  \small \cite{6926737}}    & \small $F$   & \small $\delta=\frac{\kappa_{\ell}}{1-\rho}$  
		\\
		\hline
		 \parbox{4cm}{\centering \small ADMM  \cite{shi2014linear} }   & $F$   & 
		$\small \frac{\kappa_{\ell}^4}{1-\rho}$
		\\
		\hline
	\end{tabular}}\medskip 
\caption{{ Linear rate of existing  non-accelerated algorithms over undirected graphs:  communications rounds to reach $\epsilon>0$ accuracy; $L_i$ and $\mu_i$ are the smoothness  and strong convexity constants of   $f_i$'s, respectively; $L_{\mx}\triangleq  max_{i} L_i$, $\mu_{\mn} \triangleq  \min_{i} \mu_i$; and   $\rho\in [0,1)$  is the  second  largest eigenvalue modulus of the mixing matrix [cf. (\ref{eq:rho_def})]. The rates above include the quantities   $\kappa_l$,  $\hat{\kappa}$,  and $\breve{\kappa}$ rather than the much desirable global condition number $\kappa_g\triangleq L/\mu$ ($L$ and $\mu$ are the smoothness and strong convexity constants of $F$, respectively). Furthermore, they are independent on $\beta$, implying that faster rates are not certified when $1+\beta/\mu <\kappa_g$ ($\beta$-related setting).\vspace{-0.5cm}}}
\end{table}

 \renewcommand{\arraystretch}{2}%
\begin{table}[h!]
\resizebox{\columnwidth}{!}{\begin{tabular}{|c|c|c|c|c|}
\hline 
{\bf Surrogate} & {\bf Communication Rounds} & {\bf Extra Averaging }& {\bf $\rho$ (network) }& $\beta$ \tabularnewline
\hline 
\hline 
\multirow{2}{3cm}{\centering linearization} & $\small\mathcal{O}\left(\kappa_g\,\log\left(1/\epsilon\right)\right)$ & \ding{55} & $\begin{array}{ccc}&\ensuremath{\rho=\mathcal{O}(\kappa_g^{-1}(1+\frac{\beta}{L})^{-2})} \vspace{-0.3cm}\\ & \vspace{-0.3cm}\text{or}\\ \vspace{-0.3cm}& \text{star-networks}\vspace{0.2cm}\\\end{array}$  & arbitrary\tabularnewline
\cline{2-5} 
 & $\widetilde{\mathcal{O}}\left(\dfrac{\kappa_g}{\sqrt{1-\rho}}\,\log(1/\epsilon)\right)$ & \ding{51} & arbitrary & arbitrary\tabularnewline
\hline 
\multirow{4}{3cm}{\centering local $f_{i}$} & $\small\mathcal{O}\left(1\cdot\log\left(1/\epsilon\right)\right)$ & \ding{55} & $\begin{array}{ccc}&\rho=\mathcal{O}\left( {\left(1 + \frac{\beta}{\mu}\right)^{-2}\left(\kappa_g + \frac{\beta}{\mu}\right)^{-2}}\right)\vspace{-0.3cm}\\ & \vspace{-0.3cm}\text{or}\\ \vspace{-0.3cm}& \text{star-networks}\vspace{0.2cm}\\\end{array}$  & \multirow{2}{3cm}{\centering$\beta\leq\mu$}\tabularnewline
\cline{2-4} 
 & $\widetilde{\mathcal{O}}\left(\dfrac{1}{\sqrt{1-\rho}}\,\log(1/\epsilon)\right)$ & \ding{51} & arbitrary & \tabularnewline
\cline{2-5} 
 & $\small\mathcal{O}\left(\dfrac{\beta}{\mu}\cdot\log\left(1/\epsilon\right)\right)$ & \ding{55} & $\begin{array}{ccc}&\rho=\mathcal{O}\left(\left(1 + \frac{L}{\beta}\right)^{-1} \left(\kappa_g + \frac{\beta}{\mu}\right)^{-1}\right)\vspace{-0.3cm}\\ & \vspace{-0.3cm}\text{or}\\ \vspace{-0.3cm}& \text{star-networks}\vspace{0.2cm}\\\end{array}$ & \multirow{2}{3cm}{\centering $\beta>\mu$}\tabularnewline
\cline{2-4} 
 &  $\widetilde{\mathcal{O}}\left( \dfrac{{\beta}/{\mu}}{\sqrt{1 - \rho_0}} \cdot\log (1/\epsilon)\right)$ & \ding{51} & arbitrary & \tabularnewline
\hline  
\end{tabular}}\medskip \caption{{Summary of convergence rates of SONATA over undirected graphs: number of communication rounds to reach   $\epsilon$-accuracy. In the table, 
$\beta$ is the homogeneity parameter measuring the similarity of the loss functions $f_i$'s (cf. Definition \ref{assump:homogeneity}); the other quantities are defined as in Table~\ref{tab:optim_linear_rate}. 
 The extra averaging steps are performed using Chebyshev acceleration \cite{auzinger2011iterative,pmlr-v70-scaman17a}. The $\widetilde{O}$ notation hides log dependence on $\kappa_g$ and $\beta/\mu$ (see Sec.~\ref{sec_arbitrary-topology} for the exact expressions). Rates   over time-varying directed graphs are summarized in Table~\ref{Table-SONATA-TV} (cf. Sec.~\ref{sec:linear-rate-TV}).\vspace{-0.9cm}}}\label{Table-SONATA}\end{table}

\subsection{Related works}\label{sec:related_works}
  Early works on distributed optimization aimed at decentralizing the (sub)gradient algorithm.    The Distributed  Gradient Descent (DGD) was introduced in  \cite{Nedic2009} for   unconstrained instances of \eqref{eq:P}  and in \cite{Lopes:ja} for   least squares, bot over undirected graphs. 
  A refined convergence rate analysis of DGD   \cite{Nedic2009} can be found  in \cite{Yuan:2016dv}. 
 Subsequent variants  of DGD include  the projected (sub)gradient algorithm \cite{nedic2010constrained} and   
  the push-sum gradient consensus algorithm \cite{nedic2015distributed}, the latter implementable  over 
  digraphs. 
 While  different, the updates of the agents' variables in   the above  algorithms   can be   abstracted as a combination of one (or multiple) consensus step(s)  (weighted average with neighbors variables)  
 and a  local (sub)gradient descent step, 
 controlled by a step-size (in some schemes, followed by a  proximal operation).     
 A diminishing step-size is used to reach {\it exact} consensus on the solution,  converging thus at a {\it sublinear rate}. 
With a fixed step-size $\alpha$,  linear rate of the iterates  is achievable, but it can only converge to  a $\mathcal{O}(\alpha)$-neighborhood of the  solution  \cite{Nedic2009,Yuan:2016dv}.

 Several subsequent attempts have been proposed   to cope with this speed-accuracy dilemma, leading to algorithms converging to the {\it exact} solution while employing a {\it constant} step-size. Based upon the mechanism put forth to cancel the steady state error in the individual gradient direction,  existing proposals can be roughly   organized in  three groups, namely: i) primal-based  distributed methods leveraging the idea of  gradient tracking  \cite{Xu2015augmented,LorenzoScutari-C'15,LorenzoScutari-J'16,qu2016harnessing,NaLi_Allerton_2016, nedich2016achieving,XiKha-J'16,xi2015linear,8267245,Xin:uf,Pu:ir,Khan-AB-Nesterov18}; ii) distributed schemes using ad-hoc corrections of the local optimization direction \cite{shi2015extra,zeng2015extrapush,Berahas:bs}; and iii) primal-dual-based methods  \cite{shi2014linear,ling2015dlm,mokhtari2015dqm,6926737,pmlr-v70-scaman17a}.  We elaborate next on these works, focusing   on   schemes achieving linear rate--{  Table~\ref{convex_table} organizes these schemes based upon  the setting their convergence is established 
 while Table~\ref{tab:optim_linear_rate} reports the explicit expression of the rates.}
 
  {\bf i) Gradient-tracking-based methods:} 
 In these schemes, each agent updates its own  variables along a  direction that  tracks the global gradient   $\nabla F$. This idea was   proposed independently in the NEXT algorithm \cite{LorenzoScutari-C'15,LorenzoScutari-J'16} for Problem \eqref{eq:P}  and in AUG-DGM \cite{Xu2015augmented} for strongly convex, smooth, unconstrained optimization. The work     \cite{SunScutariPalomar-C'16} introduced SONATA, extending  NEXT  over  (time-varying) digraphs. 
  A convergence rate analysis of  \cite{Xu2015augmented} was later developed in \cite{qu2016harnessing,nedich2016achieving,Xu-TAC:hs}, with \cite{nedich2016achieving} considering also (time-varying)  digraphs. Other algorithms based on the idea of gradient tracking and implementable over digraphs are ADD-OPT  \cite{XiKha-J'16} and \cite{8267245}. Subsequent schemes,  
   \cite{xi2015linear}, the Push-Pull \cite{Pu:ir}, and the $\mathcal{AB}$ \cite{Khan-TV-AB18} algorithms,      relaxed previous  conditions on the  mixing matrices used in the consensus and gradient tracking steps over digraphs, which neither  need to be   row- nor column-stochastic. 
 All the   schemes above but NEXT and SONATA are applicable only to {\it smooth, unconstrained} instances of \eqref{eq:P}, with {\it each $f_i$ strongly convex}. 
 This latter   assumption 
 is   restrictive in some applications, such as  distributed  machine learning, where not all $f_i$ are strongly convex but  $F$ is so.  



{\bf ii) Ad-hoc gradient correction-based methods:} These methods developed specific corrections of   the plain DGD direction. Specifically, EXTRA \cite{shi2015extra} and its variant over digraphs, EXTRA-PUSH \cite{zeng2015extrapush}, introduce two different weight matrices for any two consecutive iterations  as well as leverage history of gradient information. 
 They are  applicable only to {it smooth, unconstrained} problems; when each $f_i$ is  strongly convex, they generate iterates that converge linearly to  the minimizer of  $F$.  To deal with an additive   convex nonsmooth term in the objective, \cite{shi2015proximal} proposed 
 PG-EXTRA, which is  
 thus applicable to~\eqref{eq:P} over undirected graphs, {     possibly with different local nonsmooth functions}. However, linear convergence is not certified. 
 A different approach 
 is to use a linearly increasing number of consensus steps rather than correcting directly the gradient direction; this has been   studied in  \cite{Berahas:bs} for unconstrained minimization of smooth, strongly convex $f_i$'s over undirected graphs.

{\bf iii)  Primal-dual  methods:} 
 A common theme of these schemes is employing a prima-dual reformulation of the original multiagent problem whereby dual variables associated to    a properly defined (augmented) Lagrangian function serve the purpose of  correcting  the plain DGD local direction. 
   Examples of such algorithms include: i) distributed ADMM  methods \cite{jakovetic2011cooperative,shi2014linear} and their inexact implementations \cite{ling2015dlm,Maros:2019ia}; ii) distributed Augmented Lagrangian-based methods  with     randomized primal variable updates \cite{6926737};  
   and iii) a distributed dual ascent method employing  tracking of the average of the primal variable  \cite{Maros:2018vo}. 
  All these   schemes   are   applicable  only to {\it smooth, unconstrained} optimization over undirected graphs, with \cite{Maros:2018vo} handling time-varying  graphs.   
   The  extension of these methods  to   digraphs seems  not straightforward, because it is not clear how to enforce consensus via   constraints over   directed networks.

 To summarize, the above literature review   shows that currently there  exists no distributed algorithm for the general formulation   \eqref{eq:P} that provably converges at linear rate to the exact solution, 
in the presence of a nonsmooth function $G$ or constraints (cf. Table  \ref{convex_table}); let alone mentioning digraphs.  Furthermore, when it comes to the dependence of the rate on the optimization parameters, Table~\ref{Table-SONATA} shows that, even restricting to unconstrained, smooth minimization, SONATA's rates improve on existing ones--in particular, SONATA provably obtains fast convergence  if the agents' objective functions (e.g., data)  are sufficiently similar. 
  

  
 {   
\paragraph*{Concurrent works} 
  While our manuscript was under review and available  on arXiv \cite{SONATA-arxiv}, a few other related technical reports appeared online \cite{alghunaim2019linearly,rogozin2019projected,NetDane}, which we briefly discuss next. The authors in \cite{alghunaim2019linearly} studied a 
  class of  distributed proximal gradient-based  methods to solve   Problem~\eqref{eq:P} with $G\neq 0$, over undirected, static, graphs. The algorithms  reach an $\epsilon$-solution in   	$\mathcal{O}\big(\breve{\kappa}(1-\rho)^{-1}\log(1/\epsilon)\big)$ iterations/communications, where  $\breve{\kappa}\triangleq L_{\mx}/\mu$. The authors in \cite{rogozin2019projected} proposed an inexact  distributed projected gradient descent method for the unconstraint minimization of $F$ and proved a communication complexity of  $\tilde{O}\big(\kappa_g\,(1-\rho)^{-1}\log^2(1/\epsilon)\big)$ ($\tilde{\mathcal{O}}$ hides a log-dependence on $L_{\max}^2/\mu^2$), which is determined by  the global condition number $\kappa_g$; the algorithm runs over time-varying, undirected,  graphs (as long as they are  connected at each iteration). SONATA's rates compare favorably with those above. Furthermore, since both schemes \cite{alghunaim2019linearly} and \cite{rogozin2019projected}  are gradient-type methods, unlike SONATA, their performance   cannot  benefit from function similarity, resulting in  convergence rates independent on $\beta$.  On the other hand,    \cite{NetDane} explicitly considered the  $\beta$-related setting, and proposed Network-DANE, a decentralization of the DANE algorithm. It turns out that Network-DANE is a special case of SONATA; there are however some important differences in the convergence analysis/results.   First, convergence in \cite{NetDane} is established only for the {\it unconstrained} minimization of $F$ ($G=0$ and $\mathcal{K}=\mathbb{R}^d$) over  {\it undirected }graphs, with {\it each} $f_i$ assumed to be strongly convex. Second, convergence rates therein are more pessimistic than what predicted by our analysis. In fact,  the best communication complexity of  Network-DANE  reads  $\tilde{O}\big( (1+(\beta/\mu)^2)(1-\rho)^{-1/2}\log(1/\epsilon)\big)$ for quadratic   $f_i$'s and worsens  to  $\tilde{O}\big(\kappa_\ell (1+\beta/\mu)(1-\rho)^{-1/2}\log(1/\epsilon)\big)$ for nonquadratic losses. Note that the latter  is of the order of   the worst-case rate of first-order methods,   which do    not benefit from function similarity. A direct comparison with Table~\ref{Table-SONATA}, shows that SONATA' rates exhibit a better dependence on the optimization parameters ($\kappa_g$ vs. $\kappa_\ell$) and $\beta/\mu$  in all scenarios. In particular, in the $\beta$-related setting,  SONATA  retains   faster rates, even when $f_i$'s are   nonquadratic.  }

\vspace{-0.1cm}
\subsection{Paper organization}  Sec.~\ref{sec:problem_statement} introduces  the main assumptions on the optimization problem and network, along with some motivating examples from machine learning. The SONATA algorithm over  undirected graphs is studied in Sec.~\ref{sec:SONATA-NEXT}; in particular, linear convergence is proved in  Sec.~\ref{sec:linear_rate}, while {  a detailed discussion on the rate expression and its scalability properties is provided in Sec.~\ref{sec:discussion}.}  The case of time-varying, possibly directed, graphs is considered in Sec.~\ref{sec:TV-case}. {  Finally, some numerical results supporting our theoretical findings are reported in Sec. \ref{sec:num}.  The study of SONATA when $F$ is nonconvex can be found in the technical report \cite{SONATA-arxiv}. }
\vspace{-0.1cm}
 
\section{Problem \& Network Setting} \label{sec:problem_statement} This section summarizes the   assumptions on the optimization problem and network setting. We also introduce a general  learning problem over networks, which will be used as case study throughout the   paper.    
\vspace{-0.1cm}

\subsection{Assumptions on Problem \eqref{eq:P}}
{   Our algorithmic design and convergence results pertain to two   problem settings, namely: i) the one where the local functions $f_i$ are generic and unrelated (cf. Sec. \ref{sec_unrelated}), and ii) the case where they are related (cf. Sec. \ref{sec_related}). These two settings are formally introduced below. \vspace{-0.1cm}
} 
\subsubsection{The unrelated setting}\label{sec_unrelated}
  Consider the following standard assumption.
  \begin{assumption}[On Problem \eqref{eq:P}]\label{assump:p}
\begin{enumerate}[leftmargin=*,label=\theassumption\arabic*]
\item The set $\emptyset \neq \KK \subseteq \real^d$ is closed and convex; 
\item Each $f_i: \mathcal{O} \to \real $ is twice differentiable on the open set $\mathcal{O} \supseteq \KK$ and convex; 
\item $F$ satisfies $$\mu \mathbf{I} \preceq \nabla^2 F(\bx) \preceq L \mathbf{I}, \quad \forall \bx\in \mathcal{K},$$ with  $\mu >0$ and $0<L<\infty$;
\item $G: \KK \to \real$ is convex possibly nonsmooth.
\end{enumerate}
\end{assumption}
Note that \ref{assump:p}3 together with \ref{assump:p}2 imply 
\begin{equation}\label{eq:mu-L-smooth}\mu_i \mathbf{I} \preceq \nabla^2 f_i(\mathbf{x}) \preceq L_i \mathbf{I},\quad \forall \bx\in \mathcal{K},\,\,\forall i\in [m],\end{equation} for some   $\mu_i\geq 0$ and  $0<L_i<\infty$. Unlike  existing works (cf. Table~\ref{convex_table}), we do not require each $f_i$ to be strongly convex but just $F$ (cf. \ref{assump:p}3). Also,   twice differentiability of $f_i$ is not really necessary, but assumed here to simplify our derivations. 

{  Under Assumption \ref{assump:p}, we define 
the global  conditional number  associated to   \eqref{eq:P}:\vspace{-0.2cm}
 \begin{equation}\label{eq:cond_number}
 	\kappa_g\triangleq \frac{L}{\mu}.\vspace{-0.2cm}
 \end{equation}
Related quantities determining the (linear) convergence rate of existing distributed algorithms are (cf. Table \ref{tab:optim_linear_rate}): \vspace{-0.2cm}
 \begin{equation}\label{eq:cond_number_others}
 \kappa_{\ell} \triangleq \frac{L_{\mx}}{\mu_{\mn}},\quad \hat{\kappa}\triangleq \frac{L_{\mx}}{({1}/{m})\sum_i\mu_i},\quad  \breve{\kappa}\triangleq  \frac{L_{\mx}}{\mu},\quad \text{and}\quad  \bar{\kappa}\triangleq \frac{L_{\mx}}{\mu_{\mx}},
\end{equation}
 where\vspace{-0.2cm} \begin{equation}
	\label{eq:prob_param_def}
L_{\mx} \triangleq \max_{i=1,\ldots,m} L_i,   \quad  \mu_{\mn} \triangleq  \min_{i=1,\ldots,m} \mu_i, \quad \text{and} \quad \mu_{\mx} \triangleq \max_{i=1,\ldots,m} \mu_i.\vspace{-0.1cm}
\end{equation}
When $\mu_i=0$, we set $\kappa_\ell=\infty$. 
It is not difficult to check that $\kappa_g$ can be much smaller than $\breve{\kappa}$, $\bar{\kappa}$, $\hat{\kappa}$ and  $\kappa_\ell$, as shown in the following example. 

 \noindent \textbf{Example 1:} Consider the following instance of Problem \eqref{eq:P}: $$f_i(\bx)=\frac{1}{2} \bx^\top \left(\texttt{a} \mathbf{I} + m\cdot  \texttt{b}\, \text{diag}(\mathbf{e}_i)\right)\bx,\quad F(\bx)=\frac{1}{m} \sum_{i=1}^m f_i(\bx)=\frac{\texttt{a}+\texttt{b}}{2}\|\bx\|^2,$$ $G=0$, and $\mathcal{K}=\mathbb{R}^d$, where $\mathbf{e}_i$ is the $i$-th canonical vector, and \texttt{a}, \texttt{b} are some positive constants. We have $\mu_i=\texttt{a}$, $L_i=\texttt{a}+m\cdot \texttt{b}$,  and $\mu=L=\texttt{a}+\texttt{b}$. Therefore,
 $$\frac{\kappa_\ell}{\kappa_g}=\frac{\hat{\kappa}}{\kappa_g}=\frac{\bar{\kappa}}{\kappa_g}=1+m\cdot \frac{\texttt{b}}{\texttt{a}}\quad \text{and} \quad \frac{\breve{\kappa}}{\kappa_g}=\frac{1+m \cdot\texttt{b}/\texttt{a}}{1+ \texttt{b}/\texttt{a}},$$ which all grow indefinitely as   $\texttt{b}/\texttt{a}$ or $m$ increase. 
 $\hfill \square$

In the setting above,  our goal is   to design linearly convergent  distributed algorithms whose  iterations complexity is  proportional to  $\kappa_g$, instead of the larger quantities in (\ref{eq:cond_number_others}).\vspace{-0.2cm} 

\subsubsection{The $\beta$-related setting}\label{sec_related}{  
This   setting considers explicitly the case where the functions $f_i$ are    similar, in the sense defined below 
\cite{Arjevani-ShamirNIPS15}. 
{  
\begin{definition}[$\beta$-related $f_i$'s]\label{assump:homogeneity} The local functions  $f_i$'s (satisfying Assumption \ref{assump:p}) are called $\beta$-related if   
 $\left\|\nabla^2 F(\bx)-\nabla^2f_i(\bx)\right\|_2\leq \beta$, for all $\bx\in\mathcal{K}$ and some $\beta\geq 0$.
\end{definition}
}
The more similar the $f_i$'s, the smaller $\beta$.  
For arbitrary $f_i$'s, $\beta$ is of the order of
\vspace{-0.2cm}\begin{align}
	\beta\leq  \max_{i=1,\ldots, m} \sup_{\bx\in \mathcal{K},\,\|\mathbf{u}\|=1} \!\left| \mathbf{u}^\top  \left(\nabla^2 F(\bx)-\nabla^2f_i(\bx)\mathbf{u}\right) \right|\leq  \max_{i=1,\ldots, m} \!\max \left\{|L-\mu_i|, \,|\mu-L_i|\right\}.
\end{align}
The interesting case is when $1+\beta/\mu << \kappa_g$; a specific example  is discussed next.  

\paragraph*{\bf Example 2: Convex-Lipschitz-bounded learning problems over networks}\label{sec:case-study} 
Consider a stochastic learning setting whereby the ultimate goal is to minimize some population objective   \vspace{-0.1cm}
\begin{equation}\label{eq:population}
\bx^\star \in   \argmin_{\bx\in \mathcal {H} } F(\bx),\quad\text{with}\quad  F(\bx)\triangleq \mathbb{E}_{\mathbf{z}\sim \mathcal{P}} \left [ f(\bx;\mathbf{z})\right],\vspace{-0.1cm}\end{equation}
where $f:\mathcal{O}\times \mathcal{Z}\to \mathbb{R}$ is the   loss function, assumed to be  $C^2$, convex (but not strongly convex),   and  $L$-smooth on the open set  $\mathcal{O}\supset\mathcal{H}$, for all $\mathbf{z}\in \mathcal{Z}$; 
$\mathcal{H}\subseteq \mathbb{R}^d$ is the set of hypothesis classes, assumed to be convex and closed;   $\mathcal{Z}$ is the set of examples; and  $\mathcal{P}$ is the (unknown) distributed of  $\mathbf{z}\in \mathcal{Z}$. Furthermore, we assume that any $\bx^\star\in \mathcal{B}_B\triangleq \{\bx\,:\, \|\bx\|\leq B\}$, for some $0<B<\infty$.  This setting includes, for example,   supervised generalized linear models,    where $\mathbf{z}=(\mathbf{w},y)$ and $f(\bx;(\mathbf{w},y))=\ell(\boldsymbol{\phi}(\mathbf{w})^\top \bx;y)$, for some (strongly) convex loss $\ell(\bullet; y)$ and feature mapping   $\boldsymbol{\phi}$. For instance, in linear regression,  $f(\bx;(\mathbf{w},y))=(y-\boldsymbol{\phi}(\mathbf{w})^\top \bx)^2$, with $\boldsymbol{\phi}(\mathbf{w})\in \mathbb{R}^d$ and $y\in \mathbb{R}$; for logistic regression, we have   $f(\bx;(\mathbf{w},y))=\log(1+\exp(-y (\boldsymbol{\phi}(\mathbf{w})^\top \bx)))$,  with   $\mathbf{w}\in \mathbb{R}^d$ and $y\in \{-1,1\}$.
 
 To solve   (\ref{eq:population}),   the $m$ agents have   access only to a finite number, say $N=nm$,  of i.i.d. samples from the  distribution $\mathcal{P}$, evenly and randomly distributed over the network. Using the notation  introduced in Sec.~\ref{sec:intro}, the ERM problem reads:
\begin{equation}\label{eq:ERM}
\widehat{\bx} \triangleq   \argmin_{\bx\in \mathcal {H} }\widehat{F}(\bx)\triangleq \frac{1}{m} \sum_{i=1}^m f_i(\bx;\mathcal{D}^{(i)}),\qquad f_i(\bx;\mathcal{D}^{(i)})=\frac{1}{n} \sum_{j=1}^n f(\bx;\mathbf{z}_{j}^{(i)})+\frac{\lambda}{2} \|\bx\|^2, \vspace{-0.2cm}\end{equation}
where $f_i$ is regularized   empirical loss of agent $i$, 
 $\lambda$-strongly convex.  Clearly (\ref{eq:ERM}) is an instance of (\ref{eq:P}), satisfying  Assumption \ref{assump:p}.  

For the ERM problems (\ref{eq:ERM}) we derive next the associated $\beta/\mu$ and contrasts with $\kappa_g$.   $\widehat{F}$ is  $\lambda$-strongly convex; therefore, we can set $\mu=\lambda$. The optimal choice of $\lambda$ is the one minimizing the statistical error resulting in  using $\widehat{\bx}$ as proxy for  $ \bx^\star$.
We have   \cite[Th. 7]{Shwartz_et_al},  with high probability,  	$F(\widehat{\bx})-F({\bx}^\star)\leq \frac{\lambda}{2} \|\boldsymbol{\theta}^\star\|^2 + \mathcal{O}\big(\frac{G_f^2}{\lambda\,N}\big)\leq \mathcal{O}\big(\lambda\,B^2 + \frac{G_f^2}{\lambda\,N}\big)$, where 
    $G_f$ is  the Lipschitz constant of $f(\bullet; \mathbf{z})$ on $\mathcal{H}\bigcap \mathcal{B}_B$, for all  $\mathbf{z}\in \mathcal{Z}$. 
The optimal choice of $\lambda$ and resulting minimum error rate are  then\vspace{-0.1cm}
	\begin{equation}
	\label{eq:opt_lambda}
		\lambda=\mathcal{O}\left(\sqrt{\frac{G^2}{B^2\,N}}\right) \quad \Rightarrow \quad F(\widehat{\bx})-F(\bx^\star)\leq\mathcal{O}\left(\sqrt{\frac{G^2 B^2}{N}}\right).\vspace{-0.2cm}
	\end{equation}

An estimate of $\beta$ can be obtained exploring the statistical similarity of the local empirical losses $f_i$ in  (\ref{eq:ERM}). Under the additional assumption that $\nabla^2 f(\bullet;\mathbf{z})$ is $M$-Lipchitz on $\mathcal{H}$, for all $\mathbf{z}\in \mathcal{Z}$, a minor modification of \cite[Lemma 6]{Zhang-Xiao-chapter18} applied to  (\ref{eq:population})-(\ref{eq:ERM}), yields:  with high probability,\vspace{-0.2cm} \begin{equation*}\sup_{\bx\in \mathcal{B}_B} \left\|\nabla^2 f_i(\bx;\mathbf{z})-\nabla^2 \hat{F}(\bx)\right\|\leq \beta,\quad \forall \mathbf{z}\in \mathcal{Z},\,i\in [m],\vspace{-0.2cm}\end{equation*}  
  with \vspace{-0.4cm}
\begin{equation}\label{eq:beta_example}
\hspace{-0.35cm}	\beta=
\left\{ \begin{array}
{ll}
 	\widetilde{\mathcal{O}}\left(\sqrt{\frac{L^2}{n}}\right), & \text{ if } M=0;\\
 	\widetilde{\mathcal{O}}\left(\sqrt{\frac{L^2\,d}{n}}\right),& \text{ otherwise},
 \end{array}
\right. \vspace{-0.2cm}
\end{equation}
where $\widetilde{\mathcal{O}}$ hides the log-factor dependence. Note that when $f(\bullet; \mathbf{z})$ is quadratic   (i.e., $M=0$), $\beta$ scales favorably with the dimension $d$.  

Based on (\ref{eq:opt_lambda})-(\ref{eq:beta_example}), an estimate of $\beta/\mu$ and $\kappa_g$ for (\ref{eq:ERM}) reads: \vspace{-0.2cm}
\begin{equation}\label{eq:beta_vs_kappa_ML_problem}
1+	\frac{\beta}{\mu}=1+\widetilde{\mathcal{O}}\left(L \,\sqrt{d\,m} \right)\quad \text{and}\quad \kappa_g=1+\widetilde{\mathcal{O}}\left(L \,\sqrt{d\,m\,n} \right).\vspace{-0.1cm}
\end{equation} 
Note that $\kappa_g$ increases with the local sample size $n$ while $\beta/\mu$ {\it does not} (neglecting log-factors). It turns out that algorithms converging at a rate depending on $\kappa_g$ exhibit a speed-accuracy dilemma: small statistical errors in (\ref{eq:opt_lambda}) (larger $n$) are achieved at the cost of more iterations (larger $\kappa_g$). In this setting, it is  
thus desirable to design distributed algorithms whose rate depends on $\beta/\mu$ rather than 
$\kappa_g$. }}\vspace{-0.1cm}


\subsection{Network setting}
We will consider separately two network settings: i) the case where the underlying communication graph is fixed and undirected; and ii) the more general setting of time-varying directed graphs.   
\paragraph*{\it  Undirected, static graphs:} When the    network of the agent is modeled as a fixed, undirected graph, we write  $\GG \triangleq (\VV, \EE)$, where $\VV \triangleq \{1,\ldots, m\}$ denotes  the vertex    set--the set of agents--while $\EE \triangleq \{(i,j) \, |\, i,j \in \VV \}$ represents the set of edges--the communication links;  $(i,j) \in \EE$ iff there exists a communication link between agent $i$ and $j$. 
We make the following standard assumption on the graph connectivity.  
\begin{assumption}[On the network]\label{assump:network}
The graph $\GG$ is connected.
\end{assumption}

\paragraph*{\it  Directed, time-varying graphs} In this setting,  communication network  is modeled as a    time-varying digraph: time is slotted, and at time-frame $\nu$, the digraph reads $\mathcal{G}^{\nu}=\left(\mathcal{V},\mathcal{E}^{\nu}\right)$, where   the set of edges $\mathcal{E}^{\nu}$  represents the agents' communication links: $(i,j)\in \mathcal{E}^{\nu}$ there is a link going from agent $i$ to agent $j$.  
We make the following standard assumption on the  ``long-term'' connectivity property of the graphs.\smallskip 

\hypertarget{assumption:G}{\emph{\normalfont\sc Assumption} B$\,^\prime$ (On the network).}  
	{\it The graph sequence $\{\mathcal{G}^{\nu}\}$, $\nu=0,1,\ldots $, is $B$-strongly connected, i.e., there exists a finite integer $B > 0$   	such that the graph with edge set $\cup_{t=\nu B}^{(\nu+1)B-1} \mathcal{E}^{t}$ is strongly connected,  for all $\nu=0,1,\ldots $}.\smallskip 

  {  The network setting covers, as special case,   star-networks,
 i.e., architectures with a centralized node (a.k.a. master node) connected to all the others (a.k.a. workers). This is the  typical computational architecture of several   federated learning systems.  
 }


\section{The SONATA algorithm over   undirected graphs}\label{sec:SONATA-NEXT}
We recall here  the SONATA/NEXT algorithm \cite{LorenzoScutari-J'16, SONATA-companion},   customized to undirected, static, graphs.   
Each agent $i$ maintains and updates iteratively a    local copy  $\mathbf{x}_{i}\in \mathbb{R}^d$ of the global variable $\mathbf{x}$, along with the auxiliary variable $\mathbf{y}_{i}\in \mathbb{R}^d$, which   estimates  the gradient of $F$.   Denoting by $\mathbf{x}_{i}^\nu$ (resp. $\mathbf{y}_{i}^\nu$) the values of $\mathbf{x}_{i}$ (resp. $\mathbf{y}_{i}$) at iteration $\nu=0,1,\ldots,$ the SONATA algorithms is described in Algorithm \ref{alg:SONATA}. 
	\begin{algorithm}[h]	
	\caption{SONATA over undirected graphs}\label{alg:SONATA}
	\textbf{Data}: $\mathbf{x}^{0}_{i}\in \mathcal{K}$ and   $\by_i^0=\nabla f_i(\bx_i^0)$,   $i\in [m]$. 
	
	\textbf{Iterate}: $\nu=1,2,...$\vspace{0.1cm}
\begin{subequations}

  \texttt{[S.1] [Distributed Local Optimization]} Each agent $i$ solves\vspace{-0.2cm}
	\begin{equation} \hbx_i^\nu \triangleq {}   \argmin_{\bx_i \in \KK}~\underbrace{\tf_i (\bx_i ;\bx_i^\nu) + \big( \by_i^\nu - \nabla f_i (\bx_i^\nu) \big)^\top (\bx_i - \bx_i^\nu)}_{\tF_i (\bx_i; \bx_i^\nu)} + G(\bx_i), \label{eq:loc_opt}\vspace{-0.3cm}\end{equation} 
	
	\qquad and updates\vspace{-0.3cm} \begin{equation}\bx_i^{\nu+\frac{1}{2}} = {}   \bx_i^\nu + \alpha \cdot \Deltaxi{i}^\nu,\quad \text{with}\quad\Deltaxi{i}^\nu \triangleq \hbx_i^\nu - \bx_i^\nu; \label{eq:descent}\end{equation}
	

	 \texttt{[S.2] [Information Mixing]} Each agent $i$ computes \smallskip\\
	 \phantom{\texttt{[S.2]}} (a) \texttt{Consensus}  \vspace{-0.3cm}
	\begin{equation}\bx_i^{\nu + 1} =   \sum_{j=1}^m w_{ij} \bx_j^{\nu+\frac{1}{2}}, \label{eq:mix_x}\vspace{-0.3cm}\end{equation}  \phantom{\texttt{[S.2]}} (b) \texttt{Gradient tracking}\vspace{-0.2cm}
	\begin{equation}
	\by_i^{\nu + 1} =   \sum_{j=1}^m w_{ij} \big(\by_j^\nu + \nabla f_j(\bx_j^{\nu +1}) - \nabla f_j(\bx_j^\nu)\big).\label{eq:mix_y}\vspace{-0.4cm}
	 \end{equation}
	\textbf{end} \end{subequations}
\end{algorithm}
 In words, each agent $i$, given the current   iterates $\bx_i^\nu$ and $\by_i^\nu$,  first solves a strongly convex optimization problem wherein  $\tF_i$ is an approximation of the sum-cost $F$ at $\bx_i^\nu$; $\tf_i$ in (\ref{eq:loc_opt}) is a strongly convex function, which plays the role of a surrogate of   $f_i$ (cf. Assumption \ref{assump:SCA_surrogate} below) while $\mathbf{y}_i^\nu$ acts as approximation of the gradient of $F$ at $\bx_i^\nu$, that is, $\nabla F (\bx_i^\nu) \approx \by_i^\nu$ (see discussion below). Then, agent $i$      updates    $\bx_i^\nu$  along the local direction $\Deltaxi{i}^\nu$ [cf.~(\ref{eq:descent})], using the step-size $\alpha\in (0,1]$; the resulting point $\bx_i^{\nu+1/2}$ is broadcast to its neighbors.  The update $\bx_i^{\nu+1/2}\to \bx_i^{\nu+1}$ is  obtained via the consensus step (\ref{eq:mix_x}) while  the $y$-variables are updated via the  perturbed consensus   (\ref{eq:mix_y}), aiming at tracking  $\nabla F(\bx_i^\nu)$. 

The main assumptions underlying the convergence of SONATA are discussed next. 

\paragraph*{$\bullet$ \it On the subproblem (\ref{eq:loc_opt}) and surrogate functions $\tf_i$} The surrogate functions satisfy the following conditions. 
\vspace{-0.2cm}
\begin{assumption}\label{assump:SCA_surrogate}
Each  $\tf_i:\mathcal{O}\times \mathcal{O} \to \real$ is $C^2$ and satisfies
\begin{enumerate}
\item[(i)] $\nabla \tf_i (\bx ; \bx) = \nabla f_i (\bx)$, for all $\bx \in \KK$;
\item[(ii)] $\nabla \tf_i (\bullet;\bx)$ is $\tL_i$-Lipschitz continuous on $\KK$, for all $\bx\in \KK$;
\item[(iii)] $\tf_i(\bullet ; \bx)$ is $\tmu_i$-strongly convex on $\KK$, for all $\bx\in \KK$;\end{enumerate}
where $\nabla \tf_i (\bx ; \mathbf{z}) $ is the partial gradient of $\tf_i$ at $(\bx,\bz)$ with respect to the first argument.\vspace{-0.1cm}
\end{assumption}
The assumption states that $\tf_i$ should be regarded as a surrogate of $f_i$ that preserves at each iterate $\bx^\nu_i$  the first order properties of $f_i$. {  Conditions (i)-(iii)  are certainly satisfied if one uses the  classical linearization of $f_i$, that is, \vspace{-0.1cm}\begin{equation}\label{eq:sca-linearization}\tf_i(\bx_i;\bx_i^\nu)=\nabla f_i(\bx^{\nu}_i)^\top (\bx_i-\bx_i^\nu)+\frac{\tau_i}{2}\|\bx_i-\bx_i^\nu\|^2,\vspace{-0.1cm}\end{equation}
with $\tau_i>0$, which leads to the standard proximal-gradient update for $\widehat{\bx}_i$. Note that if, in addition,  $G=0$ and  $\KK=\mathbb{R}^d$,    (\ref{eq:loc_opt})--(\ref{eq:mix_x}) reduces to the standard  (ATC) consensus/gradient-tracking step (setting $\alpha=1$ and absorbing $1/\tau_i$ into the common stepsize $\gamma$): $\bx_i^{\nu+1}=\sum_j w_{ij} (\bx_i^\nu-\gamma \,\by_i^\nu)$ \cite{qu2016harnessing,nedich2016achieving,Xu2015augmented}. However,   Assumption~\ref{assump:SCA_surrogate}  allows us to cover a much wider array of approximations that better suit the geometry of the problem at hand, enhancing    convergence speed.
 For instance, on the opposite side of (\ref{eq:sca-linearization}), we have  a surrogate retaining all the structure of $f_i$, such as\vspace{-.1cm}
\begin{equation}\label{eq_sca_f_i}\tf_i(\bx_i;\bx_i^\nu)= f_i(\bx_i)+\frac{\tau_i}{2}\|\bx_i-\bx_i^\nu\|^2,\vspace{-.1cm}\end{equation}
with $\tau_i>0$. Using (\ref{eq_sca_f_i}), one can rewrite   (\ref {eq:loc_opt}) as: \vspace{-.1cm}
\begin{equation}\hbx_{i}^{\nu}=\argmin_{\bx_{i}\in\KK}\Big({1}\cdot\underset{\nabla g(\mathbf{x}_{i}^{\nu})}{\underbrace{\left(\nabla f_{i}(\mathbf{x}_{i}^{\nu})+\mathbf{y}_{i}^{\nu}\right)}}-\underset{\nabla\omega(\mathbf{x}_{i}^{\nu})}{\underbrace{\left(\nabla f_{i}(\mathbf{x}_{i}^{\nu})+\tau_{i}\mathbf{x}_{i}^{\nu}\right)}}\Big)^{\top}\!\!\!\mathbf{x}_{i}+\underset{\omega(\mathbf{x}_{i})}{\underbrace{\left(f_{i}(\mathbf{x}_{i})+\frac{{\tau_{i}}}{2}\left\Vert \mathbf{x}_{i}\right\Vert ^{2}\right)}}+G(\bx_{i}),\vspace{-.3cm}\end{equation}
which can be interpreted as  a mirror-descent update (with step-size one) for the composite minimization of    $g(\bx_i)\triangleq f_i(\bx_i)+(\by_i^{\nu})^\top (\bx_i-\bx_i^\nu)$, based on the Bregman distance associated with the reference function $\omega(\bx_i)\triangleq  f_{i}(\mathbf{x}_{i})+ {{\tau_{i}}}/{2} \Vert \mathbf{x}_{i} \Vert^2 $.

  We refer the reader to \cite{facchinei2015parallel,scutari_PartI,Scutari_Ying_LectureNote} as good sources of examples of nonlinear surrogates  satisfying Assumption~\ref{assump:SCA_surrogate};  here we only anticipate that, when the  $f_i$'s are sufficiently similar, higher order models such as (\ref{eq_sca_f_i}) yield indeed faster rates of SONATA than those achievable using linear surrogates (\ref{eq:sca-linearization}). 
Further intuition is provided next. 
 
Under Assumption~\ref{assump:SCA_surrogate}, it is not difficult to check that, for every $i\in [m]$, there exist constants $D_i^\ell$ and $D_i^u$, $D_i^\ell\leq  D_i^u$, such that \begin{equation}\label{eq:upper-lower-hessian}
 	 D_i^\ell \,\mathbf{I} \preceq \nabla^2 \tf_i (\bx ,\by) - \nabla^2 F(\bx) \preceq D_i^u \,\mathbf{I},\quad  \forall \bx, \by \in \KK;\qquad \text{let }\,\,  D_i \triangleq \max\{|D_i^\ell|, |D_i^u|\}.
 \end{equation}
For instance, \eqref{eq:upper-lower-hessian} holds with  $D_i = \max\{|\tmu_i - L|,|\tL_i - \mu|\}$.  Roughly speaking, the smaller $D_i$ the better  $\tF_i$ in  (\ref{eq:loc_opt}) approximates 
$F$.  To see this, compare   $F$ and $\tF_i$ up to the second order: there exist $\theta_1, \theta_2\in (0,1)$ such that
 \begin{equation} \label{eq_F_vs_f_tilde} \begin{aligned} 
 \tF_i (\bx_i; \bx_i^\nu) = \ &  \tf_i (\bx_i^\nu ;\bx_i^\nu) + \big( \by_i^\nu - \nabla f_i (\bx_i^\nu) + \nabla\tf_i (\bx_i^\nu ;\bx_i^\nu)\big)^\top (\bx_i - \bx_i^\nu) \\
 & + \frac{1}{2} (\bx_i - \bx_i^\nu)^\top \nabla^2 \tf_i\big(\bx_i^\nu + \theta_1 (\bx_i - \bx_i^\nu); \bx_i^\nu \big) (\bx_i - \bx_i^\nu)\\
 F(\bx_i) = \ & F(\bx_i^\nu) + \nabla F(\bx_i^\nu)^\top (\bx_i - \bx_i^\nu) \\
 & +  \frac{1}{2} (\bx_i - \bx_i^\nu)^\top \nabla^2 F \big(\bx_i^\nu + \theta_2 (\bx_i - \bx_i^\nu); \bx_i^\nu \big) (\bx_i - \bx_i^\nu).
 \end{aligned}\end{equation}
 Noting  that  $\nabla \tf_i (\bx_i^\nu; \bx_i^\nu) = \nabla f_i(\bx_i^\nu)$ [Assumption~\ref{assump:SCA_surrogate}(i)] and  $\nabla \tF_i (\bx_i^\nu; \bx_i^\nu)=\by_i^\nu$, and  anticipating $\|\nabla F(\bx_i^\nu)- \by_i^\nu\|\to 0$ as $\nu\to \infty$ (see discussion below),    it follows that  $\tF_i$ approximates  $F$  asymptotically, up to the first order. A better match,   is achieved when  $D_i$ is sufficiently small.  
One can then expect that, if  the local functions are sufficiently similar ($\beta$ is small),  surrogates  $\tf_i$ exploiting higher order information of $f_i$, such as (\ref{eq_sca_f_i}), may be more effective than  mere linearization.  Our theoretical findings 
confirm the above intuition--see Sec. \ref{sec:discussion}.    }
    
   \paragraph*{$\bullet$ \it Consensus and gradient tracking steps (\ref{eq:mix_x})-(\ref{eq:mix_y})}  In the consensus and tracking steps, the weights $w_{ij}$'s 
   satisfy  
  the following standard assumption. 
\begin{assumption}\label{assump:weight}
The weight matrix $\bW \triangleq (w_{ij})_{i,j = 1}^m$ has a sparsity pattern compliant with $\GG$, that is 
\begin{enumerate}[leftmargin=*,label=\theassumption\arabic*]
\item $w_{ii} > 0$, for all $i = 1, \ldots, m$;
\item $w_{ij} > 0$, if $(i,j) \in \EE$; and $w_{ij}=0$ otherwise;
\end{enumerate}
Furthermore, $\bW$ is doubly stochastic, that is, $\mathbf{1}^\top \bW = \mathbf{1}^\top$ and $\bW \mathbf{1} = \mathbf{1}$.
\end{assumption}

 Several rules have been proposed in the literature compliant with Assumption \ref{assump:weight}, such as    the Laplacian,  the Metropolis-Hasting, and the maximum-degree weights rules    \cite{xiao2005scheme}. 

Finally, we comment the anticipated gradient tracking property of the $y$-variables, that is, $\|\nabla F(\bx_i^\nu)- \by_i^\nu\|\to 0$ as $\nu\to \infty$. 
Define the average processes \vspace{-0.1cm}
\begin{equation}\label{eq:avg_y_def}
\bar{\by}^\nu \triangleq \frac{1}{m} \sum_{i=1}^m  \by_i^\nu \quad \text{and} \quad \overline{\pgrad}^\nu\triangleq \frac{1}{m} \sum_{i=1}^m \nabla f_i (\bx_i^\nu).\vspace{-0.2cm}
\end{equation}
Summing \eqref{eq:mix_y} over $i\in [m]$ and invoking the doubly stochasticity of $\bW$; we have \vspace{-0.2cm}
\begin{equation}\label{eq:y_update_avg}
\bar{\by}^{\nu + 1} =   \bar{\by}^\nu + \overline{\pgrad}^{\nu +1} - \overline{\pgrad}^\nu.\vspace{-0.2cm}
\end{equation}
Applying \eqref{eq:y_update_avg} inductively   and using the initial condition $\by^0_i = \nabla f_i(\bx_i^0)$, $i\in [m],$   yield \vspace{-0.3cm}
\begin{equation}\label{eq:average-consistency}
\bar{\by}^\nu = \overline{\pgrad}^\nu, \quad \forall \nu = 0,1,\ldots .
\end{equation}
That is, the average of all the $\by_i^\nu$'s  in the network is equal to that of the $\nabla f_i (\bx_i^\nu)$'s, at every iteration $\nu$. 
  Assuming that consensus on $\bx_i^\nu$'s and $\by_i^\nu$'s is asymptotically achieved, that is,  $\|\bx_i^\nu-\bx_j^\nu\|\underset{\nu\to\infty}{\longrightarrow} 0$ and $\|\by_i^\nu-\by_j^\nu\|\underset{\nu\to\infty}{\longrightarrow} 0$,    $i\neq j$,  \eqref{eq:average-consistency} would imply  the desired gradient tracking property $\|\nabla F(\bx_i^\nu)- \by_i^\nu\|\to 0$ as $\nu\to \infty$, for all $i\in [m]$.




\subsection{A special instance: SONATA on star-networks} { 

Although the main focus of the paper is the study of SONATA over meshed-networks, it is worth discussing here its special instance  over  star networks. 
Specifically, consider a star (unidirected)  graph with $m$ nodes, where one of them (the master node) connects with all the others (workers). The workers  still own only one function $f_i$ of the sum-cost $F$. Two common approaches developed in the literature to solve (\ref{eq:P}) in this setting are: (i) based upon receiving the gradients $\nabla f_i$ from the workers,  the  master solves (\ref{eq:P})  and broadcasts the updated vector variables   to the workers; (ii) based upon receiving the full gradient $\nabla F$ and the current iterate from the master, all the workers solve locally an instance of  (\ref{eq:P}) and send their outcomes to the master that averages them out, producing then the new iterate. Here we follow the latter approach; 
the algorithm is  described in    Algorithm \ref{alg:SONATA-star}, which corresponds to SONATA     (up to a proper initialization), with     weight matrix $\bW=\left[\mathbf{1},\, \mathbf{0}_{m,m-1}\right] \left[ \mathbf{1}/m, \, \mathbf{0}_{m,m-1}\right]^\top $. 
   \begin{algorithm}[t]	
	\caption{SONATA on Star-Networks (SONATA-Star)}\label{alg:SONATA-star}{ 
	\textbf{Data}: $\mathbf{x}^{0}\in \mathcal{K}$. 
	\vspace{0.1cm}
	
	\textbf{Iterate}: $\nu=1,2,...$\vspace{0.1cm}
	
	\quad \texttt{[S.1]} Each worker $i$ evaluates $\nabla f_i(\bx^\nu)$ and sends it to the master node;\vspace{0.1cm}
	
	\quad\texttt{[S.2]} The master   broadcasts  $\nabla F(\bx^\nu)=1/m \sum_{i=1}^m \nabla f_i(\bx^\nu)$ to the workers;\vspace{0.1cm}

	\quad \texttt{[S.3]} Each worker $i$ computes  \vspace{-0.2cm}
$$   \hbx_i^\nu \triangleq {}   \argmin_{\bx_i \in \KK} {\tf_i (\bx_i ;\bx^\nu) + \big( \nabla F(\bx^\nu) - \nabla f_i (\bx_i^\nu) \big)^\top (\bx_i - \bx^\nu)} + G(\bx_i),$$

\qquad \qquad and sends  $ \hbx_i^\nu$ to the master;\vspace{0.1cm}

	\quad \texttt{[S.4] } The master computes\vspace{-0.3cm} \begin{equation*}\bx^{\nu+1}=\bx^\nu +\alpha \left(\frac{1}{m}\sum_{i=1}^m   \hbx_i^\nu - \bx^\nu\right),\end{equation*} 
 \qquad \qquad and sends it back to the workers.\vspace{0.1cm}

  \textbf{end}}
 \end{algorithm}

 \paragraph*{Connection with existing schemes} SONATA-star, employing   linear surrogates [cf.~(\ref{eq:sca-linearization})] and $\alpha=1$, reduces to the proximal gradient algorithm. 
 When the surrogates (\ref{eq_sca_f_i}) are used (and still   $\alpha=1$),   SONATA-star coincides with  the DANE algorithm \cite{DANE} if $G=0$  and to the CEASE (with averaging) algorithm \cite{Fan2020} if $G\neq 0$.  Nevertheless,  our convergence rates improve on those of DANE and CEASE--see Sec.~\ref{sec:star-topology}. }
 
\subsection{Intermediate definitions} We conclude  this section introducing  some   quantities  that will be used in the rest of the paper.  
We define the   optimality gap as \vspace{-0.2cm}\begin{equation}\label{eq:opt_gap}
 	p^\nu \triangleq \sum_{i=1}^m \big(U(\bx_i^\nu) - U(\bx^\star) \big),\vspace{-0.2cm}\end{equation}
where $\bx^\star$  is  the unique solution of Problem (\ref{eq:P}). 
 
We stack the local variables and gradients in the column vectors 
\begin{equation}\label{eq:stacked_vec}
\!\bx^\nu   \triangleq [\bx_1^{\nu\top},\ldots,\, \bx_m^{\nu\top}]^\top,\,\, \by^\nu   \triangleq [\by_1^{\nu\top},\ldots, \by_m^{\nu\top}]^\top,\, \pgrad^\nu   \triangleq [\nabla f_1 (\bx_1^\nu)^\top,\ldots, \nabla f_m (\bx_m^\nu)^\top ]^\top.\end{equation}
 The average of each of the vectors above is defined as     $\bar{\bx}^\nu \triangleq (1/m) \cdot \sum_{i=1}^m \bx_i^\nu$.  
  The   consensus disagreements  on $\bx_i^\nu$'s and $\by^\nu_i$'s    are  \vspace{-0.1cm}
\begin{equation}\label{eq:err_x_def}
\bx_\bot^\nu \triangleq \bx^\nu - \mathbf{1}_m \otimes \bar{\bx}^\nu  \quad \text{and} \quad  \by_\bot^\nu \triangleq \by^\nu - \mathbf{1}_m \otimes \bar{\by}^\nu,\vspace{-0.2cm}
\end{equation}
respectively, while   the gradient   tracking error is   defined as  \vspace{-0.1cm}  { \begin{equation}\label{eq:tracking_err_def}
\bdelta^\nu   \triangleq [\bdelta_1^{\nu\top},\ldots, \bdelta_{m}^{\nu\top}]^\top, \quad \text{with}\quad \bdelta_i^\nu \triangleq \nabla F(\bx_i^\nu) -  \by_i^\nu,\quad i=1,\ldots, m.\vspace{-0.4cm}
\end{equation}}

Recalling  $L_i$, $\tL_i$,  $\tmu_i$, $D_i^\ell$ and $D_i$ as given in Assumptions~\ref{assump:p} and~\ref{assump:SCA_surrogate} and \eqref{eq:upper-lower-hessian}, we introduce the following    algorithm-dependent parameters\vspace{-0.1cm}
\begin{equation}\label{eq:alg_param_def}
\begin{aligned}
& \tmu_{\mn} \triangleq \min_{i\in [m]} \tmu_i , && \tL_{\mx}\triangleq \max_{i\in [m]}, \tL_i ,  \\
& {   D_{\mn}^\ell \triangleq \min_{i\in [m]} D_i^\ell }, && {  D_{\mx} \triangleq  {\max_{i\in [m]}} D_i. }
\end{aligned}
\end{equation}

Finally, given the weight matrix $\bW$, we define 
\begin{equation}\label{eq:def_W_hat}
  	\hbW \triangleq \bW \otimes \bI_d, \quad\text{and}\quad   \mathbf{J} \triangleq   \frac{1}{m} \mathbf{1}_m \mathbf{1}_m^\top \otimes \bI_d.  
  \end{equation}
Under Assumptions \ref{assump:network} and \ref{assump:weight}, it is well known that    {(see, e.g., \cite{Tsitsiklis_Thesis})}\vspace{-0.2cm}
\begin{equation}\label{eq:rho_def}
\rho  \triangleq  \sigma(\hbW - \bJ) < 1,\vspace{-0.2cm}
\end{equation}
where $\sigma (\bullet)$ denotes the largest singular value of its argument.

\subsection{Linear convergence rate}\label{sec:linear_rate}
 Our proof of linear rate of SONATA passes through the following steps. \textbf{Step 1:} We begin showing that the  optimality gap $p^\nu$    converges linearly up to an error of the order of $\mathcal{O} (\| \bx_\bot^\nu\|^2 + \| \by_\bot^\nu\|^2)$, see~ Proposition~\ref{prop:loc_geo_rate}.  \textbf{Step 2} proves that $\|\bx_\bot^\nu\|$ and $\|\by_\bot^\nu\|$ are also linearly convergent up to an error    $\mathcal{O}(\| \Deltaxi{}^\nu\|)$, see~ Proposition~\ref{prop:err_x}. In \textbf{Step 3} we close the loop establishing $\|\Deltaxi{}^\nu \|= \mathcal{O}(\sqrt{p^\nu}+\| \by_\bot^\nu \|)$, see~Proposition~\ref{prop:bound_delta_x}. Finally, in {\bf Step 4}, we properly chain together the above inequalities (cf. Proposition~\ref{prop:small_gain_eqs_transformed}), so that linear rate is proved for the sequences 
  $\{p^\nu\}$, $\{\| \bx_\bot^\nu\|^2 \}$, $\{\| \by_\bot^\nu\|^2 \}$, and $\{\| \mathbf{d}^\nu\|^2 \}$--see Theorems  \ref{thm:step_size_condition} and \ref{thm:linear_rate}. We will tacitly assume that    Assumptions~\ref{assump:p}, \ref{assump:network}, \ref{assump:SCA_surrogate},  and \ref{assump:weight} are satisfied.


\subsubsection{\!Step 1: \!$p^\nu$ converges linearly up to $\mathcal{O} (\| \bx_\bot^\nu\|^2 \!\!+ \!\| \by_\bot^\nu\|^2)$}\label{sec:inexact_sca}\!

Invoking  the convexity of $U$ and the doubly stochasticity of $\bW$, we can bound   $p^{\nu+1}$ as\vspace{-0.1cm}
 \begin{equation}\label{eq:bound_doubly_stoc_cvx}
p^{\nu+1} \leq \sum_{i=1}^m\sum_{j=1}^m w_{ij} \Big(U\big(\bx_j^{\nu+\frac{1}{2}}\big)  -U(\bx^\star)\Big) 
=  \sum_{i=1}^m \Big(U(\bx_i^{\nu+\frac{1}{2}})-U(\bx^\star)\Big).\vspace{-0.2cm}
\end{equation}
 
We can now bound $U(\bx_j^{\nu+\frac{1}{2}})$, regarding the local optimization~\eqref{eq:loc_opt}-\eqref{eq:descent}   as a perturbed descent on the objective, whose perturbation is due to the tracking error $\bdelta^\nu$. 
In fact, Lemma \ref{lem:loc_dec_obj} below shows that,  for   sufficiently small $\alpha$, the  local update \eqref{eq:descent}     will  decrease the   objective value $U$ up to some error, related to    $\bdelta_i^\nu$. 

{ 
\begin{lemma}\label{lem:loc_dec_obj} 
Let $\{\bx_i^\nu\}$ be the sequence generated by  
  SONATA; there holds:
\vspace{-0.1cm}
\begin{align}\label{eq:loc_dec_obj}
U(\bx_i^{\nu+\frac{1}{2}} ) \leq U(\bx_i^\nu) - \alpha \left(\left(1 - \frac{\alpha}{2}\right)\tmu_i +  \frac{\alpha}{2} \cdot D_i^\ell \right)\| \Deltaxi{i}^\nu\|^2   + \alpha\|\Deltaxi{i}^\nu\| \|\bdelta_i^\nu \|,\vspace{-0.1cm}
\end{align}
with    $D_i^{\ell}$ and $\bdelta_i^\nu$ are defined in~\eqref{eq:upper-lower-hessian} and~\eqref{eq:tracking_err_def}, respectively.
\end{lemma}

\begin{proof}
Consider the Taylor expansion of $F$:  
\begin{align}\label{eq:taylor_F}
\begin{split}
F(\bx_i^{\nu +{\frac{1}{2}}})= \,& F(\bx_i^\nu) + \nabla F(\bx_i^\nu)^\top (\alpha \Deltaxi{i}^\nu)  + (\alpha\Deltaxi{i}^\nu)^\top \mathbf{H}(\alpha\Deltaxi{i}^\nu),\\
 \stackrel{\,\,\,\,\eqref{eq:tracking_err_def}}{\,\,\,\,=}& F(\bx_i^\nu) + \big(\bdelta_i^\nu\big)^\top(\alpha \Deltaxi{i}^\nu) + \big(\mathbf{y}_i^\nu\big)^\top(\alpha \Deltaxi{i}^\nu) + (\alpha\Deltaxi{i}^\nu)^\top \mathbf{H}(\alpha\Deltaxi{i}^\nu),
\end{split}
\end{align}

where $\mathbf{H} \triangleq  \int_{0}^1 (1 - \theta) \nabla^2 F (\theta \bx_i^{\nu + \frac{1}{2}} + (1 - \theta) \bx_i^\nu) d \theta$.

Invoking the  optimality of $\hbx_i^\nu$ and defining $\tH_i \triangleq  \int_0^1 \nabla^2 \tf_i (\theta \,\hbx_i^\nu + (1 - \theta) \,\bx_i^\nu; \bx_i^\nu) d \theta$, we have 
\begin{align}\label{eq:opt_x_hat}
\begin{split}
 G(\bx_i^\nu)-G(\widehat{\bx}_i^\nu) & \geq (\Deltaxi{i}^\nu)^\top \big(\nabla \tf_i(\hbx_i^\nu;\bx_i^\nu)  +  \by_i^\nu - \nabla f_i(\bx_i^\nu)\big) 
  = (\Deltaxi{i}^\nu)^\top \big( \by_i^\nu + \tH_i \Deltaxi{i}^\nu\big), 
\end{split}
\end{align}
 where the equality follows from  $\nabla \tf_i(\bx_i^\nu;\bx_i^\nu)=\nabla f_i(\bx_i^\nu)$ and the integral form of the mean value theorem. 
Substituting (\ref{eq:opt_x_hat}) in (\ref{eq:taylor_F}) and  using the convexity of $G$ yield\vspace{-0.2cm}
\begin{align}\label{eq:descent_lem}
\begin{split}
 &F(\bx_i^{\nu +{\frac{1}{2}}})\\
\leq &\,F(\bx_i^\nu) +(   \bdelta_i^\nu)^\top (\alpha \Deltaxi{i}^\nu)  + (\alpha\Deltaxi{i}^\nu)^\top \mathbf{H}(\alpha\Deltaxi{i}^\nu)
+ \alpha \left(G(\bx_i^\nu)-G(\widehat{\bx}_i^\nu)  - (\Deltaxi{i}^\nu)^\top \tH_i \Deltaxi{i}^\nu\right)\\
\leq & \, F(\bx_i^\nu) +(   \bdelta_i^\nu)^\top (\alpha \Deltaxi{i}^\nu)  +\alpha \left(  - (\Deltaxi{i}^\nu)^\top \tH_i \Deltaxi{i}^\nu+  (\alpha\Deltaxi{i}^\nu)^\top \mathbf{H}(\Deltaxi{i}^\nu)  \right)
+ G (\bx_i^\nu) - G(\bx_i^{\nu + \frac{1}{2}}).
\end{split}
\end{align}
It remains to bound $ \alpha \bH- \tH_i$. We proceed as follows:\vspace{-0.2cm}
\begin{align}\label{eq:Delta_H_bound}
\begin{split}
& \, \alpha \bH- \tH_i \\
= &\, \alpha \int_{0}^1 (1 - \theta) \nabla^2 F (\theta \bx_i^{\nu + \frac{1}{2}} + (1 - \theta) \bx_i^\nu) d \theta -  \int_0^1 \nabla^2 \tf_i (\theta \hbx_i^\nu  + (1 - \theta) \bx_i^\nu; \bx_i^\nu) d \theta\\
\stackrel{\eqref{eq:descent}}{=} &\, \int_0^\alpha (1 - \theta/\alpha) \nabla^2 F (\theta \hbx_i^\nu + (1 - \theta) \bx_i^\nu) d \theta - \int_0^1 \nabla^2 \tf_i (\theta \hbx_i^\nu + (1 - \theta) \bx_i^\nu; \bx_i^\nu) d \theta\\
\stackrel{(a)}{\preceq}  & 
\,-\int_0^\alpha (1 - \theta/\alpha) \cdot ( D_i^\ell )\,\mathbf{I}\, d\theta - \int_0^\alpha (\theta/\alpha) \nabla^2 \tf_i (\theta \hbx_i + (1 - \theta) \bx_i^\nu;\mathbf{x}_i^\nu) d \theta\\
& \,  - \int_\alpha^1 \nabla^2 \tf_i (\theta \,\hbx_i^\nu + (1 - \theta) \,\bx_i^\nu;\mathbf{x}_i^\nu) d \theta\\
\stackrel{(b)}{\preceq} & \, -\frac{1}{2} \alpha \, ( D_i^\ell)\,\mathbf{I}  - \left( 1 - \frac{\alpha}{2}\right)  \, \tmu_i \,\mathbf{I},
\end{split}
\end{align}
where in (a) we used $\nabla^2 F(\theta \hbx_i^\nu + (1 - \theta) \bx_i^\nu) \preceq - (D_i^\ell) \mathbf{I} + \nabla^2\tf_i (\theta \hbx_i^\nu + (1 - \theta) \bx_i^\nu; \bx_i^\nu)$ [cf. \eqref{eq:upper-lower-hessian}] while (b) follows from  Assumption~\ref{assump:SCA_surrogate}(iii). Substituting \eqref{eq:Delta_H_bound} into  \eqref{eq:descent_lem} completes the proof
\end{proof}
}

We can now substitute   (\ref{eq:loc_dec_obj})   into \eqref{eq:bound_doubly_stoc_cvx} and get     
{ 
\vspace{-0.2cm}
	\begin{subequations}\label{eq:opt_err_1}
\begin{align} 
	p^{\nu + 1}  &\leq p^\nu + \sum_{i=1}^m \left\{ \alpha \|\Deltaxi{i}^\nu\|  \| \bdelta_i^\nu \| - \alpha \left(1 - \frac{\alpha}{2}\right) \tmu_i \| \Deltaxi{i}^\nu\|^2 - \frac{D_i^\ell}{2}\alpha^2  \| \Deltaxi{i}^\nu\|^2 \right\} \label{eq:opt_err_0}\\
	& \overset{(a)}{\leq} p^\nu -  \left( \left(1 - \frac{\alpha}{2}\right) \tmu_{\mn} +  \frac{\alpha D_{\mn}^\ell}{2} - \frac{1}{2} \epsilon_{opt}\right) \alpha \| \Deltaxi{}^\nu\|^2 + \frac{1}{2} \epsilon_{opt}^{-1} \,\alpha \cdot   \| \bdelta^\nu \|^2,\label{eq:opt_err_2}
\end{align}\end{subequations}
where in (a) we used  Young's inequality, with   $\epsilon_{opt}>0$ satisfying 
\begin{equation}\label{eq:eps_opt}
\left(1 - \frac{\alpha}{2}\right) \tmu_{\mn} +  \frac{\alpha D_{\mn}^\ell}{2} - \frac{1}{2} \epsilon_{opt} >0;
\end{equation}}
and $D_{\mn}^\ell$ is defined in \eqref{eq:alg_param_def}.

Next we lower bound  $\|\Deltaxi{}^\nu\|^2$ in terms of  the   optimality gap. 

\begin{lemma}\label{lem:delta_x_lb}
The following lower bound holds for $\|\Deltaxi{}^\nu\|^2$:
{ \begin{equation}\label{eq:delta_x_lb}
\alpha \,\| \Deltaxi{}^\nu \|^2 \geq \frac{\mu}{D_{\mx}^2} \left(p^{\nu+1}- (1-\alpha) p^{\nu}-\frac{\alpha}{\mu}  \| \bdelta^\nu \|^2\right), \vspace{-0.1cm}
\end{equation}where $D_{\mx}$ is defined in \eqref{eq:alg_param_def}.}
 

\end{lemma}
\begin{proof}
Invoking the optimality condition of $\hbx_i^\nu$, yields 
{ 
	\begin{align}\label{eq:FOC}
   G(\bx^\star)-G(\hbx_i^\nu)\geq -(\bx^\star - \hbx_i^\nu)^\top \Big(\nabla \tf_i ( \hbx_i^\nu ; \bx_i^\nu)  +  \by_i^\nu - \nabla f_i (\bx_i^\nu) \Big).
\end{align}}
Using the $\mu$-strong convexity of $F$, we can write
{ \begin{align*}
\hspace{-0.3cm}\begin{split}
& U(\bx^\star)   \geq  U(\hbx_i^\nu) +    G(\bx^\star)-G(\hbx_i^\nu) +\nabla F( \hbx_i^\nu )^\top (\bx^\star - \hbx_i^\nu) + \frac{\mu}{2} \| \bx^\star -  \hbx_i^\nu\|^2   \\
& \overset{(\ref{eq:FOC})}{\geq}  \!\!  U(\hbx_i^\nu) \!+\! \Big( \nabla F( \hbx_i^\nu )\! -\!\nabla \tf_i ( \hbx_i^\nu ; \bx_i^\nu) \! -\! \big( \by_i^\nu - \nabla f_i (\bx_i^\nu)\big)  \Big)^\top \!\!\!(\bx^\star - \hbx_i^\nu) + \frac{\mu}{2} \| \bx^\star -  \hbx_i^\nu\|^2\\
& \,\,\,\,=   U(\hbx_i^\nu) + \frac{\mu}{2} \Big\| \bx^\star -  \hbx_i^\nu + \frac{1}{\mu}\Big( \nabla F( \hbx_i^\nu ) -\nabla \tf_i ( \hbx_i^\nu ; \bx_i^\nu)  - \big(  \by_i^\nu - \nabla f_i (\bx_i^\nu) \big) \Big) \Big\|^2 \\
& \,\,\,\,\quad- \frac{1}{2\mu } \left\|  \nabla F( \hbx_i^\nu ) -\nabla \tf_i ( \hbx_i^\nu ; \bx_i^\nu)  - \big(\by_i^\nu - \nabla f_i (\bx_i^\nu) \big)  \right\|^2\\
&  \,\,\,\,\geq  U(\hbx_i^\nu)  - \frac{1}{2\mu } \left\|  \nabla F( \hbx_i^\nu )  \pm \nabla F(\bx_i^\nu)-\nabla \tf_i ( \hbx_i^\nu ; \bx_i^\nu)  - \big( \by_i^\nu - \nabla f_i (\bx_i^\nu) \big)  \right\|^2
\\
& \,\,\,\, \geq  U(\hbx_i^\nu)  - \frac{1}{\mu } \left \|  \nabla F( \hbx_i^\nu )  -  \nabla F(\bx_i^\nu) + \nabla f_i(\bx_i^\nu) - \nabla \tf_i ( \hbx_i^\nu ; \bx_i^\nu) \right\|^2 - \frac{1}{\mu} \|\bdelta_i^\nu \|^2 
\\
&  \,\,\,\, =  U(\hbx_i^\nu)  - \frac{1}{\mu } \left\| \int_{0}^{1}\left(\nabla^2 F(\theta \hbx_i^\nu+(1-\theta)\bx_i^\nu)-\nabla^2 \tf_i(\theta \hbx_i^\nu+(1-\theta)\bx_i^\nu; \bx_i^\nu)\right)(\mathbf{d}_i^\nu)\,\text{d}\theta \right\|^2 \!\!- \frac{1}{\mu} \|\bdelta_i^\nu \|^2 
\\
&  \,\,\,\, \geq  U(\hbx_i^\nu)  - \frac{D_i^2}{\mu } \norm{\mathbf{d}_i^\nu}^2 - \frac{1}{\mu} \|\bdelta_i^\nu \|^2.
\end{split}\vspace{-0.2cm}
\end{align*}
}

Rearranging the terms and summing over $i\in [m]$, yields  { \begin{equation}\label{eq:delta_x_lb_a}
\| \Deltaxi{}^\nu \|^2 \geq \frac{\mu}{D_{\mx}^2} \left(\sum_{i=1}^m \big( U( \hbx_i^\nu) - U(\bx^\star)\big) - \frac{1}{\mu}  \| \bdelta^\nu \|^2\right).\vspace{-0.1cm}
\end{equation}}
Using (\ref{eq:bound_doubly_stoc_cvx}) in conjunction with  $U(\bx_i^{\nu+\frac{1}{2}})\leq \alpha U(\hbx_i^\nu)+ (1-\alpha) U(\bx_i^\nu)$ leads to   \vspace{-0.2cm} \begin{equation}\label{eq:lower-bound_U_hat}
	\alpha \, \sum_{i=1}^m \left(U(\hbx_i^\nu) - U(\bx^\star)\right) \geq p^{\nu+1} - (1-\alpha) p^\nu.
\end{equation}
Combining (\ref{eq:delta_x_lb_a}) with (\ref{eq:lower-bound_U_hat}) provides the desired result (\ref{eq:delta_x_lb}). 
\end{proof}

As last step, we upper bound   $\|\bdelta_{}^\nu\|^2$ in (\ref{eq:opt_err_1}) in terms of the consensus errors $\| \bx_\bot^\nu\|^2$ and $\| \by_\bot^\nu\|^2$.

\vspace{-0.1cm}
\begin{lemma}\label{lem:tracking_err_bound}{ 
The following upper bound holds for the tracking error $\|\bdelta_{}^\nu\|^2$:
 \begin{equation}\label{eq:delta_upper_by_consensus}
 \|\bdelta_{}^\nu\|^2  \leq 4  L_{\mx}^2 \| \bx_\bot^\nu\|^2 + 2  \| \by_\bot^\nu \|^2,
\end{equation} 
where $L_{\mx}$ is defined in~\eqref{eq:prob_param_def}.}
\end{lemma}
\begin{proof}    \vspace{-0.2cm}
{ 
\begin{equation*}  
\begin{aligned}  
 \hspace{-1.3cm}
\|\bdelta_{}^\nu\|^2  \overset{\eqref{eq:tracking_err_def}}{=} & \sum_{i=1}^m \|\nabla F(\bx_i^{\nu}) \pm  \bar{\by}^\nu-  \by_i^\nu\|^2  
\\ 
  \overset{\eqref{eq:avg_y_def}}{=}   &  \frac{1}{m^2} \sum_{i=1}^m \Big\|\sum_{j=1}^m \nabla f_j(\bx_i^\nu)- \sum_{j=1}^m \nabla f_j(\bx_j^\nu) +   m\cdot \bar{\by}^\nu- m \cdot \by_i^\nu \Big\|^2
\\
\overset{\eqref{eq:mu-L-smooth},\,\eqref{eq:prob_param_def}}{\leq} &\frac{1}{m^2}\sum_{i=1}^m \!\!\left( 2 m\sum_{j=1}^m L_{\mx}^2\| \bx_i^\nu - \bx_j^\nu \|^2 + 2 m^2 \|\bar{\by}^\nu-  \by_i^\nu \|^2\right)
\\
 =  &  \,\,4  L_{\mx}^2 \| \bx_\bot^\nu\|^2 + 2  \| \by_\bot^\nu \|^2.
\end{aligned}\vspace{-0.4cm}
\end{equation*}}
 ~
\end{proof}

We are ready to prove the linear convergence of the optimality gap   up to consensus errors.  {The result is summarized in  Proposition~\ref{prop:loc_geo_rate} below. The proof follows readily multiplying~\eqref{eq:opt_err_1} and~(\ref{eq:delta_x_lb}) by $\tmu_{\mn} - \frac{L}{2} \alpha - \frac{1}{2}  \epsilon_{opt}$ and ${6 (L^2 + \tL_{\mx}^2)}/{\mu}$, respectively,    adding them  together }
 to cancel out  $\|\Deltaxi{}^\nu\|$, and   using (\ref{eq:delta_upper_by_consensus}) to bound $\|\bdelta_{}^\nu\|^2$. 

\begin{proposition}\label{prop:loc_geo_rate}
 The optimality gap  $p^\nu$ [cf. (\ref{eq:opt_gap})] satisfies  \vspace{-0.1cm}
{ 
	\begin{equation}\label{eq:opt_gap_plus_error}
p^{\nu + 1} \leq \sigma(\alpha) \cdot p^\nu + \eta(\alpha) \cdot  \left(4 L_{\mx}^2 \| \bx_\bot^\nu\|^2 + 2 \| \by_\bot^\nu \|^2\right) 
,\vspace{-0.1cm}
\end{equation}}
where   
 $\sigma(\alpha) \in (0,1)$ and $\eta(\alpha) >0$ are defined as \vspace{-0.2cm}
{ 
\begin{align}\label{eq:sigma_C_def}
\sigma(\alpha) & \triangleq 1-\alpha\,\frac{\left(1 - \frac{\alpha}{2}\right)\tmu_{\mn} + \frac{D_{\mn}^\ell}{2} \alpha - \frac{1}{2}  \epsilon_{opt}}{\frac{ D_{\mx}^2}{\mu} + \left(1 - \frac{\alpha}{2}\right)\tmu_{\mn} + \frac{D_{\mn}^\ell}{2} \alpha - \frac{1}{2}  \epsilon_{opt} },
\\
\eta(\alpha) & \triangleq \frac{ \frac{1}{2}\epsilon_{opt}^{-1} {\alpha} \cdot \frac{D_{\mx}^2}{\mu} + \frac{{\alpha}}{\mu} \cdot \left(\left(1 - \frac{\alpha}{2}\right)\tmu_{\mn} + \frac{D_{\mn}^\ell}{2} \alpha - \frac{1}{2}  \epsilon_{opt}\right)}{\frac{D_{\mx}^2}{\mu} +  \left(1 - \frac{\alpha}{2}\right)\tmu_{\mn} + \frac{D_{\mn}^\ell}{2} \alpha - \frac{1}{2}  \epsilon_{opt}};
\end{align}}
  $\epsilon_{opt}$ satisfies \eqref{eq:eps_opt};      and $L_{\mx}$ and  
  $\tmu_{\mn}$, $D_{\mn}^\ell$, $D_{\mx}$  are defined in~\eqref{eq:prob_param_def} and~\eqref{eq:alg_param_def}, respectively.   \end{proposition}

\subsubsection{Step 2: $\|\bx_\bot^\nu\|$ and $\|\by_\bot^\nu\|$   linearly converge up to    $\mathcal{O}(\| \Deltaxi{}^\nu\|)$  }\label{sec:consensus_err} 
We   upper bound $\| \bx_\bot^\nu\|$ and $ \| \by_\bot^\nu\|$  in terms of $\|\mathbf{d}^\nu\|$. 
We begin rewriting the SONATA algorithm \eqref{eq:loc_opt}-\eqref{eq:mix_y} in vector-matrix form; using \eqref{eq:stacked_vec} and \eqref{eq:def_W_hat},
 we have\vspace{-0.1cm}
 \begin{subequations}\label{alg:dsca_matrix}
\begin{align}
\bx^{\nu + 1} ={} & \hbW (\bx^\nu + \alpha \Deltaxi{}^\nu) \label{eq:x_update}\\
\by^{\nu + 1} ={} & \hbW (\by^\nu + \pgrad^{\nu +1} - \pgrad^\nu ).\label{eq:y_update}
\end{align}
\end{subequations}
Noting that $\bx_\bot^{\nu}=(\bI-\mathbf{J})\bx^{\nu}$ [similarly, $\by_\bot^{\nu}=(\bI-\mathbf{J})\by^{\nu}$] and  $(\bI-\mathbf{J})\hbW=\hbW-\mathbf{J}$ (due to the  doubly stochasticity of $\bW$), it follows from \eqref{alg:dsca_matrix} that \vspace{-0.1cm}
\begin{align}
\bx_\bot^{\nu + 1} & =  (\hbW - \bJ) (\bx_\bot^\nu + \alpha \Deltaxi{}^\nu) \label{eq:x_err_update}\\
\by_\bot^{\nu + 1}  & =   (\hbW - \bJ) (\by_\bot^\nu +\pgrad^{\nu +1} - \pgrad^\nu  ).\label{eq:y_err_update}
\end{align}
 Using (\ref{eq:x_err_update})-(\ref{eq:y_err_update}), Proposition~\ref{prop:err_x} below establishes linear  
 convergence  of the consensus errors $\bx_\bot^{\nu}$ and $\by_\bot^{\nu}$, up to a perturbation. 

\begin{proposition}\label{prop:err_x}
There holds: 
 \vspace{-0.2cm}
\begin{subequations}\label{eq:consensus-error-bounds}
\begin{align}
\| \bx_\bot^{\nu + 1} \| & \leq \rho \| \bx_\bot^\nu \| + \alpha \rho \| \Deltaxi{}^\nu\|,\label{eq:err_x_bound}\\
\| \by_\bot^{\nu +1}\| & \leq \rho  \| \by_\bot^\nu\| + 2L_{\mx} \rho \| \bx_\bot^\nu\| + \alpha  L_{\mx} \rho\|\Deltaxi{}^\nu\|,\label{eq:err_y_bound}
\end{align}
\end{subequations}
with $\rho$  and $L_{\mx} $ defined in~\eqref{eq:rho_def} and~\eqref{eq:prob_param_def}, respectively.
\end{proposition}
\begin{proof}
  We prove next \eqref{eq:err_y_bound};    \eqref{eq:err_x_bound} follows readily from \eqref{eq:x_err_update}. Using \eqref{eq:x_update},  \eqref{eq:y_err_update}, and the Lipschitz continuity of $\nabla f_i$ [cf. \eqref{eq:mu-L-smooth}], we can bound $\| \by_\bot^{\nu + 1}\|$ as\vspace{-0.1cm}
\begin{align*}
\begin{split}
\|\by_\bot^{\nu +1} \| &  \leq \rho \|\by_\bot^{\nu} \| + \rho \| \pgrad^{\nu + 1} - \pgrad^\nu\|\\
& \leq \rho \|\by_\bot^{\nu} \| + L_{\mx} \rho \| \underset{=(\hbW-\bI)\bx_{\bot}^{\nu}}{\underbrace{(\hbW-\bI)\bx^{\nu}}} + \alpha \hbW\Deltaxi{}^\nu\|\\
& \leq \rho \|\by_\bot^{\nu} \| + 2L_{\mx}\rho  \| \bx_\bot^\nu\| + \alpha  L_{\mx} \rho\|\Deltaxi{}^\nu\|,
\end{split}
\end{align*}
where in the last inequality we  used  $\|\bW\|\leq 1$. ~\end{proof}

\subsubsection{Step 3:   $\| \Deltaxi{}^\nu\|=\mathcal{O}(\sqrt{p^\nu}+\| \by_\bot^\nu\|)$ (closing the loop)}\label{sec:delta_x_bound}
Given the inequalities in Propositions \ref{prop:loc_geo_rate} and \ref{prop:err_x},   to close the loop, one needs to link   $\|\Deltaxi{}^\nu \|$ to the quantities in the aforementioned inequalities,   which is done next. 

\begin{proposition}\label{prop:bound_delta_x}
    The following upper bound holds for    $\| \Deltaxi{}^\nu \|$:\vspace{-0.1cm} 
{ \begin{equation}\label{eq:bound_delta_x}
\|\Deltax^\nu\|^2 
\leq  \frac{6}{\mu}\left(  \left( \frac{D_{\mx}}{\tmu_{\mn}} + 1\right)^2 + \frac{4 L_{\mx}^2}{\tmu_{\mn}^2 } \right) p^\nu  + \frac{3}{\tmu_{\mn}^2} \| \by_\bot^\nu\|^2.
\end{equation}
where $L_{\mx}$ and  $\tL_{\mx}$,  $\tmu_{\mn}$, $D_{\mx}$ are defined in~\eqref{eq:prob_param_def} and   \eqref{eq:alg_param_def}, respectively.}
\end{proposition}

\begin{proof}{ 
By  optimality of $\hbx_i^\nu$ and $\bx^\star$ we have\vspace{-.4cm}
\begin{align*}
\left(\nabla \tf_i (\hbx_i^\nu ; \bx_i^\nu) + \by_i^\nu - \nabla f_i (\bx_i^\nu)\right)^\top \left(\bx^\star-\hbx_i^\nu\right) + G(\bx^\star)-G(\hbx_i^\nu) &\geq 0,\\
 \nabla F (\bx^\star)^\top  \left(\hbx_i^\nu-\bx^\star\right) +  G(\hbx_i^\nu) - G(\bx^\star) & \geq 0.
\end{align*}

Summing the two inequalities above yields \begin{align*}
\begin{split}
0  \leq  & {}\left(\nabla F (\bx^\star) - \by_i^\nu + \nabla f_i (\bx_i^\nu)  - \nabla \tf_i (\hbx_i^\nu ; \bx_i^\nu)  \pm \bar{\by}^\nu\right)^\top (\hbx_i^\nu - \bx^\star)\\
 \leq & {}\left(\nabla F (\bx^\star)  - \frac{1}{m}\sum_{j=1}^m \nabla f_j(\bx_j^\nu)+ \nabla f_i (\bx_i^\nu)  - \nabla \tf_i (\hbx_i^\nu ; \bx_i^\nu)  \right)^\top (\hbx_i^\nu - \bx^\star)\\
 & {}+ \| \bar{\by}^\nu - \by_i^\nu \| \|\hbx_i^\nu - \bx^\star\|\\
  \leq & {}\left(\nabla F (\bx^\star)  - \nabla F(\bx_i^\nu) + \nabla f_i (\bx_i^\nu)  - \nabla \tf_i (\hbx_i^\nu ; \bx_i^\nu)  \right)^\top (\hbx_i^\nu - \bx^\star)\\
 & {}+ \| \bar{\by}^\nu - \by_i^\nu \| \|\hbx_i^\nu - \bx^\star\| + \norm{\nabla F(\bx_i^\nu)  - \frac{1}{m}\sum_{j=1}^m \nabla f_j(\bx_j^\nu)}\norm{\hbx_i^\nu - \bx^\star}
 \end{split}
 \end{align*}
 \begin{align*}
 \begin{split}
\leq &{}\left(\nabla F (\bx^\star)  - \nabla F(\bx_i^\nu) + \nabla f_i (\bx_i^\nu)  \pm\nabla \tf_i (\bx^\star ; \bx_i^\nu)   - \nabla \tf_i (\hbx_i^\nu ; \bx_i^\nu)  \right)^\top (\hbx_i^\nu - \bx^\star)\\
 & {}+ \| \bar{\by}^\nu - \by_i^\nu \| \|\hbx_i^\nu - \bx^\star\| +\left(\frac{1}{m} \sum_{j=1}^m L_j \norm{\bx_i^\nu - \bx_j^\nu}\right) \norm{\hbx_i^\nu - \bx^\star}\\
\leq & {} \left(\int_0^1 \Big(\nabla^2 F (\theta \bx^\star + (1 - \theta) \bx_i^\nu)  - \nabla^2 \tf_i \big(\theta \bx^\star + (1 - \theta) \bx_i^\nu; \bx_i^\nu\big) \Big)( \bx^\star - \bx_i^\nu)\,\text{d}\theta\right)^\top (\hbx_i^\nu - \bx^\star) \\
& {}- \tmu_i \norm{\hbx_i^\nu - \bx^\star}^2 + \| \bar{\by}^\nu - \by_i^\nu \| \|\hbx_i^\nu - \bx^\star\| +\left(\frac{1}{m} \sum_{j=1}^m L_j \norm{\bx_i^\nu - \bx_j^\nu}\right) \norm{\hbx_i^\nu - \bx^\star}\end{split}\end{align*} \begin{align*}\begin{split}
 \leq & \, D_i \norm{ \bx^\star - \bx_i^\nu} \norm{\hbx_i^\nu - \bx^\star} - \tmu_i \norm{\hbx_i^\nu - \bx^\star}^2 + \| \bar{\by}^\nu - \by_i^\nu \| \|\hbx_i^\nu - \bx^\star\| \\
& {}+\left(\frac{1}{m} \sum_{j=1}^m L_j \norm{\bx_i^\nu - \bx_j^\nu}\right) \norm{\hbx_i^\nu - \bx^\star}.
\end{split}
\end{align*}

Rearranging terms and using the reverse triangle inequality we obtain  the following bound for $\|\Deltaxi{i}^\nu\|$:\vspace{-0.2cm}
\begin{multline}
D_i \norm{ \bx^\star - \bx_i^\nu}   + \| \bar{\by}^\nu - \by_i^\nu \| +\left(\frac{1}{m} \sum_{j=1}^m L_j \norm{\bx_i^\nu - \bx_j^\nu}\right)\\
 \geq \tmu_i \norm{\hbx_i^\nu - \bx^\star} \geq \tmu_i \left(\|\Deltaxi{i}^\nu\| - \| \bx^\star - \bx_i^\nu\| \right).
\end{multline}
Therefore, \begin{align*}
\hspace{-0.2cm}\begin{split}
\|\Deltaxi{i}^\nu\|^2  &\leq \ 3 \left( \frac{D_i}{\tmu_i} + 1\right)^2 \norm{ \bx^\star - \bx_i^\nu}^2   + \frac{3}{\tmu_i^2} \| \bar{\by}^\nu - \by_i^\nu \|^2 + \frac{3}{\tmu_i^2}\left(\frac{1}{m} \sum_{j=1}^m L_j \norm{\bx_i^\nu - \bx_j^\nu}\right)^2
\\
&\leq   \ 3 \left( \frac{D_i}{\tmu_i} + 1\right)^2 \norm{ \bx^\star - \bx_i^\nu}^2   + \frac{3}{\tmu_i^2} \| \bar{\by}^\nu - \by_i^\nu \|^2 
+ \frac{6 L_{\mx}^2}{\tmu_i^2 m} \left(\sum_{j=1}^m \|\bx_j^\nu - \bx^\star\|^2 +  m\|\bx_i^\nu - \bx^\star \|^2\right).
\end{split}
\end{align*}
Summing over $i = 1,\ldots, m$, yields
\begin{align*}
\|\Deltax^\nu\|^2 \leq  & \ \left( 3 \left( \frac{D_{\mx}}{\tmu_{\mn}} + 1\right)^2 + \frac{12L_{\mx}^2}{\tmu_{\mn}^2 } \right)\sum_{j=1}^m\norm{\bx^\nu_j-\bx^\star}^2   + \frac{3}{\tmu_{\mn}^2} \| \by_\bot^\nu\|^2\\
\leq  &\ \frac{6}{\mu}\left(  \left( \frac{D_{\mx}}{\tmu_{\mn}} + 1\right)^2 + \frac{{4} L_{\mx}^2}{\tmu_{\mn}^2 } \right) p^\nu  + \frac{3}{\tmu_{\mn}^2} \| \by_\bot^\nu\|^2.
\end{align*}}
\end{proof}

\subsubsection{Step 4: Proof of the linear rate (chaining the inequalities)}\label{sec:small_gain}

We are now ready to prove linear rate of the SONATA  algorithm.  We build on the following intermediate result, introduced in \cite{nedich2016achieving}. 
\begin{lemma}\label{Lemma_Small_gain} Given the sequence $ \{s^\nu\}$, define  the   transformations  \vspace{-0.2cm}
\begin{equation}\label{eq:seq_transform}
\seqnorm{S} \triangleq \max_{\nu = 0, \ldots, K}  |s^\nu| z^{-\nu} \quad \text{and} \quad S(z) \triangleq \sup_{\nu \in \natural} |s^\nu| z^{-\nu},\vspace{-0.2cm}
\end{equation}
for $z \!\in \!(0,1)$.  If $S(z)$ is bounded, then $|s^\nu|=\mathcal{O}(z^\nu)$. 
\end{lemma}

We show next how to chain the inequalities \eqref{eq:opt_gap_plus_error}, \eqref{eq:consensus-error-bounds} and \eqref{eq:bound_delta_x} 
so that Lemma \ref{Lemma_Small_gain} can be applied to the    sequences $\{ p^\nu \}$,  $\{ \|\bx_\bot^\nu\|^2 \}$,  $\{ \| \by_\bot^\nu\|^2\}$  and $\{ \|\Deltaxi{}^\nu\|^2 \}$, establishing thus their linear convergence. 

\begin{proposition}\label{prop:small_gain_eqs_transformed}
\!\! Let $\seqnorm{P}$, $\seqnorm{X_\bot}$, $\seqnorm{Y_\bot}$ and $\seqnorm{D}$  denote the transformation \eqref{eq:seq_transform} applied to the   sequences $\{ p^\nu \}$,\! $\{ \|\bx_\bot^\nu\|^2 \}$,  $\{ \| \by_\bot^\nu\|^2 \}$  and $\{ \|\Deltaxi{}^\nu\|^2 \}$, respectively. Given the constants $\sigma(\alpha)$ and $\eta(\alpha)$ (defined in Proposition~\ref{prop:loc_geo_rate}) and the free parameters $\epsilon_x, \epsilon_y >0$ (to be determined), the following hold 
\begin{subequations}\label{eq:all_ineq_transf}
\begin{align}
\seqnorm{P} & \leq G_P(\alpha,z) \cdot \left( 4   L_{\mx}^2 \seqnorm{X_\bot} +  2  \seqnorm{Y_\bot}\right) + \omega_p, \label{eq:bound_P}\\
\seqnorm{X_\bot} & \leq G_X (z)  \cdot \rho^2 \alpha^2 \seqnorm{D} + \omega_x, \label{eq:bound_X}\\
\seqnorm{Y_\bot} & \leq G_Y(z) \cdot 8L_{\mx}^2 \rho^2  \seqnorm{X_\bot} + G_Y(z) \cdot 2L_{\mx}^2 \rho^2  \alpha^2 \seqnorm{D} + \omega_y,  \label{eq:bound_Y}\\
\seqnorm{D} & \leq C_1 \cdot \seqnorm{P} + C_2 \cdot \seqnorm{Y_\bot}, \label{eq:bound_Dx}
\end{align}\end{subequations}
for all \vspace{-0.3cm}\begin{align}\label{eq:bound_z}
z \in \left(\max\{ \sigma(\alpha),  \rho^2(1 + \epsilon_x), \rho^2 (1+ \epsilon_y)\},1 \right),
\end{align} where\vspace{-0.4cm}
\begin{subequations}\begin{align}
& G_P(\alpha,z) \triangleq \frac{\eta(\alpha)  }{z - \sigma(\alpha)}, && \omega_p \triangleq  \frac{z}{z - \sigma(\alpha)} \cdot p^0
\\
& G_X (z) \triangleq \frac{ (1+ \epsilon_x^{-1})}{z- \rho^2(1 + \epsilon_x)}, && \omega_x \triangleq \frac{z}{z- \rho^2(1 + \epsilon_x)} \cdot \| \bx_\bot^0 \|^2,
\\
& G_Y (z) \triangleq  \frac{ (1+ \epsilon_y^{-1})}{z- \rho^2(1 + \epsilon_y)}, && \omega_y \triangleq \frac{z}{z- \rho^2(1 + \epsilon_y)} \cdot \| \by_\bot^0 \|^2,
\\
& {C_1 \triangleq  \frac{6}{\mu}\left(  \left( \frac{D_{\mx}}{\tmu_{\mn}} + 1\right)^2 + \frac{4 L_{\mx}^2}{\tmu_{\mn}^2 } \right),} &&  {C_2 \triangleq \frac{4}{\tmu_{\mn}^2}}\label{eq:C1_C2}. 
\end{align}
\end{subequations}
\end{proposition}

\begin{proof}
Squaring  \eqref{eq:consensus-error-bounds}  and  using    Young's inequality yield 
\begin{align}\label{eq:small_gain_eqs}
\begin{split}
\|\bx_\bot^{\nu +1} \|^2 & \leq \rho^2(1 + \epsilon_x) \|\bx_\bot^{\nu} \|^2 + \rho^2 (1+ \epsilon_x^{-1})\alpha^2 \|\Deltaxi{}^\nu\|^2 \\
\| \by_\bot^{\nu +1}\|^2 &\leq \rho^2 (1+ \epsilon_y)  \| \by_\bot^\nu\|^2 + \rho^2(1 + \epsilon_y^{-1}) \Big(8L_{\mx}^2  \| \bx_\bot^\nu\|^2 +  2\alpha^2  L_{\mx}^2  \|\Deltaxi{}^\nu\|^2 \Big),
\end{split}
\end{align}
for arbitrary  $\epsilon_x,\epsilon_y>0$.
The proof is completed by taking the maximum of both sides of   \eqref{eq:opt_gap_plus_error},  \eqref{eq:bound_delta_x}, and 
\eqref{eq:small_gain_eqs} over $ \nu = 0,\ldots, K$  and using  
$\max_{\nu = 0,\ldots, K} |s^{\nu +1}| z^{-\nu} \geq z \cdot \max_{\nu = 0,\ldots, K} |s^{\nu}| \,z^{-\nu} - z \cdot |s^0|$, for  any sequence $\{s^\nu\}$ and $z\in (0,1)$. 
\end{proof}

Chaining the inequalities in~Proposition~\ref{prop:small_gain_eqs_transformed} in the way shown in Fig.~\ref{fig:small_gain}, we can bound $\seqnorm{D}$  as (see Appendix~\ref{app:pf_small_gain_chain} for the proof) 
\begin{align}\label{eq:small_gain_chain}
\seqnorm{D} \leq \PP (\alpha ,z) \cdot \seqnorm{D} + \RR (\alpha ,z),
\end{align}
where  $\PP (\alpha ,z)$ is  defined as
\begin{align}
\label{eq:stability_polynomial}
\begin{split}
\PP (\alpha ,z)  \triangleq {} &  G_P (\alpha,z) \cdot G_X(z) \cdot C_1 \cdot 4  L_{\mx}^2 \cdot \rho^2 \cdot \alpha^2\\
& + \left( G_P (\alpha,z) \cdot 2   C_1   + C_2  \right)  \cdot G_Y(z) \cdot 2 L_{\mx}^2 \rho^2  \cdot \alpha^2\\
& +\left( G_P (\alpha,z) \cdot 2   C_1   + C_2   \right) \cdot G_Y(z)\cdot 8 L_{\mx}^2 \rho^2 \cdot G_X(z) \cdot \rho^2 \cdot \alpha^2,
\end{split}
\end{align}
and $\RR(\alpha ,z)$ is a  remainder, which is bounded under  \eqref{eq:bound_z}. 
 \begin{figure}\label{fig:small_gain}
 \caption{Chain of the inequalities in Proposition~\ref{prop:small_gain_eqs_transformed} leading to (\ref{eq:small_gain_chain}).}
 \centering
  \begin{tikzpicture}[scale=.7]\footnotesize
\tikzstyle{edge} = [  shorten >=.5pt,shorten <=.5pt,line width=0.7pt,>=stealth,->]
\node (Dx1) at (0,0) {$D^K$};

\node (P2) at ($(Dx1) + (2,1)$) {$P^K$};
\node (Y2) at ($(Dx1) + (2,-1)$) {$Y_\bot^K$};

\node (X3) at ($(P2) + (2,1)$) {$X_\bot^K$};
\node (Y3) at ($(P2) + (2,-1)$) {$Y_\bot^K$};
\node (Y3') at ($(Y2) + (2,0)$) {$Y_\bot^K$};

\node (Y3combined) at ($(Y3) + (0.3,-0.5)$){};

\node (Dx4) at ($(X3) + (2,0)$){$D^K$};
\node (X4) at ($(Y3combined) + (2,1)$) {$X_\bot^K$};
\node (Dx4') at ($(Y3combined) + (2,-1)$) {$D^K$};

\node (Dx5) at ($(X4) + (2,0)$){$D^K$};

\draw[edge] (P2) -- node [above]{\eqref{eq:bound_Dx}}  (Dx1);
\draw[edge] (Y2) -- node [below]{\eqref{eq:bound_Dx}} (Dx1);

\draw[edge] (X3) -- node [above]{\eqref{eq:bound_P}} (P2);
\draw[edge] (Y3) -- node [below]{\eqref{eq:bound_P}} (P2);
\draw[shorten >=.5pt,shorten <=.5pt,line width=0.7pt] (Y3') -- (Y2);

\draw[edge] (Dx4) -- node [above]{\eqref{eq:bound_X}} (X3);
\draw[edge] (X4) -- node [above]{\eqref{eq:bound_Y}} (Y3combined);
\draw[edge] (Dx4') -- node [below]{\eqref{eq:bound_Y}} (Y3combined);

\draw[edge] (Dx5) -- node [above]{\eqref{eq:bound_X}} (X4);

\draw[dashed] (Y3.north west) rectangle (Y3'.south east);

\end{tikzpicture}\vspace{-0.7cm}
 \end{figure}
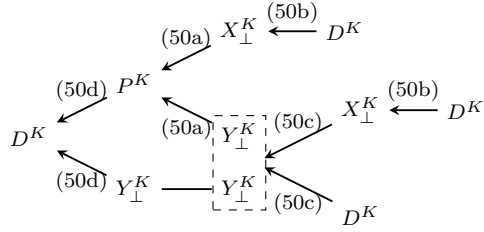

Therefore, as long as $\PP (\alpha ,z) < 1$,  \eqref{eq:small_gain_chain} implies 
\begin{equation}\vspace{-0.2cm}
\seqnorm{D} \leq \frac{\RR (\alpha ,z)}{1 - \PP (\alpha ,z)} \leq B < +\infty
\end{equation}
where  $B$ is  a constant independent of $K$. Therefore,  $D(z) \leq B$ and thus $\{ \|  \Deltaxi{}^\nu\|^2 \}$ converges R-linearly to zero at rate at least $z$ (cf. Lemma~\ref{Lemma_Small_gain}).
Applying the same argument to the other inequalities in~ Proposition \ref{prop:small_gain_eqs_transformed}, one can conclude that also the sequences  $\{ p^\nu \}$, $\{ \|\bx_\bot^\nu\|^2 \}$ and $\{ \| \by_\bot^\nu\| \}$ converge R-linearly to zero.

The last step consists to showing that there exist a sufficiently small step-size $\alpha\in (0,1]$ and   $z\in (0,1)$ satisfying \eqref{eq:bound_z}, such that $\PP (\alpha ,z) < 1$. This is proved in the Theorem \ref{thm:step_size_condition} below. 
\begin{theorem}\label{thm:step_size_condition}
 Consider Problem \eqref{eq:P} under Assumptions \ref{assump:p}-\ref{assump:network}; \!and 
the  SONA- TA  algorithm \eqref{eq:loc_opt}-\eqref{eq:mix_y}, under Assumptions~\ref{assump:SCA_surrogate} and \ref{assump:weight}, with  {$\tmu_{\mn}\geq  D_{\mn}^\ell$}.
 Then, there exists a sufficiently small step-size $\bar{\alpha}\in (0,1]$ [see the proof for its  expression] such that for all $\alpha < \bar{\alpha}$, $\{U(\bx_i^\nu)\}$ converges to $U^\star$  at an R-linear rate,   $i\in [m]$.
\end{theorem}
\begin{proof} The proof is organized in following two steps: \textbf{Step 1)} We first consider the ``marginal'' stable case by letting $z=1$, and show that there exists    $\bar{\alpha} > 0$ so that $\PP(\alpha,1) <1$, for all $\alpha \in (0, \bar{\alpha})$; \textbf{Step 2)}  Then, invoking the continuity of $\PP(\alpha,z)$, we argue that, for any  $\alpha \in (0, \bar{\alpha})$, one can find    $\bar{z}(\alpha)<1$ such that $\PP \big(\alpha,\bar{z}(\alpha) \big) < 1$. This implies the boundedness of $D^K\big( \bar{z}(\alpha) \big)$, and thus $ \|  \Deltaxi{}^\nu \|^2=\mathcal{O} \big(\bar{z}(\alpha)^\nu \big)$ (cf. Lemma~\ref{Lemma_Small_gain}). 

\noindent $\bullet$ \textbf{Step 1:} We begin optimizing the free parameters $\epsilon_x$, $\epsilon_y$, and $\epsilon_{opt}$. 
Since the goal is to find the largest $\bar{\alpha}$ so that $\PP (\alpha , 1) <1$, for all $\alpha \in (0, \bar{\alpha})$, the optimal choice of $\epsilon_x$, $\epsilon_y$, and $\epsilon_{opt}$ is the one that  minimizes $\PP (\alpha , 1)$, that is,
\vspace{-0.1cm}  \begin{equation}
  \epsilon^\star = \argmin_{\epsilon>0} \frac{1+\epsilon^{-1}}{1 - \rho^2 (1 + \epsilon)} = \frac{1 - \rho}{\rho}.\vspace{-0.1cm}
\end{equation}
We then set   $\epsilon_x = \epsilon_y= \epsilon^\star$, and  proceed   to optimize $\epsilon_{opt}$, which  appears in $\eta (\alpha)$ and $\sigma(\alpha)$. Recalling the definition of 
  $\eta (\alpha)$ and $\sigma(\alpha)$ (cf.   Proposition~\ref{prop:loc_geo_rate}) and the constraint (\ref{eq:eps_opt}),  
the problem boils down to minimize    \vspace{-0.1cm}
\begin{equation*}\label{eq:ratio_alpha}
{  G_P(\alpha,1) = \frac{ \eta(\alpha)}{1 - \sigma(\alpha)} =\frac{ \frac{1}{2} \epsilon_{opt}^{-1} \cdot \frac{D_{\mx}^2}{\mu} + \frac{1}{\mu} \cdot \left(\left(1 - \frac{\alpha}{2}\right)\tmu_{\mn} + \frac{D_{\mn}^\ell}{2} \alpha - \frac{1}{2}  \epsilon_{opt} \right)}{ \left(1 - \frac{\alpha}{2}\right)\tmu_{\mn} + \frac{D_{\mn}^\ell}{2} \alpha - \frac{1}{2}  \epsilon_{opt} },}\vspace{-0.1cm}
\end{equation*}
   subject to $\epsilon_{opt}\in (0, 2\tmu_{\mn} - \alpha( \tmu_{\mn} -   D_{\mn}^\ell))$. To have a nonempty feasible set, we require   $\alpha < {2 \tmu_{\mn}}/({\tmu_{\mn} - D_{\mn}^\ell})$ (recall that it is assumed   {$\tmu_{\mn}\geq  D_{\mn}^\ell$}).
Setting the derivative of $G_P(\alpha,1)$ with respect to $\epsilon_{opt}$ to zero, yields 
$\epsilon_{opt}^\star =\left(1 - \frac{\alpha}{2}\right) \tmu_{\mn} + \alpha  {D_{\mn}^\ell}/{2}$, which is strictly feasible, and thus the solution. 

Let $\PP^\star(\alpha,z)$ denote  the value of $\PP(\alpha,z)$ corresponding to the optimal choice of the above parameters. 
The expression of $\PP^\star (\alpha,1)$ reads\vspace{-0.2cm}
\begin{align}\label{eq:P_star}
\begin{split}
\PP^\star (\alpha ,1)  \triangleq {} &  G_P^\star (\alpha) \cdot C_1 \cdot 4  L_{\mx}^2 \cdot \frac{\rho^2 }{(1 - \rho)^2 } \cdot \alpha^2\\
& + \left( G_P^\star (\alpha) \cdot 2   C_1   + C_2  \right)   \cdot 2 L_{\mx}^2 \cdot \frac{\rho^2 }{(1 - \rho)^2 } \cdot \alpha^2\\
& +\left( G_P^\star (\alpha) \cdot 2   C_1   + C_2   \right) \cdot  8 L_{\mx}^2 \cdot \frac{\rho^4 }{(1 - \rho)^4 } \cdot \alpha^2,
\end{split}
\end{align}
where  \vspace{-0.2cm}
\begin{equation}\label{eq:r_star_def}
{  G_P^\star (\alpha) \triangleq \frac{\frac{D_{\mx}^2}{\mu} + \frac{1}{\mu} \cdot \left(\left(1 - \frac{\alpha}{2}\right) \tmu_{\mn} + \frac{D_{\mn}^\ell}{2} \alpha \right)^2}{ \left(\left(1 - \frac{\alpha}{2}\right) \tmu_{\mn} + \frac{D_{\mn}^\ell}{2} \alpha\right)^2 }.}
\end{equation}

\noindent $\bullet$ \textbf{Step 2:}  Since $\PP^\star (\bullet ,1)$ is continuous and monotonically increasing on $(0, 2 \tmu_{\mn} /$ $ (\tmu_{\mn} - D_{\mn}^\ell)$, with $\PP^\star (0 ,1) = 0$, there  exists  some $\bar{\alpha} < 2 \tmu_{\mn} / (\tmu_{\mn} - D_{\mn}^\ell)$ such that $\PP^\star (\alpha ,1) < 1,$ for all $\alpha \in (0, \bar{\alpha})$.
One can verify that, for any $\alpha\in (0, 2 \tmu_{\mn} / (\tmu_{\mn} - D_{\mn}^\ell))$, $\PP^\star(\alpha,z)$ is continuous at $z = 1$. Therefore, for any fixed $\alpha \in (0, \bar{\alpha})$, $\PP^\star (\alpha ,1) < 1$ implies the existence of some $\bar{z}(\alpha) < 1$ such  that $\PP^\star(\alpha,\bar{z}(\alpha))<1$. 

We conclude the proof   providing  the expression of  a valid  $\bar{\alpha}$. Restricting   $\alpha \leq  \tmu_{\mn} / (\tmu_{\mn} - D_{\mn}^\ell)$, we   upper bound $G_P^\star (\alpha)$ by $G_P^\star (\tmu_{\mn}/(\tmu_{\mn} - D_{\mn}^\ell))$. Using for $G_P^\star (\alpha)$ this upper bound in (\ref{eq:P_star}) and  solving the resulting $\PP^\star (\alpha ,1) < 1$ for $\alpha$, yield\vspace{-0.1cm}  
\begin{align}\label{eq:expression_alpha}
\begin{split}
\alpha < \alpha_1 \triangleq & \left( G_P^\star \left( \frac{\tmu_{\mn}}{\tmu_{\mn} - D_{\mn}^\ell} \right) \cdot C_1 \cdot 4  L_{\mx}^2 \cdot \frac{\rho^2 }{(1 - \rho)^2 } \right.\\
& \quad + \left( G_P^\star\left( \frac{\tmu_{\mn}}{\tmu_{\mn} - D_{\mn}^\ell} \right)  \cdot 2 C_1   + C_2  \right)   \cdot 2 L_{\mx}^2 \cdot \frac{\rho^2 }{(1 - \rho)^2 } \\
& \quad \left.+\left( G_P^\star \left( \frac{\tmu_{\mn}}{\tmu_{\mn} - D_{\mn}^\ell} \right) \cdot 2 C_1   + C_2   \right) \cdot  8 L_{\mx}^2 \cdot \frac{\rho^4 }{(1 - \rho)^4 }    \right)^{-\frac{1}{2}}.
\end{split}
\end{align}
Therefore, a valid $\bar{\alpha}$ is 
$\bar{\alpha} = \min\{  \tmu_{\mn} / (\tmu_{\mn} - D_{\mn}^\ell), \alpha_1\}$.
~
\end{proof}

The next theorem provides an explicit expression of the convergence rate in Theorem \ref{thm:step_size_condition} in terms of the step-size $\alpha$; the  constants $J$, $A_{\frac{1}{2}}$, and $\alpha^*$  therein are defined in~\eqref{eq:expression_J}, \eqref{eq:rate_condition_2} with $\theta = 1/2$, and ~\eqref{eq:expression_alpha_ast}, respectively.

\begin{theorem}\label{thm:linear_rate} In the setting of Theorem \ref{thm:step_size_condition},  
suppose that the step-size $\alpha$ satisfies $\alpha \in (0,\alpha_{\mx})$, with 
$\alpha_{\mx} \triangleq \min\{(1-\rho)^2/A_{\frac{1}{2}},\tmu_{\mn} /(\tmu_{\mn} - D_{\mn}),1 \}.$
Then, $U(\bx_i^\nu)-U^\star=\mathcal{O}(z^\nu)$, for all $i\in [m]$, where \vspace{-0.2cm}
\begin{equation}\label{eq:rate}
z =  
\begin{cases}
1 - J \cdot \alpha & \text{for }\alpha \in \left( 0,\min\{\alpha^*,\alpha_{\mx}\}\right), \\
\bigg(\rho + \sqrt{ \alpha A_{\frac{1}{2}}   } \bigg)^2 & \text{for } \alpha \in \left[\min\{\alpha^*,\alpha_{\mx}\}, \alpha_{\mx} \right).
\end{cases}
\end{equation}
\end{theorem}
\begin{proof}
See Appendix~\ref{app:pf_linear_rate}.\end{proof}

 \subsection{Discussion} \label{sec:discussion}{ 
 Theorem~\ref{thm:linear_rate} provides a unified set of convergence conditions for different choices of surrogates and network topologies. 
 To shed light on the expression of the rate and its dependence on the key optimization and network parameters, we customize here Theorem~\ref{thm:linear_rate} to specific network topologies and   surrogate functions. 
 We   begin considering star-networks (cf.~Sec.~\ref{sec:star-topology}) and then   move to  general graph topologies with no master node (cf.~Sec.~\ref{sec_arbitrary-topology}). We  will 
 customize the rate achieved by SONATA employing the following two  surrogate functions $\tf_i$, representing the two extreme choices in the spectrum of admissible surrogates: \begin{itemize}
  	\item \textbf{Linearization:} \vspace{-0.2cm}\begin{equation}
  		\label{eq:linear_surrogate}
\tf_i (\bx_i; \bx_i^\nu)  \triangleq  \nabla f_i (\bx_i^\nu)^\top(\bx_i - \bx_i^\nu) + \frac{L}{2}\|\bx_i - \bx_i^\nu\|^2;\vspace{-0.1cm} 
  	\end{equation}
  	\item \textbf{Local $f_i$:} \vspace{-0.3cm}\begin{equation} 
 	\label{eq:f_surrogate}
\tf_i (\bx_i; \bx_i^\nu)  \triangleq f_i (\bx_i) + \frac{\beta}{2} \| \bx_i - \bx_i^\nu \|^2.
 \end{equation}
  \end{itemize}}
 
 \subsubsection{Star-networks: SONATA-Star}\label{sec:star-topology} Convergence of SONATA-Star (Algorithm~\ref{alg:SONATA-star}) is established in Corollary \ref {cor:centralized_rate} below. 

{  \begin{corollary}\label{cor:centralized_rate}
 Consider Problem~\eqref{eq:P} under Assumption~\ref{assump:p}   over a  star-network; let $\{\bx^\nu\}$ be the sequence generated by   SONATA-Star  (Algorithm~\ref{alg:SONATA-star}), based on the surrogate functions satisfying Assumption \ref{assump:SCA_surrogate} and step-size $\alpha \in (0, \min(2 \tmu_{\mn}/(\tmu_{\mn} - D_{\mn}^\ell),1)]$.
Then, 
for all $i =1,\ldots,m$,   \vspace{-0.2cm}
\begin{equation}\label{eq:rate_central_expression}
U(\bx^\nu)-U^\star=\mathcal{O}(z^\nu),\quad \text{with}\quad z = 1 - \alpha \cdot \frac{  \left(1 - \frac{\alpha}{2}\right) \tmu_{\mn} +  \frac{\alpha D_{\mn}^\ell}{2}  }{\frac{ D_{\mx}^2}{2 \mu}  +  \left(1 - \frac{\alpha}{2}\right) \tmu_{\mn} +  \frac{\alpha D_{\mn}^\ell}{2}}.
\end{equation}
In particular, when the surrogates \eqref{eq:linear_surrogate} and \eqref{eq:f_surrogate} are employed along with $\alpha=1$, the rate above  reduces to the following expressions: \begin{itemize} 
\item \textbf{Linearization \eqref{eq:linear_surrogate}:}  $z \leq 1 - \kappa_g^{-1}$. Therefore, $U(\bx^\nu)-U^\star\leq \epsilon$ in at most $\mathcal{O}\Big(\kappa_g \log({1}/{\epsilon})\Big)$ iterations (communications);
\item \textbf{Local $f_i$ \eqref{eq:f_surrogate}:} 
\begin{equation}\label{eq:z_surrogate}
 z \leq 1 - \frac{1}{1 + 4 \cdot \frac{\beta}{\mu} \cdot \min\{1, \frac{\beta}{\mu}\}}.
 \end{equation}
Therefore, $U(\bx^\nu)-U^\star\leq \epsilon$ in at most\vspace{-0.2cm} 
\begin{equation}\label{eq:rate_SONATA_Star_f_surrogate}\left\{\hspace{-0.4cm}\begin{array}{ll}
	&\mathcal{O}\left(1\cdot \log\big(1/\epsilon\big)\right), \,\qquad   \text{if } \,\beta\leq \mu,\\
	& \mathcal{O}\left( \frac{\beta}{\mu}\cdot\,  \log\big(1/\epsilon\big)\right),  \quad  \text{if }\, \beta> \mu,
\end{array}
\right.\end{equation}
iterations (communications).
\end{itemize} 
\end{corollary}
\begin{proof} See Appendix \ref{app:proof-SONATA-star}.\end{proof}  
 
}
{ 
The following comments are in order. When linearization is employed,   SONATA-Star matches the iteration complexity 
 of the centralized proximal-gradient algorithm. 
 When the $f_i$'s are sufficiently similar,   \eqref{eq:z_surrogate}-\eqref{eq:rate_SONATA_Star_f_surrogate} proves that faster rates can be achieved if  surrogates (\ref{eq:f_surrogate}) are chosen over first-order approximations:  
 when   $\beta \ll L$,    \eqref{eq:rate_SONATA_Star_f_surrogate} is  significantly faster than $\mathcal{O}\big(\kappa_g \log({1}/{\epsilon})\big)$. 
 As case study, consider   Example 2 (cf. Sec.~\ref{sec_related}): plugging   (\ref{eq:beta_vs_kappa_ML_problem}) into Corollary~\ref{cor:centralized_rate} shows that using the surrogates (\ref{eq:f_surrogate}) yields {  $\widetilde{\mathcal{O}}\big(L \,\sqrt{{d\,m}}\cdot\log(1/\epsilon) \big)$ }iterations (communications); this  contrasts with  {  $\widetilde{\mathcal{O}}\big(L \,\sqrt{d\,m\,n}\cdot\log(1/\epsilon) \big)$}, achieved by first-order methods (and  SONATA-Star using linearization), which instead increases with the sample size $n$.


 
 \paragraph*{Comparison with DANE \& CEASE} Since SONATA-Star contains as special cases the DANE \cite{DANE} and CEASE \cite{Fan2020} algorithms, we contrast here Corollary~\ref{cor:centralized_rate}  with their convergence rates.  We recall that 
  DANE  is applicable to \eqref{eq:P} when $G=0$: For  quadratic losses, it achieves an $\epsilon$-optimal objective value in $\mathcal{O}\big((\beta/\mu)^2\cdot \log(1/\epsilon)\big)$ iterations/communications (here $\beta/\mu\geq 1$). 
   This rate is    worse than  
   \eqref{eq:rate_SONATA_Star_f_surrogate}. 
   For nonquadratic losses, \cite{DANE}  did not  show any rate improvement  of DANE over plain gradient algorithms, i.e.,  $\mathcal{O}\big(\kappa_g\cdot \log(1/\epsilon)\big)$ while SONATA-star  still retains   $\mathcal{O}\big(\beta/\mu\cdot \log(1/\epsilon)\big)$. The CEASE algorithm 
   is proved to achieve an $\epsilon$-solution on the iterates in $\mathcal{O}\big((\beta/\mu)^2\cdot \log(1/\epsilon)\big)$ iterations/communications (with $\beta/\mu\geq 1$);   SONATA  reaches the same error on the iterates in   $\mathcal{O}\big(\beta/\mu\cdot \log(\kappa_g/\epsilon)\big)$ iterations/communications, which matches the order of the mirror-decent algorithm. 
    
 In the next section we extend the study  to   networks with no centralized nodes, sheding lights on the role of the network in  achieving the same kind of results.  }


 \subsubsection{The general case}\label{sec_arbitrary-topology}{ 
 The convergence rate of SONATA over general graphs is summarized in Corollary~\ref{cor:linearization_rate}  for the linearization surrogates \eqref{eq:linear_surrogate} while   Corollaries~\ref{cor:f_rate_1}  and~\ref{cor:f_rate_2} consider  the surrogates  \eqref{eq:f_surrogate} based on local $f_i$, with Corollary~\ref{cor:f_rate_1} addressing the case $\beta\leq \mu$ and Corollary~\ref{cor:f_rate_2} the case $\beta>\mu$. The step-size $\alpha$ is tuned to   obtain favorable rate expressions.

 \begin{corollary}[Linearization surrogates]\label{cor:linearization_rate} 
In the setting of Theorem~\ref{thm:linear_rate}, let $\{\bx^\nu\}$ be the sequence generated by SONATA, using the surrogates   \eqref{eq:linear_surrogate} and   step-size   $\alpha = c\cdot \alpha_{\mx}$,  $c \in (0,1)$,    with $\alpha_{\mx} = \min\{1, (1 - \rho)^2/(\rho \cdot 110 \kappa_g (1 + \beta/L)^2)\}$. The number of iterations (communications) needed for $U(\bx_i^\nu)-U^\star\leq \epsilon$, $i\in [m]$,   is\vspace{-0.2cm}
\begin{align}
& \text{\bf Case I:} \quad && \mathcal{O} \left(\kappa_g\log (1/\epsilon)\right), & \text{if } \quad \frac{\rho}{(1 - \rho)^2}\leq \frac{1}{ 110 \,\kappa_g\, \left(1 + \frac{\beta}{L}\right)^2}, \label{eq:Case_I_linear}\\
& \text{\bf Case II:} \quad  && \mathcal{O} \left(\frac{ \big(\kappa_g + \beta/\mu\big)^2 \,\rho}{(1 - \rho)^2}\,\log (1/\epsilon)\right), & \text{otherwise}.  \label{eq:Case_II_linear}
\end{align}
\end{corollary}
\begin{proof}
See Appendix~\ref{app:linearization_rate}.
\end{proof}
 
\begin{corollary}[local $f_i$, $\beta\leq \mu$]\label{cor:f_rate_1}  
Instate assumptions of Theorem~\ref{thm:linear_rate} and suppose $\beta \leq \mu$. Consider   SONATA  using the surrogates~\eqref{eq:f_surrogate} and step-size $\alpha = c\cdot \alpha_{\mx}$,  $c \in (0,1)$, with $\alpha_{\mx} = \min\{1, (1 - \rho)^2/(M\rho) \}$ and  $M = 193\left(1 + \frac{\beta}{\mu}\right)^2\left(\kappa_g + \frac{\beta}{\mu}\right)^2 $. The number of iterations (communications) needed for $U(\bx_i^\nu)-U^\star\leq \epsilon$, $i\in [m]$,   is\vspace{-0.2cm}
\begin{align}
& \text{\bf Case I:} \quad && \mathcal{O} \left(1 \cdot \log (1/\epsilon) \right), & \text{if }\quad \frac{\rho}{(1 - \rho)^2}\leq \frac{1}{193\left(1 + \frac{\beta}{\mu}\right)^2\left(\kappa_g + \frac{\beta}{\mu}\right)^2},  \label{eq:Case_I_surrogate_beta_less_mu}\\
& \text{\bf Case II:} \quad  && \mathcal{O} \left( \frac{\kappa_g^2 \,\rho}{(1 - \rho)^2} \,\log (1/\epsilon)\right), & \text{otherwise}.\label{eq:Case_II_surrogate_beta_less_mu}
\end{align}
\end{corollary}

\begin{corollary}[local $f_i$, $\beta> \mu$]\label{cor:f_rate_2}  
Instate assumptions of Theorem~\ref{thm:linear_rate} and suppose $\beta > \mu$.  Consider   SONATA   using the surrogates~\eqref{eq:f_surrogate} and step-size    $\alpha = c\cdot \alpha_{\mx}$, $c \in (0,1)$, with $\alpha_{\mx} = \min\{1, (1 - \rho)^2/(M\rho) \}$ and  $M = 253\left(1 + \frac{L}{\beta}\right) \left(\kappa_g + \frac{\beta}{\mu}\right)$.  The number of iterations (communications) needed for $U(\bx_i^\nu)-U^\star\leq \epsilon$, $i\in [m]$,   is\vspace{-0.2cm}
\begin{align}
& \text{\bf Case I:}   && \mathcal{O} \left( \frac{\beta}{\mu}\cdot  \log (1/\epsilon) \right) & \text{if }  \, \frac{\rho}{(1 - \rho)^2}\leq \frac{1}{253\left(1 + \frac{L}{\beta}\right) \left(\kappa_g + \frac{\beta}{\mu}\right)},\label{eq:Case_I_surrogate_beta_big_mu} \\
& \text{\bf Case II:}   && \mathcal{O} \left( \frac{\left(\kappa_g + (\beta/\mu)\right)^2 \rho}{(1 - \rho)^2}\,\log (1/\epsilon) \right), & \text{otherwise }.\label{eq:Case_II_surrogate_beta_big_mu}
\end{align}
\end{corollary}
The proof of Corollaries~\ref{cor:f_rate_1} and~\ref{cor:f_rate_2}    can be found in Appendix~\ref{app:f_rate}.}

{  Several comments are in order.   

\noindent $\bullet$ {\bf Order of the  rate of centralized (nonaccelerated) methods (Case I):} For a fixed optimization problem, if the network is sufficiently connected ($\rho$ ``small''), its impact on the rate becomes negligible (the bottleneck is the optimization), and SONATA matches  the {\it network-independent} rate order achieved on   star-topologies (cf. Corollary \ref{cor:centralized_rate}) by the proximal gradient algorithm when linearization is employed [cf.~\eqref{eq:Case_I_linear}]  and by the mirror-descent scheme when the local $f_i$'s are used in the surrogates [cf.~\eqref{eq:Case_I_surrogate_beta_less_mu} and \eqref{eq:Case_I_surrogate_beta_big_mu}].  

\noindent $\bullet$ {\bf  Network-dependent rates (Case II):} As expected, the convergence rate deteriorates as $\rho$ increases, i.e., the network connectivity gets worse. This translates in  a less favorable    dependence of the complexity on $\kappa_g$ and $\beta/\mu$ (by a square factor) and  network scalability   of the order of $\rho/(1-\rho)^2$. When $\beta \sqrt{\rho} =\mathcal{O}(L)$ (e.g., the network is decently connected or $\beta=\mathcal{O}(L)$), the complexity becomes $\mathcal{O}\left(\kappa_g^2 (1 - \rho)^{-2} \log (1/\epsilon)\right)$, which compares favorably with that of existing distributed schemes, determined instead by the more pessimistic local quantities \eqref{eq:cond_number_others}. 
The scalability of the rate with the network connectivity, $(1-\rho)^{-2}$,  can be improved leveraging multiple rounds of communications  or accelerated consensus protocols, as discussed below.

  
\noindent $\bullet$ {\bf Linearization \eqref{eq:linear_surrogate} vs. local $f_i$  \eqref{eq:f_surrogate} surrogates:}  
As already observed in the setting of  star-networks, the use of the local losses as surrogates employs a form of preconditioning in the local agents subproblems. When the $f_i$'s are sufficiently similar to each other,  so that $1+\beta/\mu<\kappa_g$, exploiting local Hessian information via  \eqref{eq:f_surrogate} provably reduces the iteration/communication complexity over linear models \eqref{eq:linear_surrogate}--contrast (\ref{eq:Case_I_linear}) with  \eqref{eq:Case_I_surrogate_beta_less_mu} and  \eqref{eq:Case_I_surrogate_beta_big_mu}. Note that these faster rates are  achieved without exchanging any matrices over the network, which is a key   feature of SONATA.  On the other hand, when the functions $f_i$ are heterogeneous, the local surrogates \eqref{eq:f_surrogate} are no longer informative of the average-loss $F$ and using linearization might yield better rates. Although these design recommendations are based on  sufficient conditions,  numerical results seem to confirm  the above conclusions--see Sec.~\ref{sec:num}.  

\noindent $\bullet$ {\bf Multiple communications rounds and acceleration:} 
The discussion above shows that rates of the order of those of centralized methods can be achieved  if the network is sufficiently connected (Case I). When this is not the case, one can still achieve the same iteration complexity at the cost of multiple, finite, rounds of communications per iteration. 
Specifically, let $\rho_0$ be the connectivity  of the  given network and suppose we run $K$ steps  of communications per iteration (computation) in (\ref{eq:x_update})-(\ref{eq:y_update}); this yields an effective network with improved connectivity $\rho = \rho_0^K$. One can then choose $K$ so that the ratio $\rho_0^K/(1- \rho_0^K)^2$ satisfies the condition triggering     Case I in the Corollaries~\ref{cor:linearization_rate}--\ref{cor:f_rate_2}, as briefly summarized next.  

\textbf{1) Linearization:} Invoking Corollary~\ref{cor:linearization_rate}, one can check that the order of such a $K$ is  
  $K = \mathcal{O} (\log (\kappa_g(1+ \beta/L)^2)/\log (1/\rho_0)) = \mathcal{O} (\log (\kappa_g (1+ \beta/L)^2)/(1- \rho_0))$; therefore, SONATA using the surrogates   \eqref{eq:linear_surrogate} reaches an $\epsilon$-solution    in  $\mathcal{O} \left(\kappa_g\log (1/\epsilon)\right)$ iterations and $\mathcal{O}\left(\kappa_g\cdot (1 - \rho_0)^{-1}\log(\kappa_g (1 + \beta/L)^2)\log (1/\epsilon)\right)$ communications.  The dependence on the network connectivity $\rho_0$ can be further improved  leveraging   Chebyshev polynomials (see, e.g., \cite{auzinger2011iterative,pmlr-v70-scaman17a}): the final communication complexity of SONATA  reads \vspace{-0.5cm}
\begin{equation*}\mathcal{O}\left(\frac{\kappa_g}{\sqrt{1-\rho_0}}\cdot \log\left(\kappa_g(1+\beta/L)^{2}\right)\,\log(1/\epsilon)\right).\vspace{-0.1cm}\end{equation*}

{\bf 2) Local $f_i$ surrogates:} Considering the case $\beta\geq \mu$ (Corollary~\ref{cor:f_rate_2}), we can show that SONATA using the surrogates  \eqref{eq:f_surrogate} and employing multiple rounds of communications per iteration, reaches an  $\epsilon$-solution    in  $\mathcal{O} \left(\beta/\mu\cdot \log (1/\epsilon)\right)$ iterations and $\mathcal{O}\left(\beta/\mu\cdot\log\big((\kappa_g+ \beta/\mu)(1+L/\beta)\big) (1 - \rho_0)^{-1}\log (1/\epsilon)\right)$ communications. If Chebyshev  polynomials are used to accelerate the communications, the communication complexity further improves to \vspace{-0.2cm} 
\begin{equation*}
\mathcal{O}\left( \frac{{\beta}/{\mu}}{\sqrt{1 - \rho_0}}\cdot \log\big((\kappa_g+ \beta/\mu)(1+L/\beta)\big) \log (1/\epsilon)\right). \vspace{-0.2cm} \end{equation*}}

\section{The SONATA algorithm over    directed time-varying graphs}  \label{sec:TV-case}
 In this section we extend SONATA  and its convergence analysis  to solve   Problem~\eqref{eq:P}  over {\it directed,  time-varying  graphs} (Assumption \hyperlink{assumption:G}{B$\,^\prime$}). Note that \eqref{eq:loc_opt}-\eqref{eq:mix_y} is not readily   applicable to this   setting, as  constructing a doubly stochastic weight matrix   compliant with a directed graph  is generally infeasible or computationally costly--{see e.g. \cite{5530578}}.    
Conditions on the weight matrices can be relaxed if the consensus/tracking schemes \eqref{eq:mix_x}-\eqref{eq:mix_y} are properly changed to deal with the lack of doubly stochasticity.

 Here, we consider the perturbed push-sum protocols as proposed   in the companion paper \cite{SONATA-companion} (but in the Adapt-Then-Combine (ATC) form). The resulting distributed algorithm, still termed SONATA, is formally described in Algorithm~\ref{alg:SONATA_TV}. 
\vspace{-0.1cm}

	\begin{algorithm}[h]	
	\caption{SONATA over time-varying directed graphs}\label{alg:SONATA_TV}
	\textbf{Data}: $\mathbf{x}^{0}_{i}\in \mathcal{K}$,   $\by_i^0=\nabla f_i(\bx_i^0)$,  and $\phi_i^0=1$,  $i\in [m]$. 
	
	\textbf{Iterate}: $\nu=1,2,...$\vspace{0.1cm}
\begin{subequations}

  \texttt{[S.1] [Distributed Local Optimization]} Each agent $i$ solves\vspace{-0.2cm}
	\begin{equation} 	\hbx_i^\nu \triangleq {}   \argmin_{\bx_i \in \KK}~\tf_i (\bx_i ;\bx_i^\nu) + \big(\by_i^\nu - \nabla f_i (\bx_i^\nu) \big)^\top (\bx_i - \bx_i^\nu) + G(\bx_i),
	\label{eq:loc_opt2} 
	\vspace{-0.3cm}\end{equation} 
	
	\qquad and updates\vspace{-0.3cm} \begin{equation}\bx_i^{\nu+\frac{1}{2}} = {}   \bx_i^\nu + \alpha \cdot \Deltaxi{i}^\nu,\quad \text{with}\quad\Deltaxi{i}^\nu \triangleq \hbx_i^\nu - \bx_i^\nu; \label{eq:descent2}\end{equation}
	

	 \texttt{[S.2] [Information Mixing]} Each agent $i$ computes \smallskip\\
	 \phantom{\texttt{[S.2]}} (a) \texttt{Consensus}  \vspace{-0.3cm}
	\begin{equation}\phi_i^{\nu+1}=\sum_{j=1}^m c^\nu_{ij}\phi_j^\nu,\quad   
	\bx_i^{\nu + 1} = {}  \frac{1}{\phi_i^{\nu+1}} \sum_{j=1}^m  {c^\nu_{ij}\phi_j^\nu}  \bx_j^{\nu+\frac{1}{2}},
	\label{eq:mix_x2}\vspace{-0.3cm}\end{equation}  \phantom{\texttt{[S.2]}} (b) \texttt{Gradient tracking}\vspace{-0.2cm}
	\begin{equation}
	\by_i^{\nu + 1} =  \frac{1}{\phi_i^{\nu+1}} \sum_{j=1}^m  c^\nu_{ij}\left(\phi_j^\nu \,\by_j^\nu +  \nabla f_j(\bx_j^{\nu +1}) - \nabla f_j(\bx_j^\nu) \right),\label{eq:mix_y2}
\vspace{-0.4cm}
	 \end{equation}
	\textbf{end} \end{subequations}
\end{algorithm}

   In the perturbed push-sum protocols \eqref{eq:mix_x2}-\eqref{eq:mix_y2},   $\mathbf{C}^\nu\triangleq (c^\nu_{ij})_{i,j=1}^m$ satisfies the assumption below. 
\begin{assumption}
	\label{assumption:TVwights}
	For each $\nu\geq 0$, the weight matrix $\mathbf{C}^\nu \triangleq (c^\nu_{ij})_{i,j = 1}^m$ has a sparsity pattern compliant with $\GG^\nu$, i.e., there exists a constant $c_\ell$ such that, for all $\nu=0,1,\ldots,$
	\begin{enumerate}[leftmargin=*,label=\theassumption\arabic*]
		\item $c_{ii}^\nu\geq c_\ell >0$, for all $i\in [m]$;
		\item $c_{ij}^\nu\geq c_\ell>0$, if $\left(j,i\right)\in \mathcal{E}^{\nu}$; and $c_{ij}^\nu=0$ otherwise.
	\end{enumerate}
	Moreover, $\mathbf{C}^{\nu}$ is column stochastic, i.e., $\mathbf{1} ^\top\mathbf{C}^{\nu}= \mathbf{1}^{\top}$, for all $\nu=0,1,\ldots.$
\end{assumption}

We conclude this section stating    the counterparts of the definitions introduced in Sec.~\ref{sec:problem_statement},   adjusted here to  the case of directed time-varying graphs. 
Using the column stochasticity of $\mathbf{C}^\nu$ and~\eqref{eq:mix_y2},  one can see that opposed to \eqref{eq:y_update_avg}, the average gradient is  now preserved on the weighted average of the $\by_i$'s:\vspace{-0.1cm}
\begin{equation}\label{eq:avg_y_eq_TV}
\frac{1}{m}\sum_{i=1}^m \phi_i^{\nu +1} \by_i^{\nu + 1} = \frac{1}{m}\sum_{i=1}^m \phi_i^{\nu } \by_i^{\nu }  + \overline{\pgrad}^{\nu+1} - \overline{\pgrad}^{\nu}, \vspace{-0.1cm}
\end{equation}
where  $\overline{\pgrad}^\nu$ is defined in \eqref{eq:avg_y_def}.
This suggests to decompose $\by^\nu$ into its weighted average and the consensus error, defined respectively as\vspace{-0.3cm}
\begin{equation}
\wavg{\by}{\nu}\triangleq   \frac{1}{m}\sum_{i=1}^m  \phi_i^\nu\by_i^\nu \quad \text{and} \quad \var{y}{\nu} \triangleq \by^\nu - \mathbf{1}_m \otimes \wavg{\by}{\nu}.\vspace{-0.2cm}
\end{equation}
Accordingly, we   define the weighted average of $\bx^\nu$ and the consensus error as\vspace{-0.2cm}
\begin{equation}
\wavg{\bx}{\nu}\triangleq   \frac{1}{m}\sum_{i=1}^m  \phi_i^\nu\bx_i^\nu \quad \text{and} \quad \var{x}{\nu} \triangleq \bx^\nu - \mathbf{1}_m \otimes \wavg{\bx}{\nu}.\vspace{-0.2cm}
\end{equation}
In addition, we also generalize the definition of the optimality gap  as\vspace{-0.2cm}
\begin{equation}\label{eq:opt_gap_TV}
\optgap^\nu \triangleq \sum_{i=1}^m \phi_i^\nu p_i^\nu, \quad \text{with} \quad p_i^\nu \triangleq  \big(U(\bx_i^\nu) - U^\star \big).\vspace{-0.2cm}
\end{equation}

Finally, apart from the problem parameters $L_i$, $L_{\mx}$, $L$, $\mu$ [cf. \eqref{eq:prob_param_def}] and   algorithm parameters $\tmu_{\mn}$, $\tL_{\mx}$, $D_{\mn}^\ell$, $D_{\mx}$ [cf. \eqref{eq:alg_param_def}], we introduce the following network parameters, borrowed from \cite[Prop. 1]{SONATA-companion}: 
\begin{align}\label{eq:net_param_def}
\phi_{lb} \triangleq c_\ell^{2(m-1)B},  \quad 
\phi_{ub} \triangleq m - c_\ell^{2(m-1)B}, 
\end{align}
with $c_\ell$ and $B$ given in Assumptions~\ref{assumption:TVwights} and \hyperlink{assumption:G}{B$\,^\prime$}, respectively; and 
\begin{align}
c_0 \triangleq 2 m \cdot  \frac{1 + \tilde{c_\ell}^{-(m-1)B}}{1 - \tilde{c_\ell}^{-(m-1)B}}, \quad \rho_B  \triangleq (1 - \tilde{c_\ell}^{(m-1)B})^{\frac{1}{(m-1)B}}, \quad \tilde{c_\ell} \triangleq c_\ell^{2(m-1)B + 1}/m.
\end{align}
Furthermore, we will use the following lower and upper bounds of $\phi_i^\nu$ \cite[Prop. 1]{SONATA-companion} \vspace{-0.2cm} 
\begin{equation*} 
\phi_{lb}\leq \phi^\nu_i\leq \phi_{ub}, \quad \text{for all } i\in [m],\quad \nu=0,1,\ldots . 
\end{equation*}



\subsection{Linear convergence rate}\label{sec:linear_rate-directed}\!
The proof of linear convergence of SONATA (Algorithm~\ref{alg:SONATA_TV}) follows  the same path of the one developed in   Sec. \!\!\ref{sec:linear_rate} for the case of undirected graphs. Hence, we omit similar   derivations and  highlight only the key differences.    We will tacitly assume that    Assumptions~\ref{assump:p}, \hyperlink{assumption:G}{B$\,^\prime$}, \ref{assump:SCA_surrogate},  and  \ref{assumption:TVwights}  are satisfied. 

\subsubsection{Step 1: $\optgap^{\nu}$ converges linearly up $\mathcal{O} (\| \var{x}{\nu}\|^2 + \| \var{y}{\nu}\|^2)$}\label{sec:inexact_scaTV}
This is  counterpart of Proposition~\ref{prop:loc_geo_rate} (cf. Sec.~\ref{sec:linear_rate}), and stated as follows. 
{  \begin{proposition}\label{prop:loc_geo_rate2}
	The optimality gap sequence $\{\optgap^{\nu}\}$ satisfies:\vspace{-0.1cm}
		\begin{equation}\label{eq:opt-gap-descentTV}
		\optgap^{\nu + 1} \leq \sigma(\alpha) \cdot \optgap^\nu + \eta(\alpha) \cdot   \phi_{ub} \cdot \left(8 L_{\mx}^2 \| \var{x}{\nu}\|^2 + 2  \| \var{y}{\nu} \|^2\right),
		\end{equation}  where the constants $L_{\mx}$ and $\tmu_{\mn}$ are defined in~\eqref{eq:prob_param_def} and~\eqref{eq:alg_param_def}, respectively; and $\sigma(\alpha) \in (0,1)$ and $\eta(\alpha) >0$ are   defined in~\eqref{eq:sigma_C_def}.
\end{proposition}
\begin{proof} The proof follows closely   that of Proposition \ref{prop:loc_geo_rate}   and thus is omitted. For completeness, we report it in  the  supporting materials. 
Here, we only notice that, instead of    (\ref{eq:bound_doubly_stoc_cvx}), we built on:
	$\sum_{i=1}^m \phi_i^{\nu+1}	U(\bx_i^{\nu + 1}) \leq \sum_{i=1}^m \phi_i^{\nu}U\big(\bx_i^{\nu+\frac{1}{2}}\big)$, 
	where we used  $\sum_{j=1}^m {c^\nu_{ij}\phi_j^\nu}/{\phi_i^{\nu+1}}=1$, for all $i\in [m]$.%
\end{proof}
}

\subsubsection{Step 2: Decay of $\|\var{x}{\nu}\|$ and $\|\var{y}{\nu}\|$ }\label{sec:consensus_errTV} 

{ 

\begin{lemma}\label{lemma:consensus-decay} 
The following bounds hold for $\|\var{x}{\nu}\|$ and $\|\var{y}{\nu}\|$:\vspace{-0.3cm}

	\begin{subequations}\label{eq:consensus-boundsTV}
			\begin{align}
		\begin{split}\label{x_B_decay}
		\norm{\var{x}{\nu}}^2 \leq  2 c_0^2 \rho_B^{2 \nu} \norm{\var{x}{0}}^2 + \frac{2 c_0^2 \rho_B^2}{1-\rho_B} \sum_{t = 0}^{\nu - 1} \rho_B^{\nu -1 - t} \alpha^2 \norm{\Deltaxi{}^t}^2 
		\end{split}\\
		\begin{split}\label{y_B_decay}
		\norm{\var{y}{\nu}}^2 \leq  2 c_0^2 \rho_B^{2 \nu} \norm{\var{y}{0}}^2 + \frac{2 c_0^2 \rho_B^2 m  L_{\mx}^2 \phi_{lb}^{-2}}{1- \rho_B} \sum_{t=0}^{\nu -1}  \rho_B^{\nu -1 - t}   \left(8 \norm{\var{x}{t}}^2 + 2 \alpha^2 \norm{\Deltaxi{}^t}^2\right).
		\end{split}
		\end{align}
	\end{subequations}
	where $B$ and  $\rho_{B}$ 
	are   defined in~\eqref{eq:net_param_def}, and $\epsilon_x$ and $\epsilon_y$ are arbitrary positive constants (to be determined).

\end{lemma} 
\begin{proof}
Using the result in~\cite[Lemma 5]{nedic2010convergence} and~\cite[Lemma 3, 11]{SONATA-companion}, we obtain  
	\begin{align}
\begin{split}
\norm{\var{x}{\nu}} \leq  c_0 \left( \rho_B^{\nu}  \norm{\var{x}{0}} +   \sum_{t = 0}^{\nu - 1}  \rho_B^{(\nu - 1) - t} ( \rho_B \alpha \norm{\Deltaxi{}^t}) \right)
\end{split}\\
\begin{split}
\norm{\var{y}{\nu}} \leq c_0 \left( \rho_B^\nu \norm{\var{y}{0}} +  \sqrt{m} L_{\mx} \phi_{lb}^{-1} \sum_{t = 0}^{\nu - 1} \rho_B^{(\nu-1) - t} \cdot \rho_B \left(2 \norm{\var{x}{t}} + \alpha \norm{\Deltaxi{}^t}\right)\right).
\end{split}
\end{align}
The rest of the proof follows similar steps as \cite[Lemma 2]{Xu2015augmented}, hence it is omitted.  
\end{proof}
}

\subsubsection{Step 3:   $\| \Deltaxi{}^\nu\|=\mathcal{O}(\sqrt{\optgap^\nu}+\| \var{y}{\nu}\|)$   }\label{sec:delta_x_bound_directed}


	\begin{proposition}\label{prop:bound_delta_x_TV}
		The following upper bound holds for $\| \Deltaxi{}^\nu \|$:
{ 	\begin{equation}\label{eq:bound_delta_x_TV}
		\|\Deltax^\nu\|^2 
		\leq \frac{6}{\mu\phi_{lb}}\left(  \left( \frac{D_{\mx}}{\tmu_{\mn}} + 1\right)^2 + \frac{{4} L_{\mx}^2}{\tmu_{\mn}^2 } \right) \optgap^\nu  + \frac{3}{\tmu_{\mn}^2} \|\var{y}{\nu}\|^2,
		\end{equation}}
		where $L_{\mx}$, $\tL_{\mx}$, $\tmu_{\mn}$, and $D_{\mx}$ are defined in~\eqref{eq:prob_param_def} and \eqref{eq:alg_param_def}, respectively.
\end{proposition}

\begin{proof} The proof follows similar path of that of Proposition \ref{prop:bound_delta_x} and thus is omitted.  \end{proof}

\subsection{Establishing linear rate}\label{sec:linear-rate-TV}
{We can now prove linear rate following the path introduced in Sec. \ref{sec:linear_rate}; for sake of simplicity, we will use the same notation as in Sec. \ref{sec:linear_rate}}.  We begin applying the transformation   \eqref{eq:seq_transform} to the sequences $\{ \optgap^\nu \}_{\nu \in \natural_+}$,  $\{ \|\var{x}{\nu}\|^2 \}$, $\{ \| \var{y}{\nu}  \|^2 \}$, and $\{ \|\Deltaxi{}^\nu\|^2 \}$,   satisfying the inequalities   \eqref{eq:opt-gap-descentTV},  \eqref{x_B_decay}, \eqref{y_B_decay}, and (\ref{eq:bound_delta_x_TV}), respectively.  


\begin{proposition}\label{prop:small_gain_eqs_transformed2}
	Let $\seqnorm{P_{\bphi}}$, $\seqnorm{D}$,  $\seqnorm{X_{\bphi,\bot}}$, and $\seqnorm{Y_{\bphi,\bot}}$  denote the transformation \eqref{eq:seq_transform}  of the sequences $\{ \optgap^\nu \}$, $\{ \|\Deltaxi{}^\nu\|^2 \}$, $\{ \|\var{x}{\nu}\|^2 \}$ and $\{ \| \var{y}{\nu}  \|^2$ $ \}$. Given the constants $\sigma(\alpha)$ and $\eta(\alpha)$,    defined in Proposition~\ref{prop:loc_geo_rate2},  and the free parameters $\epsilon_x, \epsilon_y >0$, the following holds:
	\begin{align}
	\seqnorm{P_{\bphi}}  & \leq  G_{P}(\alpha,z)  \cdot \left( 8  \phi_{ub} L_{\mx}^2  \seqnorm{X_{\bphi,\bot}} +   2 \phi_{ub}  \seqnorm{Y_{\bphi,\bot}}\right) +  \omega_p
	\label{eq:bound_P2}
	\\
	\seqnorm{X_{\bphi,\bot}} & \leq   G_X(z) \cdot \rho_B^2 \alpha^2 D^K(z) + \omega_x
	\label{eq:bound_X2}
	\\
	\seqnorm{Y_{\bphi,\bot}} & \leq   G_Y(z) \cdot 2 m \phi_{lb}^{-2} L_{\mx}^2 \rho_B^2 \left(4\seqnorm{X_{\bphi,\bot}}+\alpha^2 D^K(z)\right)+ \omega_y\label{eq:bound_Y2}
	\\
	\seqnorm{D} &  \leq  C_1 \cdot \seqnorm{P_{\bphi}} + C_2 \cdot \seqnorm{Y_{\bphi,\bot}},
	\label{eq:bound_Dx2}
	\end{align}
	for all \vspace{-0.3cm}
	\begin{equation}\label{eq:bound_z2}
	z \in \left(\max\left\{ \sigma(\alpha), \rho_B\right\},1 \right), \vspace{-0.2cm}
	\end{equation}	where\vspace{-0.2cm}
	\begin{align}
	& G_{P}(\alpha,z) \triangleq \frac{\eta(\alpha)}{z - \sigma(\alpha)}, && \omega_p \triangleq \frac{z}{z - \sigma(\alpha)} \cdot \optgap^0
	\label{G_p}
	\\
	& G_{X}(z) \triangleq \frac{2 c_0^2}{(1 - \rho_B)(z - \rho_B)} ,&& \omega_x \triangleq 2 c_0^2  \norm{\var{x}{0}}^2\\
	&G_Y(z) \triangleq \frac{2 c_0^2}{(1 - \rho_B)(z - \rho_B)}, && \omega_y \triangleq 2 c_0^2  \norm{\var{y}{0}}^2
	\\
	&{  C_1 \triangleq  \frac{6}{\mu\phi_{lb}}\left(  \left( \frac{D_{\mx}}{\tmu_{\mn}} + 1\right)^2 + \frac{4 L_{\mx}^2}{\tmu_{\mn}^2 } \right),} && {  C_2 \triangleq \frac{4}{\tmu_{\mn}^2}}\label{eq:C1_C2_TV}.
	\end{align}
\end{proposition}
\begin{proof}\label{proof-small-gain-TV}
The proof of the first  two inequalities~\eqref{eq:bound_P2} and~\eqref{eq:bound_Dx2} follows the same steps of  those used to prove   Proposition \ref{prop:small_gain_eqs_transformed}.
Applying~\cite[Lemma 21]{tian2020achieving} to~\eqref{x_B_decay} and~\eqref{y_B_decay} respectively gives~\eqref{eq:bound_X2} and~\eqref{eq:bound_Y2}.

\end{proof}

Chaining the inequalities in~Proposition~\ref{prop:small_gain_eqs_transformed2} as done in   for (\ref{eq:all_ineq_transf})  (cf.~Fig.~\ref{fig:small_gain}), we can bound $\seqnorm{D}$  as\vspace{-0.2cm}
\begin{equation}\label{eq:small_gain_chain2}
\seqnorm{D} \leq \PP (\alpha ,z) \cdot \seqnorm{D} + \RR (\alpha ,z),
\end{equation}
where  $\PP (\alpha ,z)$ is defined as\vspace{-0.2cm}
\begin{align}\label{def_P_alpha_TV}
\begin{split}
\PP (\alpha ,z)  \triangleq {}& G_P(\alpha,z) \cdot G_X(z) \cdot C_1 \cdot 8 \phi_{ub}  L_{\mx}^2 \cdot \rho_B^2 \cdot \alpha^2 
\\
& + \left(G_P(\alpha,z) \cdot 2 \phi_{ub}  \cdot C_1 + C_2\right) \cdot G_Y(z) \cdot 2m \phi_{lb}^{-2} L_{\mx}^2 \cdot \rho_B^2 \cdot  \alpha^2 
\\
& + \left(G_P(\alpha,z) \cdot 2 \phi_{ub}  \cdot C_1 + C_2\right) \cdot G_Y(z) \cdot 8m\phi_{lb}^{-2} L_{\mx}^2 \cdot  G_X(z)  \cdot \rho_B^4 \cdot \alpha^2
\end{split}
\end{align}
and $\RR (\alpha ,z)$ is a bounded remainder term.

Comparing~\eqref{def_P_alpha_TV} to~\eqref{eq:stability_polynomial} we can see that they share the same form and only differ in coefficients. Therefore, with the same argument as in the proof of Theorem~\ref{thm:step_size_condition} we can easily arrive at the following conclusion.

\begin{theorem}\label{thm:step_size_condition2}
	Consider  Problem~\eqref{eq:P} under Assumptions \ref{assump:p}, and \hyperlink{assumption:G}{B$\,^\prime$}; and  SONATA (Algorithm~\ref{alg:SONATA_TV}) under Assumptions~\ref{assump:SCA_surrogate} and \ref{assumption:TVwights}, with   $\tmu_{\mn}\geq  D_{\mn}^\ell$.
 Then, there exists    a sufficiently small step-size $\bar{\alpha}\in (0,1]$ 
 such that, for all $\alpha < \bar{\alpha}$, $\{U(\bx_i^\nu)\}$ converges to $U^\star$  at an R-linear rate,   $i\in [m]$.
\end{theorem}
\begin{proof}
	We provide the proof in the  supporting material.
\end{proof}

{  For sake of completeness, we provide an explicit expression of the linear rates   in terms of the step-size $\alpha$ in the supporting material--see Theorem \ref{thm:linear_rate_TV}. Table~\ref{Table-SONATA-TV} summarizes the expression of the rates achieved by SONATA  using the surrogate functions \eqref{eq:linear_surrogate} and  \eqref{eq:f_surrogate}--a formal statement of these results along with the proofs can be found   in the supporting material-see  Corollaries \ref{cor:linearization_rate_TV},  \ref{cor:f_rate_1_TV} and \ref{cor:f_rate_2_TV}.   


	 \renewcommand{\arraystretch}{2}%
\begin{table}[h!]
\resizebox{\columnwidth}{!}{\begin{tabular}{|c|c|c|c|}
\hline 
{\bf Surrogate} & {\bf Communication Rounds} &  {\bf $\rho_B$ (network) }& $\beta$ \tabularnewline
\hline 
\hline 
\multirow{2}{3cm}{\centering linearization} & $\small\mathcal{O}\left(\kappa_g\,\log\left(1/\epsilon\right)\right)$ &  $\begin{array}{ccc}&\ensuremath{\rho_B=\mathcal{O}(\kappa_g^{-1}(1+\frac{\beta}{L})^{-2})} \vspace{-0.3cm}
\\ & \vspace{-0.3cm}
\text{or}
\\ 
\vspace{-0.3cm}& \text{star-networks}\vspace{0.2cm}
\\
\end{array}$  & arbitrary
\tabularnewline
\cline{2-3} 
& $\mathcal{O} \left(\frac{ \big(\kappa_g + \beta/\mu\big)^2 \rho_{B}}{(1 - \rho_{B})^2}\,\log (1/\epsilon)\right)$ &   arbitrary & \tabularnewline
\hline 
\multirow{4}{3cm}{\centering local $f_{i}$} & $\small\mathcal{O}\left(1\cdot\log\left(1/\epsilon\right)\right)$ &  $\begin{array}{ccc}&\rho_B=\mathcal{O}\left( {\left(1 + \frac{\beta}{\mu}\right)^{-2}\left(\kappa_g + \frac{\beta}{\mu}\right)^{-2}}\right)\vspace{-0.3cm}\\ & \vspace{-0.3cm}\text{or}\\ \vspace{-0.3cm}& \text{star-networks}\vspace{0.2cm}
\\
\end{array}$  & \multirow{2}{3cm}{\centering$\beta\leq\mu$}
\tabularnewline
\cline{2-3} 
& $\mathcal{O} \left( \frac{\kappa_g^2 \rho_B}{(1 - \rho_B)^2} \,\log (1/\epsilon)\right)$ &   arbitrary & \tabularnewline
\cline{2-4} 
 & $\small\mathcal{O}\left(\dfrac{\beta}{\mu}\cdot\log\left(1/\epsilon\right)\right)$ &  $\begin{array}{ccc}&\rho_B=\mathcal{O}\left(\left(1 + \frac{L}{\beta}\right)^{-1} \left(\kappa_g + \frac{\beta}{\mu}\right)^{-1}\right)\vspace{-0.3cm}\\ & \vspace{-0.3cm}\text{or}\\ \vspace{-0.3cm}& \text{star-networks}\vspace{0.2cm}\\\end{array}$ & \multirow{2}{3cm}{\centering $\beta>\mu$}
 \tabularnewline
 \cline{2-3} 
 &$\mathcal{O} \left(\frac{ \big(\kappa_g + \beta/\mu\big)^2 \rho_{B}}{(1 - \rho_{B})^2}\,\log (1/\epsilon)\right)$  &   arbitrary & \tabularnewline
\hline  
\end{tabular}}\medskip \caption{{  Summary of convergence rates of SONATA over time-varying directed graphs: number of communication rounds to reach   $\epsilon$-accuracy. }
\vspace{-0.7cm}}
\label{Table-SONATA-TV}\end{table} 
	


The rate estimates in Table~\ref{Table-SONATA-TV} are almost identical to those  obtained in Sec.  \ref{sec_arbitrary-topology}, with the difference that the network dependence now is expressed throughout    $\rho_B$ rather than  $\rho$.  Therefore, similar comments--as those stated in Sec. \ref{sec_arbitrary-topology}--apply to the rates in Table~\ref{Table-SONATA-TV}. For example, if the network is sufficiently connected ($\rho_B$ ``small''), its impact on the rate becomes negligible and SONATA matches  the {\it network-independent} rate achieved on   star-topology (cf. Corollary \ref{cor:centralized_rate}) or  centralized settings. Specifically,     when linearization surrogate  \eqref{eq:linear_surrogate} is used,  this rate  coincides with the rates of centralized proximal gradient algorithm.}

{  
\section{Numerical Results}\label{sec:num} In this section, we corroborate numerically  the complexity results proved in Corollaries~\ref{cor:linearization_rate}--\ref{cor:f_rate_2}. 
As  a test problem, we consider the distributed ridge regression: \vspace{-0.2cm}
\begin{equation}~\label{p:ridge_regression}
\min_{\bx \in \mathbb{R}^d}~\frac{1}{m} \left\{\frac{1}{2n} \| \bA_i \bx - \bb_i\|^2 + \lambda \|\bx\|^2\right\},\vspace{-0.1cm}
\end{equation}
where the loss function of agent  $i$ is   $f_i (\bx)= \frac{1}{2n} \| \bA_i \bx - \bb_i\|^2 + \lambda \|\bx\|^2$ [agent $i$ owns data $(\bA_i, \bb_i)$]. Problem parameters are generated as follows.
Each row of the measurement matrix $\bA_i$ is independently and identically drawn from distribution $\mathcal{N} (\0, \boldsymbol{\Sigma})$; and $\bb_i$ is generated according to the linear model $\bb_i = \bA_i \bx^* + \bn_i$, where $\bx^*$ is the ground truth,   generated according to   $\mathcal{N} (5 \cdot \1, \bI)$, and $\bn_i \sim \mathcal{N} (\0, 0.1 \cdot \bI)$ is the   measurement noise.  The covariance matrix $\boldsymbol{\Sigma}$ is constructed according to the eigenvalue decomposition $\boldsymbol{\Sigma} =  \sum_{j = 1}^{d} \lambda_j \bu_j \bu_j^\top$, where the eigenvalues $\{\lambda_j\}_{j = 1}^d$ are uniformly distributed in $[\mu_0, L_0]$. The eigenvectors, forming   $\mathbf{U} = [\bu_1, \ldots, \bu_d]$, are obtained via the QR decomposition of a random $d \times d$ matrix with standard Gaussian i.i.d. elements. The network is generated using an Erd\H{o}s-R\'{e}nyi model $G(m,p)$, with $m = 30$ nodes and each edge independently  included in the graph with probability $p = 0.5$.

To investigate the impact of $\kappa_g$ and $\beta$ on the convergence rate, we specifically  consider the following two scenarios:
\begin{itemize}
\item[(S.I) ] {\bf Changing $\kappa_g$ with fixed $\beta$}:  We generate a sequence of instances of~\eqref{p:ridge_regression} with  fixed $\beta$ and increasing $\kappa_g$. To do so, we  use the same   data set  $\{\bA_i ,\bb_i\}$  across  the  different instances and  change  the regularization parameter $\lambda$, so that the condition number $\kappa_g$ ranges in $[K_\ell,  K_u]$.

\item[(S.II)] {\bf Changing $\beta$ with (almost) fixed  $\kappa_g$}:  We generate  instances of~\eqref{p:ridge_regression} with decreasing $\beta$ and (almost) fixed $\kappa_g$. To do so, we set $\lambda = 0$ and   increased the local sample size $n$ from $N_\ell$ to $N_u$;  we set $N_\ell$ sufficiently large so that the empirical condition number $\kappa_g$ is close to $L_0/\mu_0$ for  all  instances.
\end{itemize}

We run SONATA using surrogates  {\eqref{eq:linear_surrogate} (linearization) and~\eqref{eq:f_surrogate} (local $f_i$)--we term it as    SONATA-L and   SONATA-F, respectively.
	The simulations parameters of the different experiments are summarized in  Table~\ref{tab:sim_params}; and the algorithmic parameters are set according to Corollaries~\ref{cor:linearization_rate}--
	\ref{cor:f_rate_2}.\footnote{\label{param} The expressions are not tight in terms of the absolute constants. To show convergence rate in both Cases I and II in Corollary~\ref{cor:linearization_rate}-\ref{cor:f_rate_2}, we enlarged the   second term in the expression of $\alpha_{\mx}$ by a constant factor.} 
	We measure the algorithm's complexity using 
	$
	T_\epsilon = \inf\left\{ \nu\geq 0 \,|\, \frac{1}{m} \sum_{i=1}^{m}(F(\bx_i^\nu) - F^\star)\right.$ $\left.\leq 10^{-7}\right\}
	$. 
	
	 In Table~\ref{tab:iter_complexity_figs}, we report the corresponding  iteration complexity of SONATA for each  simulation setup (s.1)-(s.6) in   Table~\ref{tab:sim_params}. Each figure is generated under one particular realization of the problem setting. Further, in order to compare the complexity of SONATA across different settings, all the simulations share the same network parameters, as well as the same data set whenever the problem parameters are the same. 
	 The results of our experiments are reported in 
	  Table~\ref{tab:iter_complexity_figs}; the curve  are generated using only one random realization for visualization clarity. However, the behavior of the curves (e.g., scalability with respect to the parameters) is representative and  consistent across all the random experiments we conducted. \begin{table}[t!]\label{tab:sim_params}
	\centering
	\begin{tabular}{ c | c | c}
		\hline
						 & Setting (S.I) & Setting (S.II) \\\hline
	Linearization 	 & \makecell{(s.1) \\[.5ex]$n = 10^3$ \\ $\mu_0 = 1$, $L_0 = 10^3$ \\ $K_\ell = 10$, $K_u = 100$} &  
	\makecell{(s.4) \\[.5ex] $\lambda = 0$ \\ $\mu_0 = 1$, $L_0 = 5$, $\kappa_g \approx 5$ \\ $N_\ell = 10$, $N_u = 10^3$} \\\hline
	Local $f_i$ ($\beta \geq \mu$)	& \makecell{(s.2) \\[.5ex] same as above}
			&	\makecell{(s.5) \\[.5ex] same as above} \\\hline			
	Local $f_i$ ($\beta < \mu$)	&  \makecell{(s.3) \\[.5ex] $n = 10^5$ \\ $\mu_0 = 1$, $L_0 = 20$ \\ $K_\ell = 1.1$, $K_u = 19$} &  
	\makecell{(s.6) \\[.5ex] $\lambda = 0$ \\ $\mu_0 = 1$, $L_0 = 2$, $\kappa_g \approx 2$ \\ $N_\ell = 2 \times 10^3$, $N_u = 10^5$ } \\\hline
	\end{tabular}\medskip \caption{Simulation setup and parameter setting.} \vspace{-0.8cm}
\end{table}

	  The following comments are in order.

	 \noindent $\bullet$ \textbf{Scalability with respect to $\kappa_g$.} Consider setting (S.I) wherein $\beta$ is fixed and $\lambda$ is changing. Figures for (s.1)-(s.3) show that when $\alpha = 1$ (blue curve), the iteration complexity of SONATA-L scales linearly with respect to $\kappa_g$ [as predicted by Corollary~\ref{cor:linearization_rate}], while that of SONATA-F is invariant whenever $\beta < \mu$ [as stated in Corollary~\ref{cor:f_rate_1}]. When $\beta \geq \mu$,  the iteration complexity of SONATA-F grows as $\lambda$ increases since $\beta/\mu$ decreases [cf. Corollary~\ref{cor:f_rate_2}]. However, the increasing rate is much slower than SONATA-L, due to the fact that  $(\beta/\mu)/\kappa_g = \beta/L \ll 1$ for large $\lambda$. When $\alpha <1$, the iteration complexity scales quadratically with respect to $\kappa_g$, in all settings, as predicted by our theory.

\noindent $\bullet$ \textbf{Scalability with respect to $\beta$.} Consider now  setting (S.II), where  we decrease the local sample size $n$ to increase $\beta$. 
In contrast to setting (S.I), Figures for (s.4) and (s.5) show that, with $\alpha = 1$, the iteration complexity of SONATA-F scales linearly with $\beta/\mu$ when $\beta >\mu$, while that of SONATA-L is invariant--this is consistent with Corollaries~\ref{cor:linearization_rate} and~\ref{cor:f_rate_2}. When $\alpha < 1$, the iteration complexity scales quadratically with respect to $\beta/\mu$. Finally, the plot associated with  (s.6) simply reveals that when $\beta < \mu$, iteration complexity of SONATA-F remains bounded, as stated in Corollary~\ref{cor:f_rate_1}.

\noindent $\bullet$ \textbf{Linearization versus Local $f_i$.} We compare the performance of SONATA-L and SONATA-F in the setting (S.II), with  parameters  $\lambda = 0$, $\mu_0 = 1$, $L_0 = 100$, $N_\ell = 10$, $N_u = 10^5$. We consider a relatively connected network with edge activation probability $p = 0.9$ so that the step-size can be set to $\alpha = 1$, for all experiments. Note that such connectivity can also be achieved with a less connected network by running multiple but fixed rounds of   consensus steps. Fig.~\ref{fig:lin_vs_surrogate} compares the iteration complexity as $\beta$ increases, averaged over $100$ Monte-Carlo realizations. We can see that for small $\beta$ SONATA-F converges faster than SONATA-L; while for large $\beta$ SONATA-L is faster. This can be explained using our results in Corollaries~\ref{cor:linearization_rate} and~\ref{cor:f_rate_2}. As the complexity of SONATA-F and SONATA-L scales   proportionally to $\beta/\mu$ and $\kappa_g$, respectively, when $\beta/\mu$ is comparatively  smaller  than $\kappa_g$, SONATA-F enjoys a better rate. But as $\beta/\mu$ increases, the rate deteriorates and eventually gets worse than that of SONATA-L.

}

\begin{table}[t]\label{tab:iter_complexity_figs}
	\centering
	\begin{tabular}{ c | c | c}
		\hline
		& Setting (S.I) & Setting (S.II) \\\hline
		Linearization 	 & 		\makecell{(s.1)\\	\includegraphics[width = 4cm, height = 2cm  ]{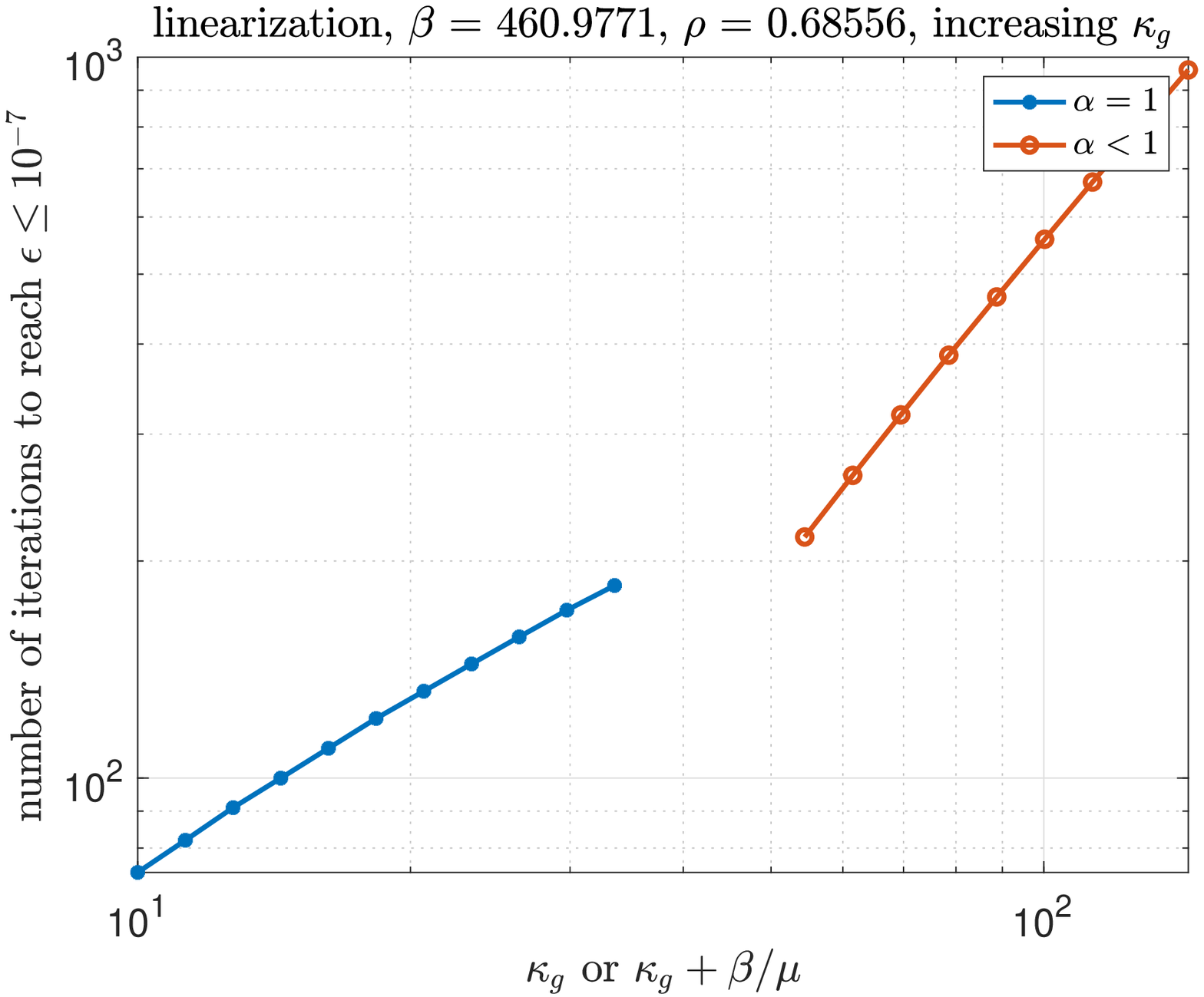}	\\ \includegraphics[width = 4cm, height = 2cm ]{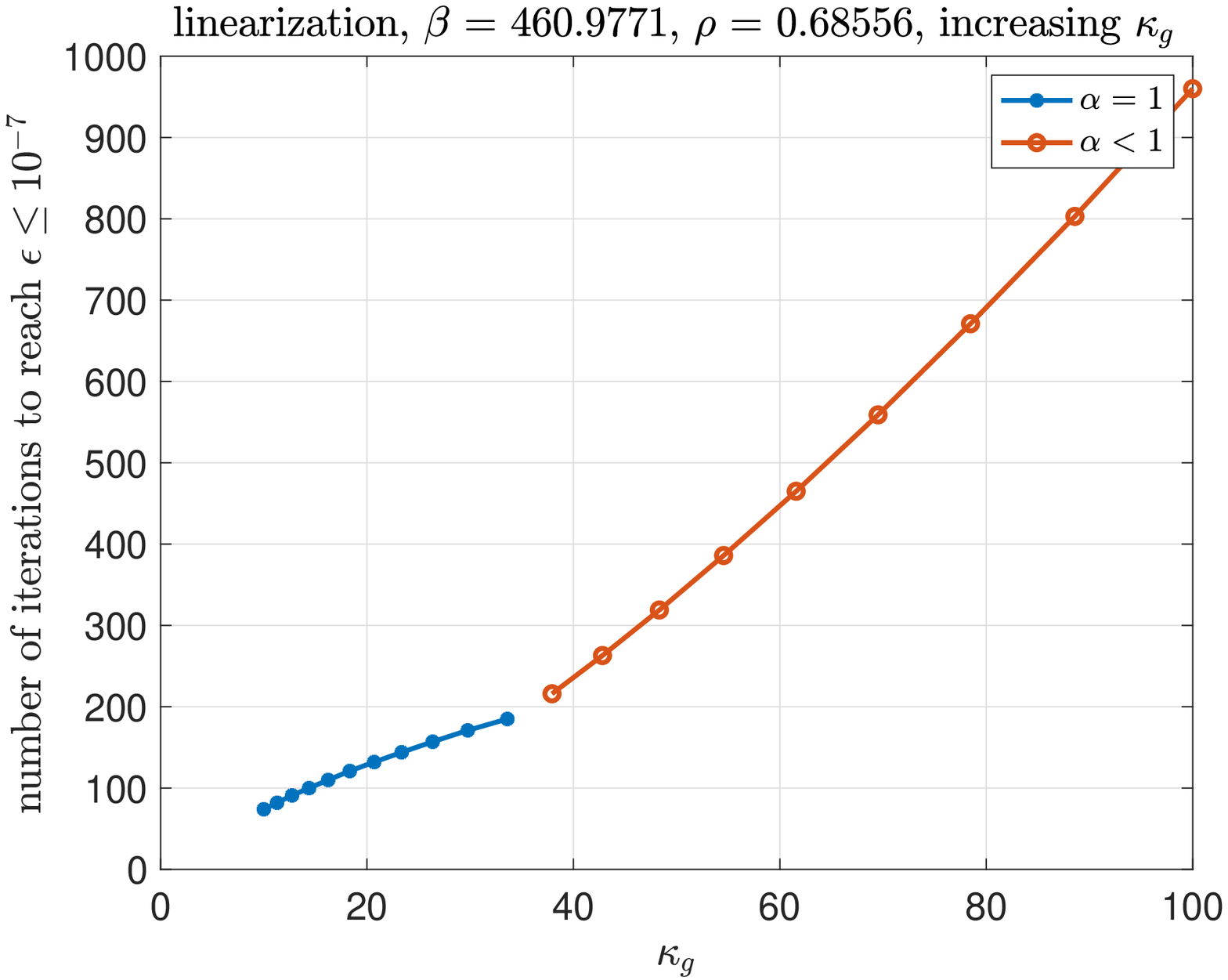}} &  \makecell{(s.4) \\ \includegraphics[width = 4cm, height = 2cm ]{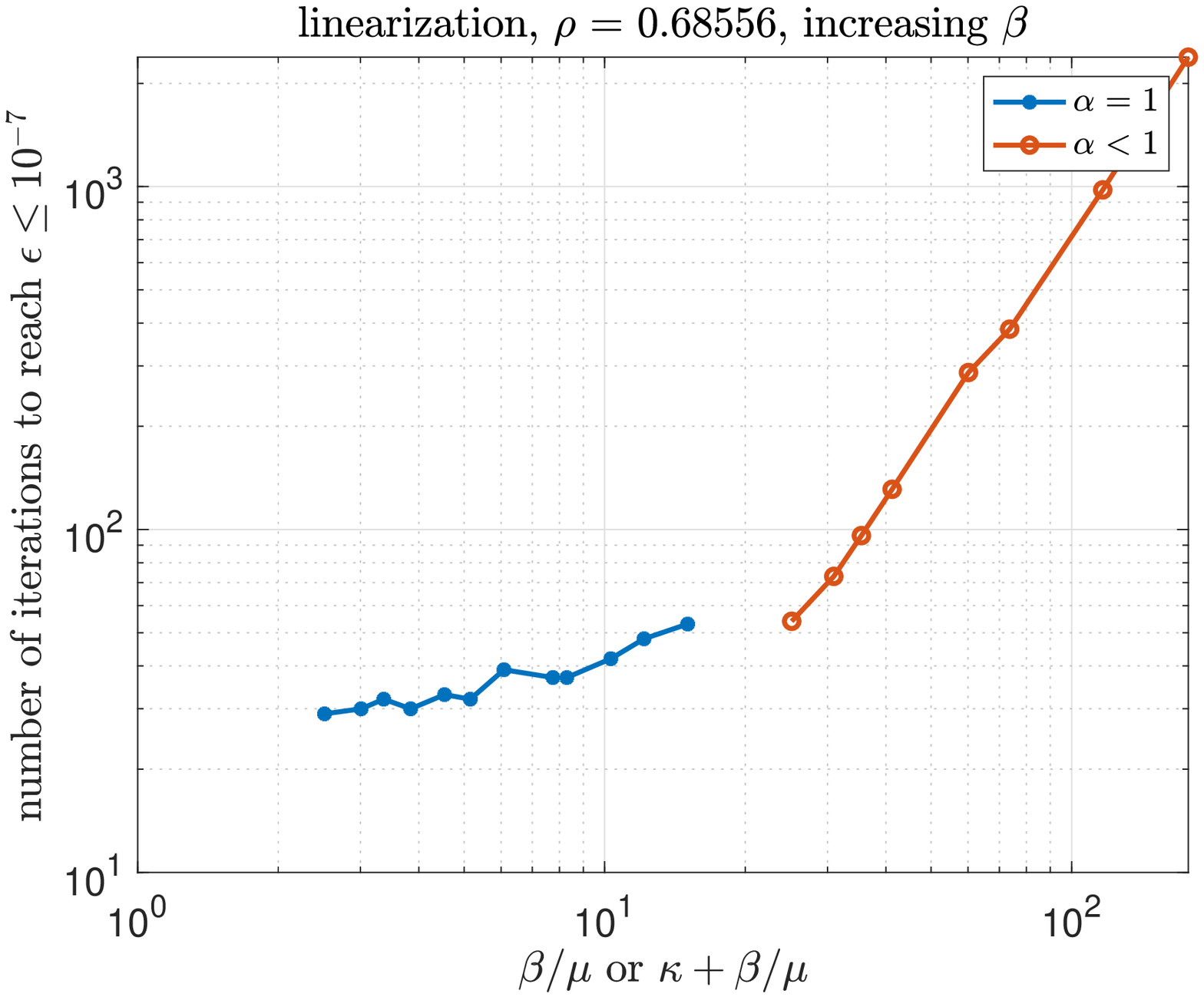} \\ \includegraphics[width = 4cm, height = 2cm ]{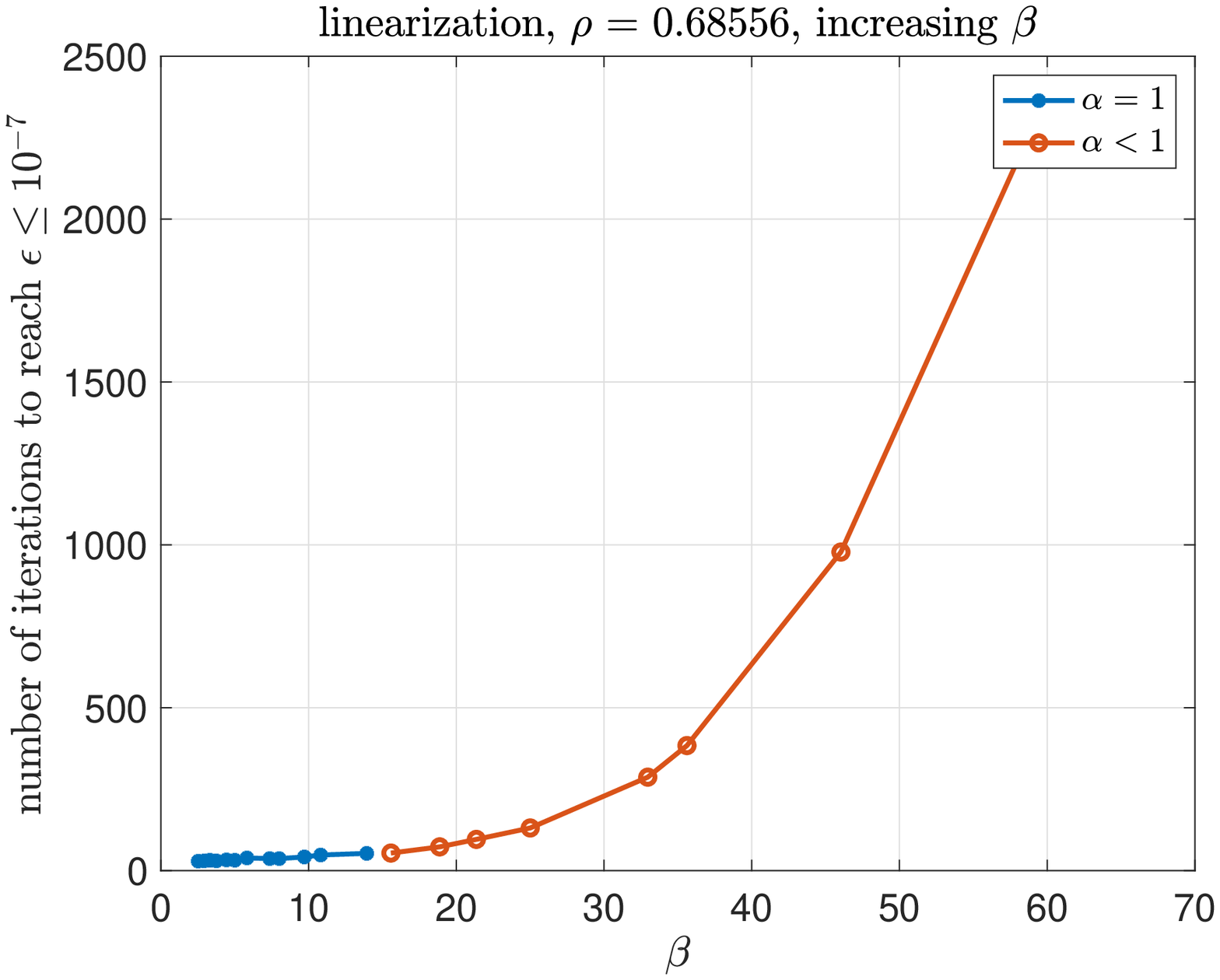} } \\\hline
		Local $f_i$ ($\beta \geq \mu$)	& \makecell{(s.2)\\ \includegraphics[width = 4cm, height = 2cm]{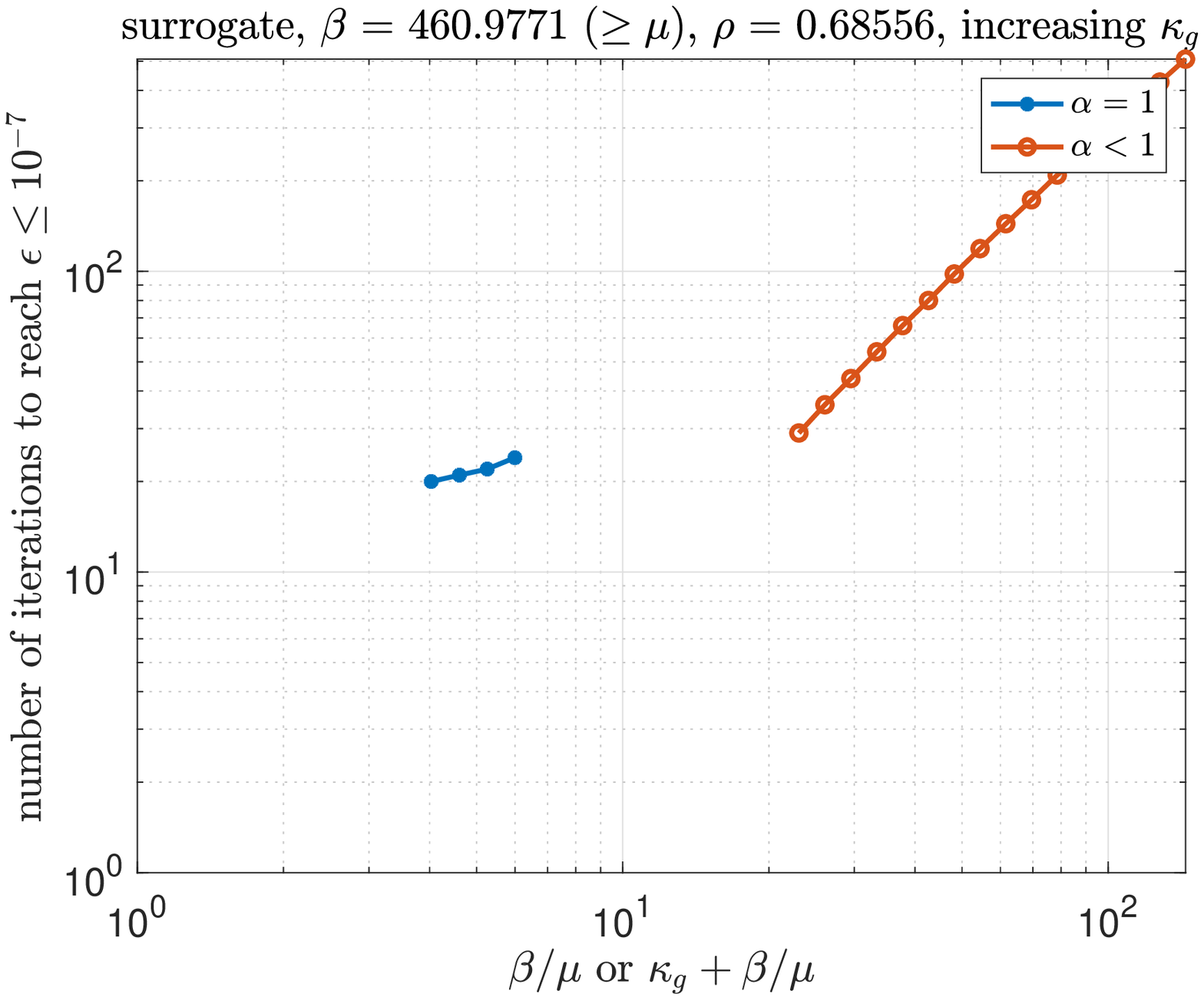} \\  \includegraphics[width = 4cm, height = 2cm]{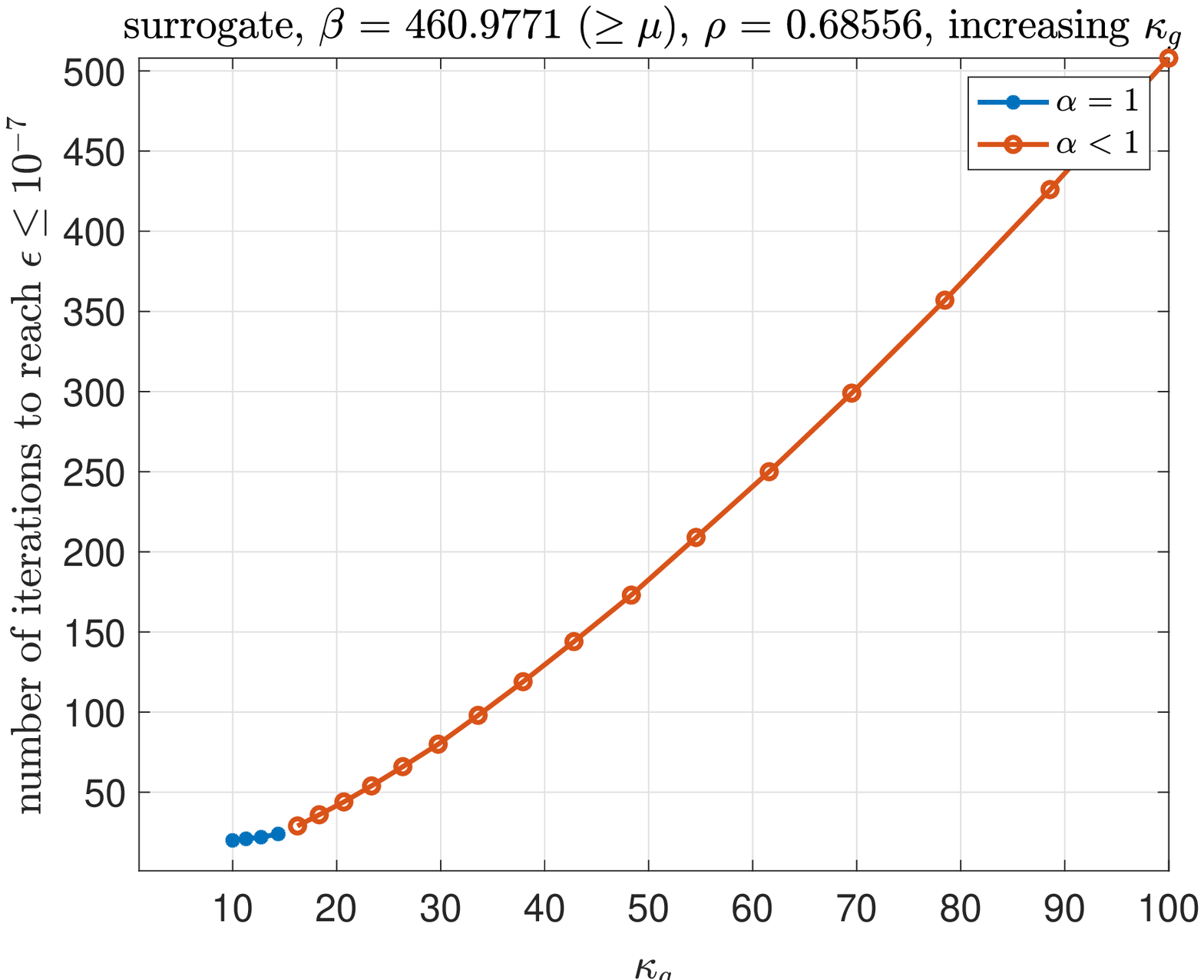}}
		& \makecell{(s.5)\\ \includegraphics[width = 4cm, height = 2cm]{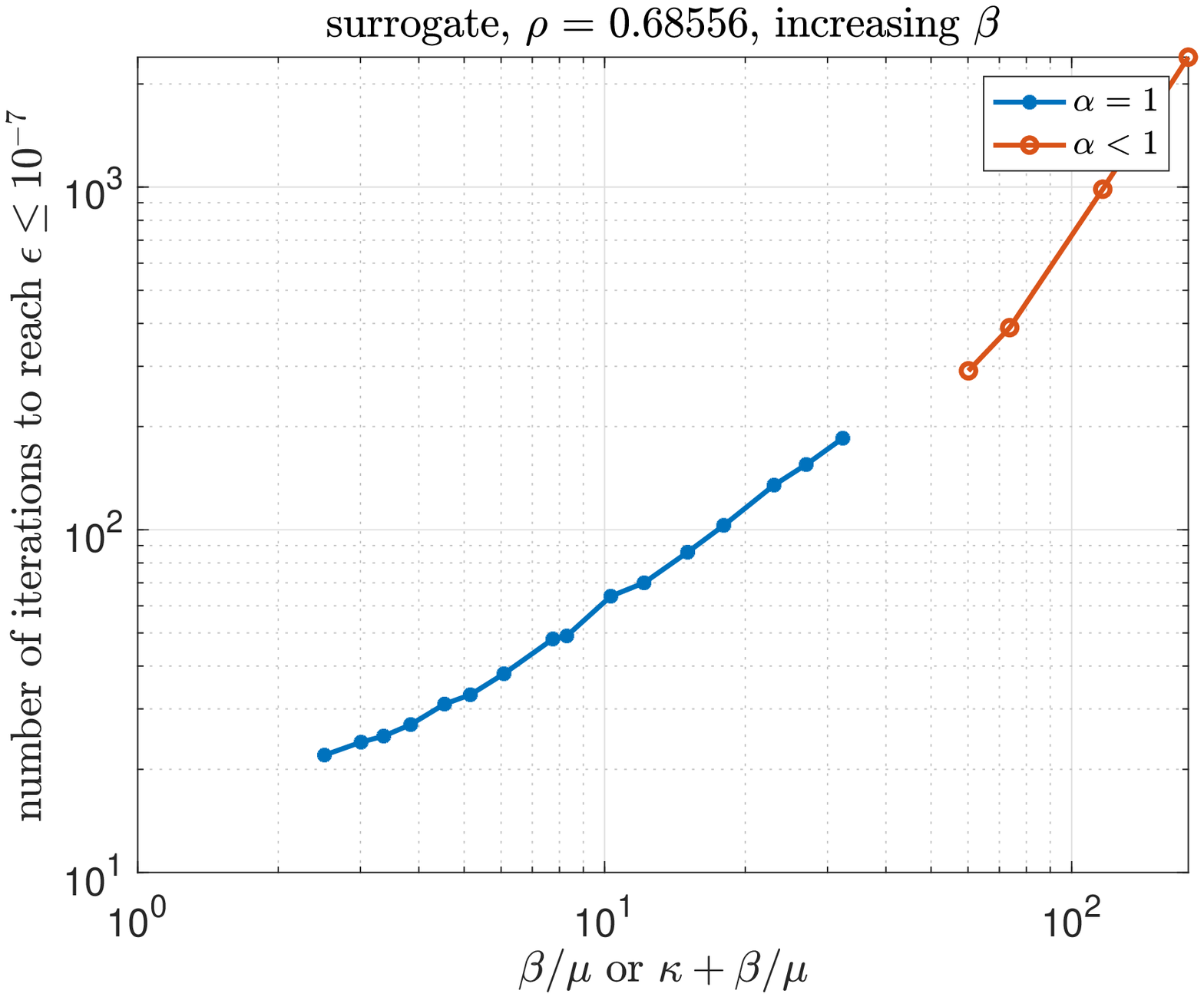} \\ \includegraphics[width = 4cm, height = 2cm]{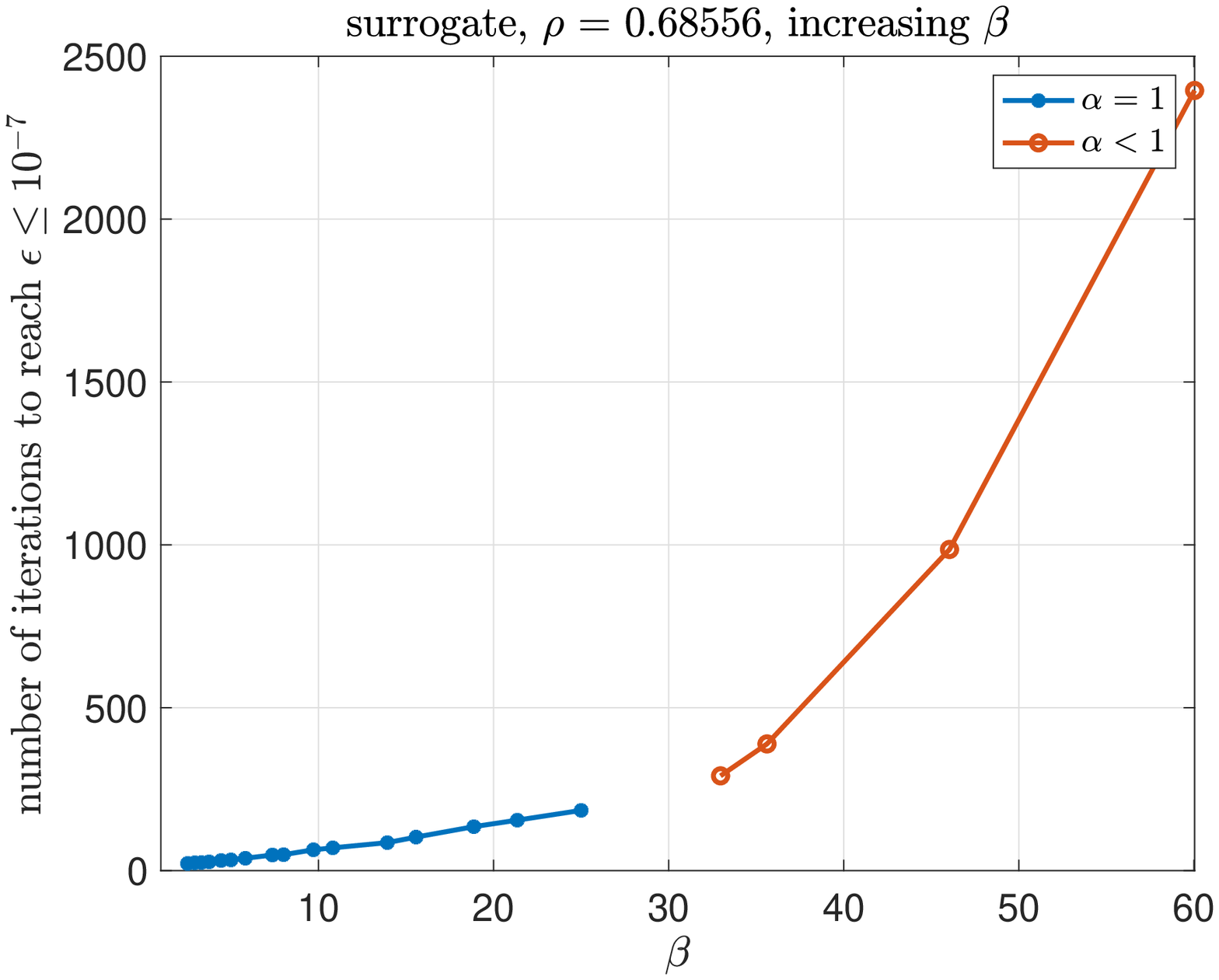}} \\\hline
		Local $f_i$ ($\beta < \mu$)	&  \makecell{(s.3) \\ \includegraphics[scale = .23]{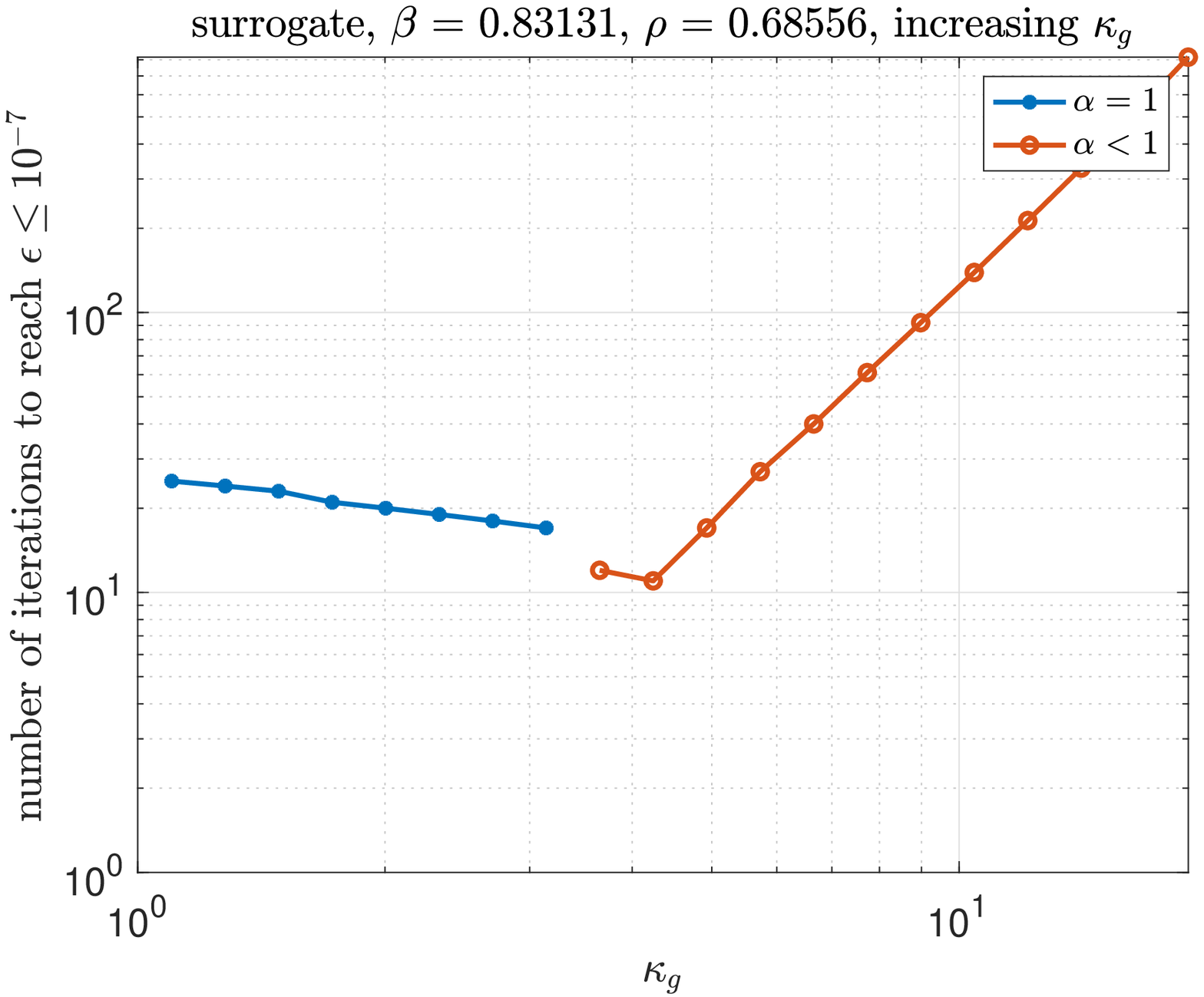}} &  \makecell{(s.6) \\\includegraphics[scale = .23]{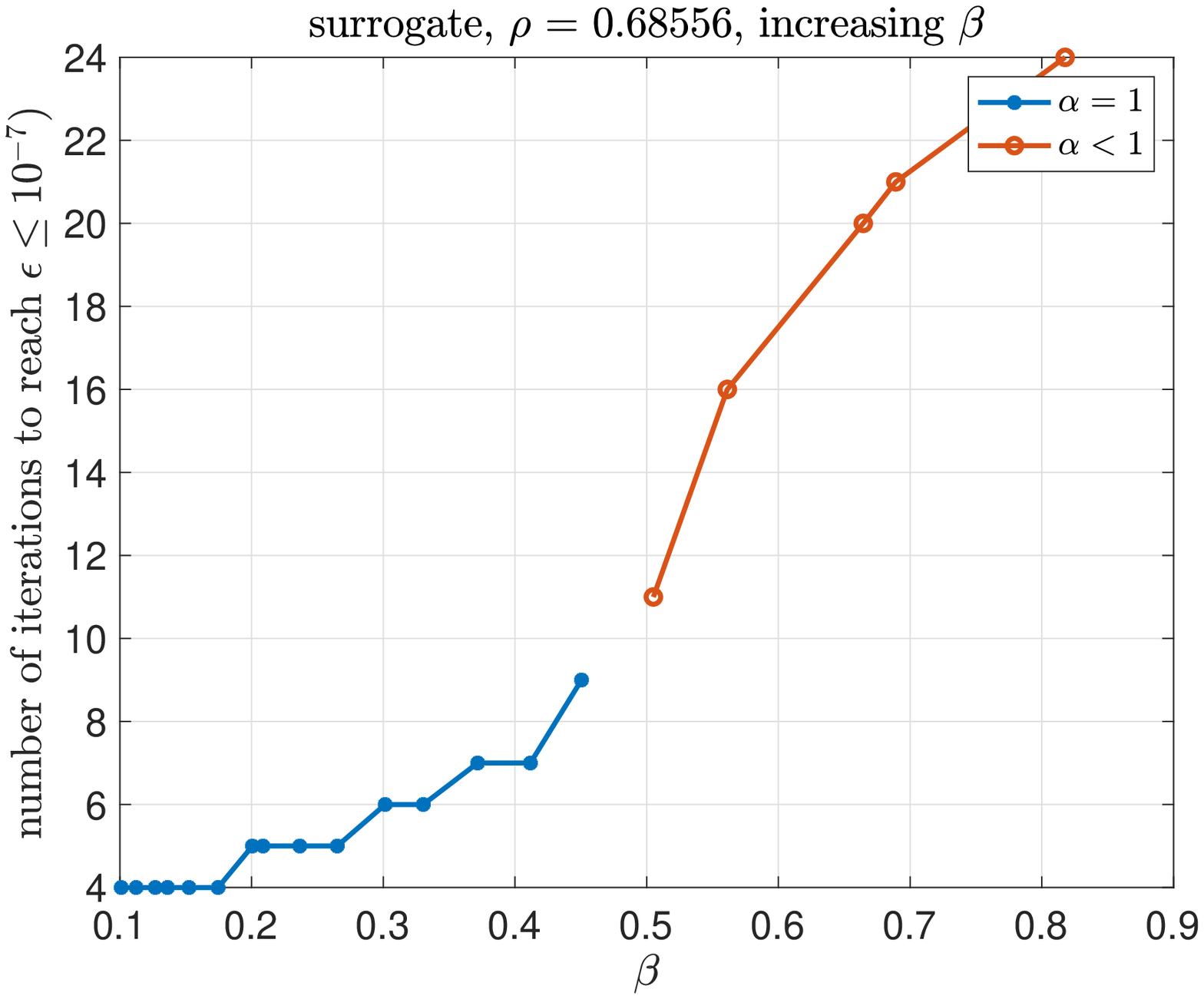}}
		 \\\hline
	\end{tabular}\medskip\caption{{ Iteration complexity of SONATA under the simulation settings in Table~\ref{tab:sim_params}. Left (S.I): scalability of iteration complexity with respect to the condition number $\kappa_g$; Right (S.II): scalability of the iteration complexity with respect to the similarity parameter $\beta$.}}
\end{table}
\begin{figure}[h!]\label{fig:lin_vs_surrogate}
	\centering
\includegraphics[scale = .35]{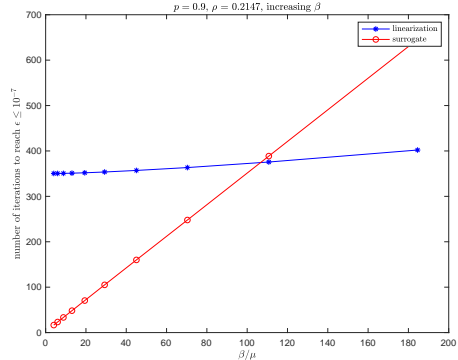}\vspace{-0.6cm}\caption{Complexity of SONATA-L versus SONATA-F.}
\end{figure}
\appendix

\section{Proof of~\eqref{eq:small_gain_chain}} \label{app:pf_small_gain_chain}
Chaining the  inequalities in~\eqref{eq:all_ineq_transf}  as   shown in Fig.~\ref{fig:small_gain}, we have\vspace{-0.2cm} 
\begin{align*}
D^K (z)  
&  {\leq} C_1 \cdot \seqnorm{P} + C_2 \cdot \seqnorm{Y_\bot}\\
& \leq C_1 \cdot \Big(G_P(\alpha,z) \cdot \left( 4   L_{\mx}^2 \seqnorm{X_\bot} +  2  \seqnorm{Y_\bot}\right) + \omega_p \Big) + C_2 \cdot \seqnorm{Y_\bot}\\
& = C_1 \cdot G_P(\alpha,z) \cdot 4   L_{\mx}^2 \seqnorm{X_\bot} + (C_1 \cdot G_P(\alpha,z) \cdot 2 + C_2) \seqnorm{Y_\bot}+ C_1\cdot \omega_p\\
& \leq C_1 \cdot G_P(\alpha,z) \cdot 4  L_{\mx}^2 \cdot G_X (z)  \cdot \rho^2 \alpha^2 \seqnorm{D}   \\
& \quad + (C_1 \cdot G_P(\alpha,z) \cdot 2  + C_2) \cdot  G_Y(z) \cdot 8L_{\mx}^2 \rho^2  \seqnorm{X_\bot}  \\
& \quad + (C_1 \cdot G_P(\alpha,z) \cdot 2  + C_2) \cdot  G_Y(z) \cdot 2L_{\mx}^2 \rho^2  \alpha^2 \seqnorm{D}  \\
& \quad + C_1\cdot \omega_p + (C_1 \cdot G_P(\alpha,z) \cdot 2  + C_2) \cdot \omega_y + C_1 \cdot G_P(\alpha,z) \cdot 4   L_{\mx}^2 \cdot \omega_x \\
& \leq C_1 \cdot G_P(\alpha,z) \cdot 4   L_{\mx}^2 \cdot G_X (z)  \cdot \rho^2 \alpha^2 \seqnorm{D}   \\
& \quad + (C_1 \cdot G_P(\alpha,z) \cdot 2  + C_2) \cdot  G_Y(z) \cdot 8L_{\mx}^2 \rho^2  \cdot G_X (z)  \cdot \rho^2 \alpha^2 \seqnorm{D}  \\
& \quad + (C_1 \cdot G_P(\alpha,z) \cdot 2  + C_2) \cdot  G_Y(z) \cdot 2L_{\mx}^2 \rho^2  \alpha^2 \seqnorm{D}  \\
& \quad + C_1\cdot \omega_p + (C_1 \cdot G_P(\alpha,z) \cdot 2  + C_2) \cdot \omega_y + C_1 \cdot G_P(\alpha,z) \cdot 4   L_{\mx}^2 \cdot \omega_x \\
& \quad + (C_1 \cdot G_P(\alpha,z) \cdot 2  + C_2) \cdot  G_Y(z) \cdot 8L_{\mx}^2 \rho^2 \cdot \omega_x.
\end{align*}
Notice that, under  \eqref{eq:bound_z},   $G_P(\alpha,z)$, $G_X(z)$, $G_Y(z)$,  and  $\omega_p$, $\omega_x$, $\omega_y$ are all bounded, which implies that the reminder $\mathcal{R}(\alpha,z)$ in \eqref{eq:all_ineq_transf} is bounded as well.$\hfill \square$ 

\section{Proof of Theorem~\ref{thm:linear_rate}}\label{app:pf_linear_rate}
We find the smallest $z$ satisfying (\ref{eq:bound_z}) such  that $\PP(\alpha,z) < 1$, for   {$\alpha\in (0, \alpha_{\mx})$}, with $\alpha_{\mx}\in (0,1)$ to be determined. 

Let us begin considering the condition $z>\sigma(\alpha)$ in (\ref{eq:bound_z}). To simplify the analysis, we impose instead the following stronger version  \vspace{-0.2cm}
\begin{align}\label{eq:rate_condition_1}
z \geq   \sigma(\alpha) +   \frac{ (\theta \cdot \alpha) \cdot \left( \left(1 - \frac{\alpha}{2}\right)\tmu_{\mn} + \frac{D_{\mn}^\ell}{2} \alpha - \frac{1}{2}  \epsilon_{opt} \right) }{\frac{ D_{\mx}^2}{\mu} +  \left(1 - \frac{\alpha}{2}\right)\tmu_{\mn} + \frac{D_{\mn}^\ell}{2} \alpha - \frac{1}{2}  \epsilon_{opt} }
\end{align}
for some $\theta \in (0,1)$, which will be chosen to tighten the bound. Notice that    the RHS of (\ref{eq:rate_condition_1})  is   strictly larger than  $\sigma(\alpha)$ but still strictly less than one, for any   $\alpha\in (0, (2 \tmu_{\mn} -\epsilon_{opt} ) / (\tmu_{\mn} - D_{\mn}^\ell))$, with given  $\epsilon_{opt}\in (0, 2 \tmu_{\mn})$.  

Observe that in the expression of $\PP(\alpha,z)$, the only coefficient multiplying $\alpha^2$ that depends on  $\alpha$ is the optimization gain
$G_P(\alpha,z) \triangleq  {\eta (\alpha)}/({z - \sigma(\alpha)}).$ Using (\ref{eq:rate_condition_1}),
$G_P(\alpha,z)$ can be upper bounded as\vspace{-0.1cm} 
\begin{equation}
\label{eq:bound_r}
\begin{aligned} 
G_P(\alpha,z)& \leq \inf_{\epsilon_{opt}\in (0, 2\tmu_{\mn} -  \alpha (\tmu_{\mn} - D_{\mn}^\ell))} \frac{ \frac{1}{2} \epsilon_{opt}^{-1} \cdot \frac{D_{\mx}^2}{\mu} + \frac{1}{\mu} \cdot \left( \left(1 - \frac{\alpha}{2}\right)\tmu_{\mn} + \frac{D_{\mn}^\ell}{2} \alpha - \frac{1}{2}  \epsilon_{opt} \right)}{\left(1 - \frac{\alpha}{2}\right)\tmu_{\mn} + \frac{D_{\mn}^\ell}{2} \alpha - \frac{1}{2}  \epsilon_{opt} } \cdot \theta^{-1}\\
& =    
G_P^\star(\alpha)\cdot \theta^{-1},
\end{aligned} 
\vspace{-0.2cm} 
\end{equation}
where  the minimum is attained at $\epsilon_{opt}^\star  \triangleq \tmu_{\mn} -  \frac{\alpha}{2} (\tmu_{\mn} - D_{\mn}^\ell)$;   and $G_P^\star(\alpha)$ is defined in~\eqref{eq:r_star_def}.
Substituting the upper bound~\eqref{eq:bound_r}  in $\PP(\alpha,z)$ and setting therein   $  \epsilon_{opt} = \epsilon_{opt}^\star$, we get the following sufficient condition for $\PP(\alpha,z)<1$:\vspace{-0.2cm} 
\begin{multline}\label{eq:step_size_bound}
G_P^\star(\alpha)\cdot  \theta^{-1} \cdot C_1 \cdot 4  L_{\mx}^2  \cdot G_X(z) \cdot \rho^2 \cdot \alpha^2\\
+ \left( G_P^\star(\alpha)\cdot  \theta^{-1}\cdot 2   C_1   + C_2  \right)  \cdot G_Y(z) \cdot 2 L_{\mx}^2 \rho^2  \cdot \alpha^2\\
+\left( G_P^\star(\alpha)\cdot  \theta^{-1} \cdot 2   C_1   + C_2   \right) \cdot G_Y(z)\cdot 8 L_{\mx}^2 \rho^2 \cdot G_X(z) \cdot \rho^2 \cdot \alpha^2 <1.
\end{multline}
To minimize the left hand side, we set  $\epsilon_x = \epsilon_y = (\sqrt{z} - \rho)/\rho$. Furthermore, using the fact that $G_P^\star(\alpha)$ is monotonically increasing on $\alpha\in (0, 2\tmu_{\mn}/(\tmu_{\mn} - D_{\mn}^\ell))$, and restricting  $\alpha\in (0, \tmu_{\mn}/(\tmu_{\mn} - D_{\mn}^\ell)]$, a sufficient condition for~\eqref{eq:step_size_bound} is \vspace{-0.2cm}
\begin{equation}\label{eq:step_size_bound_1}
\alpha \leq  \alpha(z) \triangleq \left(A_{1,\theta}\frac{1}{(\sqrt{z} - \rho)^2} + A_{2 ,\theta} \frac{1}{(\sqrt{z} - \rho)^2} + A_{3 ,\theta} \frac{1}{(\sqrt{z} - \rho)^4}\right)^{-1/2},\vspace{-0.2cm}
\end{equation}
where $A_{1,\theta}$, $A_{2,\theta}$ and $A_{3,\theta}$ are constants defined as\vspace{-0.2cm}
\begin{align*}
A_{1,\theta} & \triangleq G_P^\star(\tmu_{\mn}/(\tmu_{\mn} - D_{\mn}^\ell))\cdot  \theta^{-1} \cdot C_1 \cdot 4  L_{\mx}^2   \cdot \rho^2 \\
A_{2,\theta} & \triangleq \left( G_P^\star(\tmu_{\mn}/(\tmu_{\mn} - D_{\mn}^\ell))\cdot  \theta^{-1}\cdot 2   C_1   + C_2  \right)    \cdot 2 L_{\mx}^2 \rho^2   \\
A_{3,\theta} & \triangleq \left( G_P^\star(\tmu_{\mn}/(\tmu_{\mn} - D_{\mn}^\ell))\cdot  \theta^{-1} \cdot 2   C_1   + C_2   \right) \cdot 8 L_{\mx}^2 \rho^4.
\end{align*}
Condition~\eqref{eq:step_size_bound_1} shows the rate $z$ must satisfy \vspace{-0.2cm} 
\begin{equation}\label{eq:rate_condition_2}
z \geq \Big(\rho + \sqrt{A_\theta \alpha} \Big)^2,\quad \text{with}\quad A_{\theta} \triangleq \sqrt{A_{1,\theta} + A_{2,\theta} + A_{3,\theta}}.
\end{equation}

{Notice that, under   $\epsilon_x=\epsilon_y = (\sqrt{z} - \rho)/\rho$, \eqref{eq:rate_condition_2} implies  $z > \rho^2(1 + \epsilon_x) =  \rho^2 (1+ \epsilon_y) = \rho \sqrt{z}$, which are the  other  two conditions on $z$ in (\ref{eq:bound_z}).} Therefore, overall, $z$ must satisfy     
\eqref{eq:rate_condition_1} and \eqref{eq:rate_condition_2}.
Letting $\epsilon_{opt} = \epsilon_{opt}^\star$ in~\eqref{eq:rate_condition_1}, the condition simplifies to\vspace{-0.3cm} 
\begin{equation*}\label{eq:rate_condition_3}
z \geq 1 -  \frac{    \tmu_{\mn} - \frac{\alpha}{2} (\tmu_{\mn} - D_{\mn}^\ell)   }{\frac{2 D_{\mx}^2}{\mu} +  \tmu_{\mn} - \frac{\alpha}{2} (\tmu_{\mn} - D_{\mn}^\ell)  }\cdot (1- \theta) \alpha.\vspace{-0.2cm}
\end{equation*}

Therefore,  the overall convergence rate can be upper bounded by $\mathcal{O}(\bar{z}^\nu)$, where\vspace{-0.2cm}
\begin{equation}\label{eq:rate_expression_complete}
\bar{z} = \inf_{\theta \in (0,1)}\max \left\{ \Big(\rho + \sqrt{A_\theta \alpha} \Big)^2, 1 -  \frac{    \tmu_{\mn} - \frac{\alpha}{2} (\tmu_{\mn} - D_{\mn}^\ell)   }{\frac{2 D_{\mx}^2}{\mu} +  \tmu_{\mn} - \frac{\alpha}{2} (\tmu_{\mn} - D_{\mn}^\ell)  }\cdot (1- \theta) \alpha \right\}.
\end{equation}

\noindent Finally, we further simplify  ~\eqref{eq:rate_expression_complete}. 
Letting $\theta = 1/2$ and using  $\alpha \in (0,  \tmu_{\mn} /(\tmu_{\mn} - D_{\mn}^\ell)]$, the second term in \eqref{eq:rate_expression_complete} can be upper bounded by 
\begin{equation}\label{eq:expression_J}
1- \underbrace{\frac{\tmu_{\mn} \mu}{ 4 D_{\mx}^2 + \tmu_{\mn} \mu}\cdot \frac{1}{2}}_{\triangleq J} \alpha.
\end{equation}
The condition  $\bar{z} <1$   imposes the following upper bound on $\alpha$:
$\alpha < \alpha_{\mx} = \min\{(1-\rho)^2/A_{\frac{1}{2}},\tmu_{\mn} /(\tmu_{\mn} - D_{\mn}^\ell),1 \}.$
Eq.~\eqref{eq:rate_expression_complete} then simplifies to\vspace{-0.2cm} 
\begin{equation}
\bar{z}=\max \left\{ \bigg(\rho + \sqrt{ \alpha A_{\frac{1}{2}}   } \bigg)^2, 1-J\alpha\right\}.\vspace{-0.2cm} 
\end{equation}
Note that as $\alpha$ increases from $0$,  the first term in the max operator above is monotonically increasing from $\rho^2<1$  while the second term is monotonically decreasing from $1$. Therefore, there must exist some $\alpha^*$ so that the two terms are equal, which is\vspace{-0.2cm} 
\begin{equation}\label{eq:expression_alpha_ast}
\alpha^* =  \left(\frac{-\rho \sqrt{A_{\frac{1}{2}}} + \sqrt{A_{\frac{1}{2}}+ J (1 - \rho^2)}}{A_{\frac{1}{2}}+ J}\right)^2.
\end{equation}

To conclude, given the step-size satisfying $\alpha \in (0,\alpha_{\mx}) $, the sequence $\{ \|\Deltaxi{}^\nu\|^2 \}$ converges at rate $\mathcal{O} (z^\nu)$, with $z$ given in (\ref{eq:rate}). $\hfill \square$
{ 
\section{Proof of Corollary \ref{cor:centralized_rate}}\label{app:proof-SONATA-star}

 Since $\mathbf{W}=\mathbf{J}$,  we have  $\bdelta^\nu = \mathbf{0}$; then \eqref{eq:opt_err_0}  and \eqref{eq:delta_x_lb}  reduce to   
\begin{equation}\label{eq:p_nu_star_top}
	p^{\nu + 1}  
	\leq p^\nu -  \left( \left(1 - \frac{\alpha}{2}\right) \tmu_{\mn} +  \frac{\alpha D_{\mn}^\ell}{2} \right) \alpha \| \Deltaxi{}^\nu\|^2\end{equation}
and\vspace{-0.4cm} 
\begin{equation}\label{eq:lower-bound-d_star_top}
\alpha \,\| \Deltaxi{}^\nu \|^2 \geq \frac{2 \mu}{ D_{\mx}^2} \left(p^{\nu+1}- (1-\alpha) p^{\nu}\right),
\end{equation}
respectively. Combining (\ref{eq:p_nu_star_top}) and (\ref{eq:lower-bound-d_star_top}) and using 
 $\alpha < 2 \tmu_{\mn}/(\tmu_{\mn} - D_{\mn}) $, yield  
\begin{align}
p^{\nu + 1}   \leq \left( 1 - \alpha \cdot \frac{  \left(1 - \frac{\alpha}{2}\right) \tmu_{\mn} +  \frac{\alpha D_{\mn}^\ell}{2}  }{\frac{ D_{\mx}^2}{ 2\mu}  +  \left(1 - \frac{\alpha}{2}\right) \tmu_{\mn} +  \frac{\alpha D_{\mn}^\ell}{2} } \right) p^\nu,
\end{align}
which proves \eqref{eq:rate_central_expression}.

 We customize next \eqref{eq:rate_central_expression} to the specific choices of the surrogate functions. 
 
\noindent $\bullet$  \textbf{Linearization:} Consider   the choice of $\tf_i$ as in \eqref{eq:linear_surrogate}.  We have $\tmu_{\mn} = L$; and we can set 
  $D_{\mn}^\ell = 0$,   $D_{\mx} = L - \mu$,  and $\alpha = 1$. Substituting these values in \eqref{eq:rate_central_expression}, we obtain   $z \leq 1 - \kappa^{-1}_g$.

\noindent $\bullet$  \textbf{Local $f_i$:}  Consider now $\tf_i$ as in \eqref{eq:f_surrogate}. By $\nabla^2 f_i(\mathbf{x}) \succeq \mathbf{0}$, for all   $\mathbf{x}\in \KK$, and   Definition~\ref{assump:homogeneity}, we have 
 $\mathbf{0}  \preceq \nabla^2 \tf_i (\bx ,\by) - \nabla^2 F(\bx) \preceq 2\beta\mathbf{I}$, for all $\bx, \by \in \KK$. Therefore,  
 we can set $D_{\mn}^\ell = 0$, $D_{\mx} = 2\beta$, and $\tmu_{\mn} = \beta + (\mu - \beta)_+$. 
 Using these values in \eqref{eq:rate_central_expression}, yields \vspace{-0.2cm}
 \begin{align}
 z   
\begin{cases}
= 1 - \alpha   \cdot  \frac{  \beta \left( 1- \frac{\alpha}{2}\right)  }{\frac{ 2 \beta^2}{\mu}  +  \beta \left( 1- \frac{\alpha}{2}\right) }, & \text{if }\mu \leq \beta\\
 \leq   1 -    \alpha \cdot \frac{\mu \left( 1- \frac{\alpha}{2}\right) }{\frac{ 2\beta^2}{\mu}  + \mu \left( 1- \frac{\alpha}{2}\right)}, & \text{if } \mu > \beta.
\end{cases} 
 \end{align}
Finally, setting $\alpha = \min \{1, 2\tmu_{\mn}/((\mu - \beta)_++ \beta)\}=1$ in the expression above, yields \eqref{eq:z_surrogate}.  \hfill $\square$}
\section{Proof of Corollary~\ref{cor:linearization_rate}}\label{app:linearization_rate} 
{  According to 
	Theorem~\ref{thm:linear_rate},  the rate $z$ can be bounded as \begin{equation}\label{eq_rate_summary} z \leq \max \{z_1, z_2\},\quad \text{ with }\quad 
	z_1 \triangleq 1 - \alpha\cdot J \,\,\text{ and }\,\,  z_2 \triangleq \Big(\rho + \sqrt{\alpha A_{\frac{1}{2}}}\Big)^2,\end{equation} where $J$ and $A_{\frac{1}{2}}$ are defined in (\ref{eq:expression_J}) and (\ref{eq:rate_condition_2}), respectively. 
	
	The proof consists in bounding properly $z_1$ and $z_2$ based upon  the surrogate   \eqref{eq:linear_surrogate} postulated in the corollary.  
	We begin
	particularizing the expressions of $J$ and $A_{\frac{1}{2}}$.   
	Since $\nabla^2 \tf_i (\bx_i; \bx_i^\nu) = L$, one can set $\tmu_{\mn} =  L$, and (\ref{eq:upper-lower-hessian}) holds with     $D_{\mn}^\ell = 0$ and  $D_{\mx} = L -\mu$. Furthermore,   by Assumption~\ref{assump:homogeneity}, it follows that $ \beta\geq \lambda_{\max}(\nabla^2 f_i (\bx)) -L$, for all $\mathbf{x}\in \mathcal{K}$; hence, one can set  $L_{\mx} = L + \beta$.  
		Next, we will substitute the above values into the expressions of $J$ and $A_{\frac{1}{2}}$. 
		
		To do so, we need to particularize first the  quantities $G_P^\star \left(\frac{\tmu_{\mn}}{\tmu_{\mn} - D_{\mn}^\ell}\right)$ [cf. (\ref{eq:r_star_def})], $C_1$   and  $C_2$ [cf. (\ref{eq:C1_C2})]: \vspace{-0.2cm}
	\begin{align*}
	&G_P^\star \left(\frac{\tmu_{\mn}}{\tmu_{\mn} - D_{\mn}^\ell}\right) = G_P^\star \left(1\right)  = \frac{4 (L - \mu)^2 + L^2}{ \mu L^2 }, \\& C_1 = \frac{6}{\mu L^2}\left(   (2 L - \mu)^2  + 4 (L + \beta)^2 \right), \quad \text{and}\quad   C_2 =\frac{4}{L^2}.
	\end{align*}
	Accordingly, the expressions of $J$ and $A_{\frac{1}{2}}$ read: \begin{equation}\label{eq:J_custom}
	J  =  
	{ \frac{1}{2}}\frac{\kappa_g }{4 (\kappa_g - 1)^2 +\kappa_g}\in \left[\frac{1}{8\kappa_g}, \frac{1}{2}\right],
	\end{equation}  
	and 
	\begin{align}\label{eq:A_1/2_upper}
	\begin{split}
	& (A_{\frac{1}{2}})^2 \\
	= {}& G_P^\star(1)\cdot  2 \cdot C_1 \cdot 4  L_{\mx}^2   \cdot \rho^2 + \left( G_P^\star(1)\cdot  4 \cdot C_1   + C_2  \right)    \cdot 2 L_{\mx}^2 \rho^2\\
	& +  \left( G_P^\star(1)\cdot  4 \cdot C_1   + C_2   \right) \cdot 8 L_{\mx}^2 \rho^4\\
	= 
	&  \left( 24 G_P^\star(1)  \cdot C_1   + 5 C_2  \right)    \cdot 2 L_{\mx}^2 \rho^2\\
	=
	&  \left( 24 \cdot \frac{4 (L - \mu)^2 + L^2}{ \mu L^2 } \cdot \frac{6}{\mu L^2}\left(   (2 L - \mu)^2  + 4 (L + \beta)^2 \right)   + 20 L^{-2}  \right)    \cdot 2 (L + \beta)^2 \rho^2\\
	\leq 
	& \left(24 \cdot \frac{5}{\mu} \cdot \frac{24}{\mu L^2}\left( L^2  +  (L + \beta)^2 \right)   + 20 L^{-2} \right)\cdot 2 (L + \beta)^2 \rho^2\\
	=
	& \left(24 \cdot 24 \cdot 5 \left( 1 +  \left(1 + \frac{\beta}{L} \right)^2 \right) \left(1 + \frac{\beta}{L} \right)^2  \kappa_g^2 + 20  \left(1 + \frac{\beta}{L} \right)^2\right)\cdot 2  \rho^2
	\\
	\leq
	& 110^2 \cdot \kappa_g^2 \left(1 + \frac{\beta}{L} \right)^4 \rho^2, 
	\end{split}
	\end{align}
	where in the last inequality we have used the fact that  $\kappa_g \geq 1$.
	
	Using the above expressions, in the sequel we upperbound  $z_1$ and $z_2$. 
	
	
	By (\ref{eq:A_1/2_upper}),   we have 	
	\begin{equation} z_2\leq \bar{z}_2 \triangleq \Big(\rho + \sqrt{\alpha M \rho}\Big)^2,\quad \text{with}\quad 
	M \triangleq 110 \cdot \kappa_g (1 + \beta/L)^2.\end{equation}
	Since   $\alpha \in (0,1]$ must be chosen so that $z \in (0,1]$, we impose $\max\{z_1, \bar{z}_2\}<1$, implying 
	$\alpha \leq \min\{J^{-1}, (1-\rho)^2/(M \rho),1 \}$. 
	Since $J^{-1} >1$ [cf. \eqref{eq:J_custom}],    the  condition on $\alpha$ reduces to $\alpha \leq \alpha_{\mx} \triangleq  \min\{(1-\rho)^2/ (M \rho),1\}$. Choose  $\alpha = c \cdot \alpha_{\mx}$, for some given  $c \in (0,1)$.
	Depending on the value of $\rho$, either $\alpha_{\mx}=1$ or $\alpha_{\mx}=(1-\rho)^2/(M \rho)$. 
}

{ 
	\noindent $\bullet$  \textbf{Case I: $\alpha_{\mx} = 1$.} This corresponds to the case  $ M \rho \leq (1 - \rho)^2$, which happens when the network is sufficiently connected ($\rho$ is small). Note that, we also have  $\rho \leq 1/110$,  otherwise $M \rho \geq 110 \,\kappa_g \, \rho >1 > (1 - \rho)^2$. 
	In this setting,   $\alpha = c\cdot \alpha_{\mx} =c$, and  \vspace{-0.1cm}\begin{align*}
	z_1 &=   1 - c \cdot J,\\
	\bar{z}_2 & = \Big(\rho + \sqrt{c M \rho}\Big)^2 \overset{(a)}{\leq}
	\Big( 1 - (1 - \rho) + \sqrt{c (1 - \rho)^2}\Big)^2\\
	& = \left(1 - \left(1 - \sqrt{  c} \right)(1 - \rho) \right)^2 \leq 1 - \left(1 - \sqrt{  c} \right)^2 (1 - \rho)^2\\
	&\overset{(b)}{\leq} 1 - (1 - \sqrt{c})^2 (1 - 1/110)^2,
	\end{align*}
	where in (a) we used $ M \rho \leq (1 - \rho)^2$ and (b) follows from  $\rho \leq 1/110$. 
	
	Therefore, $z$ can be bounded as
	\begin{align}
	\begin{split}
	z  \leq \max \{z_1, \bar{z}_2\} & \leq 1 - c \cdot  \left(1 - \sqrt{  c} \right)^2\cdot  \left(1 -  {1}/{110}\right)^2 \cdot J\\
	& \leq  1 - c \cdot  \left(1 - \sqrt{  c} \right)^2\cdot \left(1 - {1}/{110}\right)^2 \cdot \frac{1 }{8 \kappa_g}.
	\end{split}
	\end{align} }

{ \noindent \noindent $\bullet$ \textbf{Case II: $\alpha_{\mx} = (1-\rho)^2/(M \rho)$.}  This corresponds to the case $ M \rho \geq (1 - \rho)^2$.    
	We have $\alpha = c \cdot \alpha_{\mx}=c\cdot (1-\rho)^2/(M \rho)$, \vspace{-0.2cm}
	\begin{equation*}
	z_1 = 1 - \frac{J \, c}{{M \rho}} \cdot   {(1-\rho)^2} \quad \text{and} \quad \bar{z}_2 = 1 - \left(1 - \sqrt{  c} \right)^2 (1 - \rho)^2.
	\end{equation*}
	We claim that $({J\,c})/({M \rho})<1$. Suppose this is not the case, that is, $M \rho \leq   {J c}$. Since  {${J c}<1/{2}$} [cf. \eqref{eq:J_custom}] and $M\geq 110\,\kappa$, $M \rho \leq   {J c}$ would imply  $\rho<1/(220\kappa_g)$. This however is in contradiction with the assumption 
	$ M \rho \geq (1 - \rho)^2$, as it would lead to $1/2>M\rho \geq (1-\rho)^2>(1-1/(220\kappa_g))^2$.  
	
	Using $({J\,c})/({M \rho})<1$, we can bound $z$
	\begin{align*}\vspace{-0.3cm}
	\begin{split}
	z  & \leq \max\{z_1,\bar{z}_2\} \leq 1 - \frac{c\,J}{{M \rho}} \cdot  \left(1 - \sqrt{  c} \right)^2     {(1-\rho)^2} \\
	& \leq  1 -c \cdot  \left(1 - \sqrt{  c} \right)^2 \cdot \frac{1 }{8 \kappa_g }   \cdot \frac{(1-\rho)^2}{110 \cdot \kappa_g \cdot (1 + \beta/L)^2 \cdot \rho}.
	\end{split}
	\end{align*}\vspace{-0.3cm}
}
\vspace{-0.6cm}
\section{Proof of Corollaries \ref{cor:f_rate_1}  and \ref{cor:f_rate_2}}\label{app:f_rate}{  
\vspace{-0.3cm} We follow similar steps as in Appendix \ref{app:linearization_rate} but customized to  the surrogate   \eqref{eq:f_surrogate}. We begin particularizing the expressions of $J$ and $A_{\frac{1}{2}}$.
	
	In the setting of the corollary, we have:  $\nabla^2 \tf_i (\bx; \by) = \nabla^2 f_i (\bx) + \beta \bI$,  for all $\mathbf{y}\in \KK$;  $\nabla^2 f_i(\mathbf{x}) \succeq \mathbf{0}$, for all   $\mathbf{x}\in \KK$; and,  by Assumption~\ref{assump:homogeneity}, 
	$\mathbf{0}  \preceq \nabla^2 \tf_i (\bx ,\by) - \nabla^2 F(\bx) \preceq 2\beta\mathbf{I}$, for all $\bx, \by \in \KK$. Therefore,  
	we can set $D_{\mn}^\ell = 0$, $D_{\mx} = 2 \beta$,   $\tmu_{\mn} = \beta + (\mu - \beta)_+=\max\{\beta, \mu\}$, and $L_{\mx} = L + \beta$. 
	Using these values, 
	$G_P^\star \left(\frac{\tmu_{\mn}}{\tmu_{\mn} - D_{\mn}^\ell}\right)$, $C_1$, and $C_2$   can be simplified as follows:
	\begin{align*}
	&G_P^\star \left(\frac{\tmu_{\mn}}{\tmu_{\mn} - D_{\mn}^\ell}\right) = G_P^\star \left(1\right)  = \frac{ 16 \beta^2 +  \max\{\beta, \mu\}^2}{ \mu \max\{\beta, \mu\}^2 },  \\
	& C_1 = \frac{6}{\mu}\left(  \left( \frac{2 \beta}{\max\{\beta, \mu\}} + 1\right)^2 + \frac{4 (L + \beta)^2}{\max\{\beta, \mu\}^2 } \right), \quad \text{and}\quad C_2 =\frac{4}{\max\{\beta, \mu\}^2}.
	\end{align*}
	Accordingly, the expressions of $J$ and $A_{\frac{1}{2}}$ read: \vspace{-0.2cm}\begin{equation}\label{eq:J_II}J = \frac{1}{2}\frac{1}{1 + 16 \left(\frac{\beta}{\mu}\right) \cdot \min \left\{1, \frac{\beta}{\mu}\right\}},\vspace{-0.3cm} \end{equation}  and \vspace{-0.2cm}
	\begin{align*}
	\begin{split}
	& (A_{\frac{1}{2}})^2 \\
	\leq 
	&  \left( 24 G_P^\star(1)  \cdot C_1   + 5 C_2  \right)    \cdot 2 L_{\mx}^2 \rho^2\\
	\leq
	&  \left( 24 \cdot \frac{ 16 \beta^2 +  \max\{\beta, \mu\}^2}{ \max\{\beta, \mu\}^2 }\cdot \frac{6}{\mu^2}\left(  \left( \frac{2 \beta}{\max\{\beta, \mu\}} + 1\right)^2 + \frac{4 (L + \beta)^2}{\max\{\beta, \mu\}^2 } \right)  + \frac{20} {\max\{\beta, \mu\}^{2}}  \right)    \cdot 2 (L + \beta)^2 \rho^2\\
	=
	&
	\begin{cases}
	\left( 24 \cdot 17\cdot 6 \cdot \left(9 + 4\left( 1+\frac{L}{\beta}\right)^2 \right) \cdot  \left(\kappa_g + \frac{\beta}{\mu}\right)^2 + 20 \left(1 + \frac{L}{\beta}\right)^2 \right)    \cdot 2 \rho^2, & \beta>\mu,\\
	\left( 24 \cdot \left(\frac{16 \beta^2}{\mu^2} + 1\right)\cdot 6 \left(\kappa_g + \frac{\beta}{\mu}\right)^2\left(  \left( \frac{2 \beta}{\mu} + 1\right)^2 + 4 \left(\kappa_g + \frac{\beta}{\mu}\right)^2 \right)  + 20 \left(\kappa_g + \frac{\beta}{\mu}\right)^2  \right)    \cdot 2 \rho^2, & \beta\leq \mu;
	\end{cases}\\
	\leq & M^2 \rho^2,
	\end{split}
	\end{align*}
	where\vspace{-0.2cm}
	\begin{align} 
	M = 
	\begin{cases}\label{eq:def_M_surrogate}
	253\left(1 + \frac{L}{\beta}\right) \left(\kappa_g + \frac{\beta}{\mu}\right), &  \beta>\mu,\\
	193\left(1 + \frac{\beta}{\mu}\right)^2\left(\kappa_g + \frac{\beta}{\mu}\right)^2, & \beta\leq \mu.
	\end{cases}
	\end{align}
	
	Similarly to the proof of Corollary~\ref{cor:linearization_rate}, we bound $z \leq \max \{z_1, z_2\}$ as
	\begin{equation}\label{eq_rate_summary_2} z \leq \max \{z_1, \bar{z}_2\},\quad \text{ with }\quad 
	z_1 \triangleq 1 - \alpha\cdot J \,\,\text{ and }\,\,  \bar{z}_2   \triangleq \Big(\rho + \sqrt{\alpha\,M\,\rho}\Big)^2,\end{equation} where $J$ and $M$ are now given by (\ref{eq:J_II}) and (\ref{eq:def_M_surrogate}),  respectively. 
	For  $\max \{z_1, z_2\}<1$, we require $\alpha \leq \alpha_{\mx} \triangleq  \min\{1,  (1-\rho)^2/(M \rho)\}$,   and choose  $\alpha = c \cdot \alpha_{\mx}$, with arbitrary $c \in (0,1)$.
	We  study  separately the cases $\beta>\mu$ and $\beta\leq \mu$.
	
	\noindent  \textbf{1)  $\beta>\mu$.} In this case we have \begin{equation} \label{eq:M_and_L_Case_I}M =  253\left(1 + \frac{L}{\beta}\right) \left(\kappa_g + \frac{\beta}{\mu}\right)\quad  \text{and}\quad  J =  \frac{1}{2}\frac{1}{1 + 16 \left({\beta}/{\mu}\right) } \geq \frac{1}{34 (\beta /\mu)}.\end{equation} Since $\alpha = c \alpha_{\mx} = c \min\{1,  (1-\rho)^2/(M \rho)\}$, we study next the case $\alpha_{\mx}=1$ and $\alpha_{\mx}= (1-\rho)^2/(M \rho)$ separately.
	
	\begin{itemize}
		\item
		\textbf{Case I: $\alpha_{\mx} = 1$}. We have $ M \rho \leq (1 - \rho)^2$, $\alpha = c$, and thus
		\begin{equation*}
		z_1  = 1 - c \cdot J \quad \text{and} \quad \bar{z}_2    \leq 1 - \left(1 - \sqrt{  c} \right)^2 (1 - \rho)^2.
		\end{equation*}
		Since $M \geq 253$ and $(1 - \rho)^2 \leq 1$, it must be $\rho \leq 1/253$.
		Therefore, the rate $z$ can be bounded as
		\begin{align*}
		\begin{split}
		z  \leq \max \{z_1, \bar{z}_2\} & \leq 1 - c \cdot  \left(1 - \sqrt{  c} \right)^2\cdot  J \cdot (1 - \rho)^2\\
		& \leq  1 - c \cdot  \left(1 - \sqrt{  c} \right)^2\cdot \left(1 - \frac{1}{253}\right)^2 \cdot \frac{1}{34} \cdot \frac{\mu}{\beta}.
		\end{split}
		\end{align*} 
		
		\item   \textbf{Case II: $\alpha_{\mx} = (1-\rho)^2/(M\rho)$}. This corresponds to  $ M\rho \geq (1 - \rho)^2$, $\alpha = c \cdot (1-\rho)^2/(M\rho)$, and
		\begin{equation*}
		z_1 = 1 - \frac{J \, c}{{M \rho}} \cdot   {(1-\rho)^2}  \quad \text{and} \quad \bar{z}_2 \leq 1 - \left(1 - \sqrt{  c} \right)^2 (1 - \rho)^2.
		\end{equation*}
		Using the same argument as in the proof of Corollary \ref{cor:linearization_rate}--Case II, one  can show that  {$(c\,J)/(M\rho)<1$}. Therefore, 
		\begin{align*}
		z &\leq \max\{z_1,\bar{z}_2\}\leq 1 -   \left(1 - \sqrt{  c} \right)^2 \cdot c  \, J \cdot   \frac{(1-\rho)^2}{M \rho}\\
		&\overset{\eqref{eq:M_and_L_Case_I}}{\leq}  1 -   c \cdot  \left(1 - \sqrt{  c} \right)^2 \cdot \frac{1}{34 } \cdot   \frac{(1-\rho)^2}{253\left(\kappa_g + \frac{\beta}{\mu}\right)^2  \rho}.
		\end{align*}
	\end{itemize}
	\noindent \textbf{2)  $\beta\leq \mu$.}  In this case we have \begin{equation}\label{eq:M_and_L_Case_II}M = 193\left(1 + \frac{\beta}{\mu}\right)^2\left(\kappa_g + \frac{\beta}{\mu}\right)^2\quad \text{and}\quad J = \frac{1}{2}\frac{1}{1 + 16 \left({\beta}/{\mu}\right)^2}.\end{equation}
	
	\begin{itemize}\item \textbf{Case I: $\alpha_{\mx} = 1$.} Following the same reasoning as $\mu \leq \beta$, we can prove
		\begin{align}
		z \leq \max\{z_1,\bar{z}_2\}\leq  1 - c \cdot  \left(1 - \sqrt{  c} \right)^2\cdot \left(1 - \frac{1}{193}\right)^2 \cdot \frac{1}{2 + 32 \left(\frac{\beta}{\mu}\right)^2}.
		\end{align}
		
		\item \textbf{Case II: $\alpha_{\mx} = (1-\rho)^2/(M\rho)$.} We claim that $(c\,J)/(M\rho)\leq 1$, otherwise  $\rho\leq c/386$, which would lead to the following contradiction 
		$c/2\geq (c\,J) > M \rho\geq (1-\rho)^2 \geq (1-c/386)^2$.  
		Therefore,
		\begin{align*}
		z & \leq \max\{z_1,\bar{z}_2\}\leq 1 - c \cdot  \left(1 - \sqrt{  c} \right)^2 \cdot \frac{1}{2 + 32 \left(\frac{\beta}{\mu}\right)^2} \frac{(1-\rho)^2}{193\left(1 + \frac{\beta}{\mu}\right)^2\left(\kappa_g + \frac{\beta}{\mu}\right)^2 \rho}\\
		& \leq 1 - c^\prime \cdot \frac{(1 - \rho)^2}{\kappa_g^2 \,\rho},
		\end{align*}
	\end{itemize}
	where $c^\prime\in (0,1)$ is a suitable constant, independent on $\beta/\mu$, $\kappa_g,$ and $\rho$.\hfill $\square$ }

}

 \bibliographystyle{tfs}
\bibliography{references}

\newpage 

\section*{Supporting Material}

\setcounter{section}{0}

\renewcommand{\thesection}{\Roman{section}}

\section{Proof of Proposition  \ref{prop:loc_geo_rate2}}\label{SUpp_proof_prop_4_1}
We begin introducing some intermediate results. 
\begin{lemma}\label{lem:loc_dec_obj_TV} 
	Consider Problem~\eqref{eq:P} under Assumption~\ref{assump:p}; and 
 SONATA    (Algorithm~\ref{alg:SONATA_TV}) under Assumptions~\ref{assump:SCA_surrogate} and \ref{assumption:TVwights}. 
	Then, there  holds\vspace{-0.1cm}
	\begin{align}\label{eq:loc_dec_obj_TV}
	U(\bx_i^{\nu+\frac{1}{2}} ) \leq U(\bx_i^\nu) - \alpha \left(\left(1 - \frac{\alpha}{2}\right)\tmu_i +  \frac{\alpha}{2} \cdot D_i^{\ell} \right)\| \Deltaxi{i}^\nu\|^2   + \alpha\|\Deltaxi{i}^\nu\| \|\bdelta_i^\nu \|,\vspace{-0.1cm}
	\end{align}
	with $\bdelta_i^\nu$  defined in~\eqref{eq:tracking_err_def}.
\end{lemma}
\begin{proof}
	Consider the Taylor expansion of $F$:  
	\begin{align}\label{eq:taylor_F_TV}
	\begin{split}
	F(\bx_i^{\nu +{\frac{1}{2}}})= \,& F(\bx_i^\nu) + \nabla F(\bx_i^\nu)^\top (\alpha \Deltaxi{i}^\nu)  + (\alpha\Deltaxi{i}^\nu)^\top \mathbf{H}(\alpha\Deltaxi{i}^\nu),\\
	\stackrel{\eqref{eq:tracking_err_def}}{=}& F(\bx_i^\nu) + \big(\bdelta_i^\nu\big)^\top(\alpha \Deltaxi{i}^\nu) + \big(\mathbf{y}_i^\nu\big)^\top(\alpha \Deltaxi{i}^\nu) + (\alpha\Deltaxi{i}^\nu)^\top \mathbf{H}(\alpha\Deltaxi{i}^\nu),
	\end{split}
	\end{align}
	where $\mathbf{H} \triangleq  \int_{0}^1 (1 - \theta) \nabla^2 F (\theta \bx_i^{\nu + \frac{1}{2}} + (1 - \theta) \bx_i^\nu) d \theta$.
	
	Invoking the  optimality of $\hbx_i^\nu$, we have 
	\begin{align}\label{eq:opt_x_hat_TV}
	\begin{split}
	G(\bx_i^\nu)-G(\widehat{\bx}_i^\nu) & \geq (\Deltaxi{i}^\nu)^\top \big(\nabla \tf_i(\hbx_i^\nu;\bx_i^\nu)  +  \by_i^\nu - \nabla f_i(\bx_i^\nu)\big) 
	= (\Deltaxi{i}^\nu)^\top \big( \by_i^\nu + \tH_i \Deltaxi{i}^\nu\big) \\
	\end{split}
	\end{align}
	where the equality follows from  $\nabla \tf_i(\bx_i^\nu;\bx_i^\nu)=\nabla f_i(\bx_i^\nu)$ and the integral form of the mean value theorem; and  $\tH_i \triangleq  \int_0^1 \nabla^2 \tf_i (\theta \,\hbx_i^\nu + (1 - \theta) \,\bx_i^\nu; \bx_i^\nu) d \theta$. 
	
	Substituting (\ref{eq:opt_x_hat_TV}) in (\ref{eq:taylor_F_TV}) and  using the convexity of $G$ yield
	\begin{align}\label{eq:descent_lem_TV}
	\begin{split}
	&F(\bx_i^{\nu +{\frac{1}{2}}})\\
	\leq &\,F(\bx_i^\nu) +(   \bdelta_i^\nu)^\top (\alpha \Deltaxi{i}^\nu)  + (\alpha\Deltaxi{i}^\nu)^\top \mathbf{H}(\alpha\Deltaxi{i}^\nu)
	+ \alpha \left(G(\bx_i^\nu)-G(\widehat{\bx}_i^\nu)  - (\Deltaxi{i}^\nu)^\top \tH_i \Deltaxi{i}^\nu\right)\\
	\leq & \, F(\bx_i^\nu) +(   \bdelta_i^\nu)^\top (\alpha \Deltaxi{i}^\nu)  +\alpha \left(  - (\Deltaxi{i}^\nu)^\top \tH_i \Deltaxi{i}^\nu+  (\alpha\Deltaxi{i}^\nu)^\top \mathbf{H}(\Deltaxi{i}^\nu)  \right)
	+ G (\bx_i^\nu) - G(\bx_i^{\nu + \frac{1}{2}}).
	\end{split}
	\end{align}\vspace{-0.5cm}
	It remains to bound $ \alpha \bH- \tH_i$. We proceed as follows:
	\begin{align}\label{eq:Delta_H_bound_TV}
	\begin{split}
	& \, \alpha \bH- \tH_i \\
	= &\, \alpha \int_{0}^1 (1 - \theta) \nabla^2 F (\theta \bx_i^{\nu + \frac{1}{2}} + (1 - \theta) \bx_i^\nu) d \theta -  \int_0^1 \nabla^2 \tf_i (\theta \hbx_i^\nu  + (1 - \theta) \bx_i^\nu; \bx_i^\nu) d \theta\\
	\stackrel{\eqref{eq:descent}}{=} &\, \int_0^\alpha (1 - \theta/\alpha) \nabla^2 F (\theta \hbx_i^\nu + (1 - \theta) \bx_i^\nu) d \theta - \int_0^1 \nabla^2 \tf_i (\theta \hbx_i^\nu + (1 - \theta) \bx_i^\nu; \bx_i^\nu) d \theta\\
	\stackrel{(a)}{\preceq}  & 
	\,-\int_0^\alpha (1 - \theta/\alpha) \cdot ( D_i^{\ell} )\,\mathbf{I}\, d\theta - \int_0^\alpha (\theta/\alpha) \nabla^2 \tf_i (\theta \hbx_i + (1 - \theta) \bx_i^\nu;\mathbf{x}_i^\nu) d \theta\\
	& \,  - \int_\alpha^1 \nabla^2 \tf_i (\theta \,\hbx_i^\nu + (1 - \theta) \,\bx_i^\nu;\mathbf{x}_i^\nu) d \theta\\
	\stackrel{(b)}{\preceq} & \, -\frac{1}{2} \alpha \, ( D_i^{\ell})\,\mathbf{I}  - \left( 1 - \frac{\alpha}{2}\right)  \, \tmu_i \,\mathbf{I},
	\end{split}
	\end{align}
	where in (a) we used $\nabla^2 F(\theta \hbx_i^\nu + (1 - \theta) \bx_i^\nu) \preceq - (D_i^{\ell}) \mathbf{I} + \nabla^2\tf_i (\theta \hbx_i^\nu + (1 - \theta) \bx_i^\nu; \bx_i^\nu)$ [cf.~\eqref{eq:upper-lower-hessian}] while (b) follows from the fact that $\tf_i$  is  $\tmu_i$-strongly convex (cf. Assumption~\ref{assump:SCA_surrogate}). Substituting \eqref{eq:Delta_H_bound_TV} into  \eqref{eq:descent_lem_TV} completes the proof
\end{proof}

We connect now the individual decreases in (\ref{eq:loc_dec_obj_TV}) with that of the   optimality gap $\optgap^\nu$, defined in (\ref{eq:opt_gap_TV}).  Notice that   \begin{equation}\label{eq:bound_doubly_stoc_cvx_TV}
\sum_{i=1}^m \phi_i^{\nu+1}U(\bx_i^{\nu + 1}) \leq \sum_{i=1}^m\sum_{j=1}^m c_{ij}\phi_j^\nu U\Big(\bx_j^{\nu+\frac{1}{2}}\Big) 
=  \sum_{i=1}^m \phi_i^\nu U(\bx_i^{\nu+\frac{1}{2}}),
\end{equation}
due to the convexity of $U$, column-stochasticity of $\{c_{ij}^\nu\}_{i,j}$  and $\sum_{j=1}^m {c^\nu_{ij}\phi_j^\nu}/{\phi_i^{\nu+1}}=1$, for all $i=1,\ldots, m$. Summing (\ref{eq:loc_dec_obj_TV}) over $i=1,\ldots m$, and using (\ref{eq:bound_doubly_stoc_cvx_TV}),  we obtain 
\begin{equation}\label{eq:opt_err_1_TV}
	\begin{aligned} 
	\optgap^{\nu + 1}  &\leq \optgap^\nu + \sum_{i=1}^m \phi_i^\nu\left\{ \alpha \|\Deltaxi{i}^\nu\|  \| \bdelta_i^\nu \| - \alpha \left(1 - \frac{\alpha}{2}\right) \tmu_i \| \Deltaxi{i}^\nu\|^2 - \frac{D_i^{\ell}}{2}\alpha^2  \| \Deltaxi{i}^\nu\|^2 \right\} 
	\\
	& \overset{(a)}{\leq} \optgap^\nu -  \left( \left(1 - \frac{\alpha}{2}\right) \tmu_{\mn} +  \frac{\alpha D_{\mn}}{2} - \frac{1}{2} \epsilon_{opt}\right) \alpha \sum_{i=1}^m \phi_i^\nu\| \Deltaxi{i}^\nu\|^2 + \frac{1}{2} \epsilon_{opt}^{-1} \,\alpha \cdot  \phi_{ub}\cdot \| \bdelta^\nu \|^2,
	\end{aligned}
	\end{equation}
where in (a) we used  Young's inequality, with   $\epsilon_{opt}>0$ satisfying 
\begin{equation}\label{eq:eps_opt_TV}
\left(1 - \frac{\alpha}{2}\right) \tmu_{\mn} +  \frac{\alpha D_{\mn}^L}{2} - \frac{1}{2} \epsilon_{opt} >0.
\end{equation}

Next we lower bound  $\|\Deltaxi{}^\nu\|^2$ in terms of  the   optimality gap. 

\begin{lemma}\label{lem:delta_x_lb_TV}
	In the setting of Lemma \ref{lem:loc_dec_obj}, there holds:
	\begin{equation}\label{eq:delta_x_lb_TV}
	\alpha\sum_{i=1}^m\phi_i^\nu\| \Deltaxi{i}^\nu \|^2 \geq \frac{\mu}{D_{\mx}^2} \left(\optgap^{\nu+1} - (1-\alpha) \optgap^\nu - \frac{\alpha}{\mu}
	\sum_{i=1}^m\phi_i^\nu \| \bdelta_i^\nu \|^2\right)
	\vspace{-0.1cm}
	\end{equation}with $D_{\mx}$ defined in \eqref{eq:alg_param_def}.
	
	
\end{lemma}
\begin{proof}
	Invoking the optimality condition of $\hbx_i^\nu$, yields 
	
	\begin{align}\label{eq:FOC_TV}
	G(\bx^\star)-G(\hbx_i^\nu)\geq -(\bx^\star - \hbx_i^\nu)^\top \Big(\nabla \tf_i ( \hbx_i^\nu ; \bx_i^\nu)  +  \by_i^\nu - \nabla f_i (\bx_i^\nu) \Big).
	\end{align}
	Using the $\mu$-strong convexity of $F$, we can write
	\begin{align*}
	\hspace{-0.2cm}\begin{split}
	& U(\bx^\star)   \geq  U(\hbx_i^\nu) +    G(\bx^\star)-G(\hbx_i^\nu) +\nabla F( \hbx_i^\nu )^\top (\bx^\star - \hbx_i^\nu) + \frac{\mu}{2} \| \bx^\star -  \hbx_i^\nu\|^2   \\
	& \overset{(\ref{eq:FOC})}{\geq}  \!\!  U(\hbx_i^\nu) \!+\! \Big( \nabla F( \hbx_i^\nu )\! -\!\nabla \tf_i ( \hbx_i^\nu ; \bx_i^\nu) \! -\! \big( \by_i^\nu - \nabla f_i (\bx_i^\nu)\big)  \Big)^\top \!\!\!(\bx^\star - \hbx_i^\nu) + \frac{\mu}{2} \| \bx^\star -  \hbx_i^\nu\|^2\\
	& \,\,\,\,=   U(\hbx_i^\nu) + \frac{\mu}{2} \Big\| \bx^\star -  \hbx_i^\nu + \frac{1}{\mu}\Big( \nabla F( \hbx_i^\nu ) -\nabla \tf_i ( \hbx_i^\nu ; \bx_i^\nu)  - \big(  \by_i^\nu - \nabla f_i (\bx_i^\nu) \big) \Big) \Big\|^2 \\
	& \,\,\,\,\quad- \frac{1}{2\mu } \left\|  \nabla F( \hbx_i^\nu ) -\nabla \tf_i ( \hbx_i^\nu ; \bx_i^\nu)  - \big(\by_i^\nu - \nabla f_i (\bx_i^\nu) \big)  \right\|^2
		\end{split}
		\end{align*}
			\begin{align*}\begin{split}
	&  \,\,\,\,\geq  U(\hbx_i^\nu)  - \frac{1}{2\mu } \left\|  \nabla F( \hbx_i^\nu )  \pm \nabla F(\bx_i^\nu)-\nabla \tf_i ( \hbx_i^\nu ; \bx_i^\nu)  - \big( \by_i^\nu - \nabla f_i (\bx_i^\nu) \big)  \right\|^2
	\\
	& \,\,\,\, \geq  U(\hbx_i^\nu)  - \frac{1}{\mu } \left \|  \nabla F( \hbx_i^\nu )  -  \nabla F(\bx_i^\nu) + \nabla f_i(\bx_i^\nu) - \nabla \tf_i ( \hbx_i^\nu ; \bx_i^\nu) \right\|^2 - \frac{1}{\mu} \|\bdelta_i^\nu \|^2 
	\\
	&  \,\,\,\, =  U(\hbx_i^\nu)  - \frac{1}{\mu } \left\| \int_{0}^{1}\left(\nabla^2 F(\theta \hbx_i^\nu+(1-\theta)\bx_i^\nu)-\nabla^2 \tf_i(\theta \hbx_i^\nu+(1-\theta)\bx_i^\nu; \bx_i^\nu)\right)(\mathbf{d}_i^\nu)\,\text{d}\theta \right\|^2 \!\!- \frac{1}{\mu} \|\bdelta_i^\nu \|^2 
	\\
	&  \,\,\,\, \geq  U(\hbx_i^\nu)  - \frac{D_i^2}{\mu } \norm{\mathbf{d}_i^\nu}^2 - \frac{1}{\mu} \|\bdelta_i^\nu \|^2 ,
	\end{split}
	\end{align*}
	where $D_i = \max\{ |D_i^{\ell}|,| D_i^{u}|\}$.

	Rearranging the terms and summing over $i=1,\ldots, m$, yields  \begin{equation}\label{eq:delta_x_lb_a_TV}
	\sum_{i=1}^m\phi_i^\nu\| \Deltaxi{i}^\nu \|^2 \geq \frac{\mu}{D_{\mx}^2} \left(\sum_{i=1}^m \phi_i^\nu\big( U( \hbx_i^\nu) - U(\bx^\star)\big) - \frac{1}{\mu} \sum_{i=1}^m \phi_i^\nu \| \bdelta_i^\nu \|^2\right).\vspace{-0.1cm}
	\end{equation}
	Using (\ref{eq:bound_doubly_stoc_cvx}) in conjunction with  $U(\bx_i^{\nu+\frac{1}{2}})\leq \alpha U(\hbx_i^\nu)+ (1-\alpha) U(\bx_i^\nu)$ leads to   \vspace{-0.2cm} 
	\begin{equation}\label{eq:lower-bound_U_hat_TV}
	\alpha \, \sum_{i=1}^m \phi_i^\nu\left(U(\hbx_i^\nu) - U(\bx^\star)\right) \geq \optgap^{\nu+1} - (1-\alpha) \optgap^\nu.
	\end{equation}
	Combining (\ref{eq:delta_x_lb_a_TV}) with (\ref{eq:lower-bound_U_hat_TV}) yields the desired result (\ref{eq:delta_x_lb_TV}). 
\end{proof}

As last step, we upper bound   $\|\bdelta_{}^\nu\|^2$ in (\ref{eq:opt_err_1}) in terms of the consensus errors $\| \bx_\bot^\nu\|^2$ and $\| \by_\bot^\nu\|^2$.

\vspace{-0.1cm}
\begin{lemma}\label{lem:tracking_err_bound_TV}
	The tracking error $\|\bdelta_{}^\nu\|^2$ can be bounded as
	\begin{equation}\label{eq:delta_upper_by_consensus_TV}
	\|\bdelta_{}^\nu\|^2  \leq 8  L_{\mx}^2 \| \var{x}{\nu}\|^2 + 2  \| \var{y}{\nu} \|^2.,
	\end{equation} 
	where $L_{\mx}$ is defined in~\eqref{eq:prob_param_def}.
\end{lemma}
\begin{proof}    
	\begin{equation*}  
	\begin{aligned}  
	\begin{split}  \hspace{-1.3cm}
	\|\bdelta_{}^\nu\|^2  \overset{\eqref{eq:tracking_err_def}}{=} & \sum_{i=1}^m \|\nabla F(\bx_i^{\nu}) \pm  \wavg{\by}{\nu}-  \by_i^\nu\|^2  
	\\ 
	\overset{\eqref{eq:avg_y_eq_TV}}{=} & \frac{1}{m^2} \sum_{i=1}^m \Big\|\sum_{j=1}^m \nabla f_j(\bx_i^\nu)- \sum_{j=1}^m \nabla f_j(\bx_j^\nu) +   m\cdot \wavg{\by}{\nu}- m \cdot \by_i^\nu \Big\|^2
	\\
	 \overset{A2,\,\eqref{eq:prob_param_def}}{\leq} &\frac{1}{m^2}\sum_{i=1}^m \!\!\left( 2 m\sum_{j=1}^m L_{\mx}^2\| \bx_i^\nu - \bx_j^\nu \|^2 + 2 m^2 \|\wavg{\by}{\nu}-  \by_i^\nu \|^2\right)
	\\
	\leq & 8  L_{\mx}^2 \| \var{x}{\nu}\|^2 + 2  \| \var{y}{\nu} \|^2.
	\end{split}
	\end{aligned}\vspace{-0.4cm}
	\end{equation*}
	~
\end{proof}

The linear convergence of the optimality gap   up to consensus errors as stated in Proposition   follows readily multiplying (\ref{eq:delta_x_lb_TV}) by $\left(1 - \frac{\alpha}{2}\right) \tmu_{\mn} +  \frac{\alpha D_{\mn}}{2} - \frac{1}{2} \epsilon_{opt}$ and   adding with \eqref{eq:opt_err_1_TV}  
to cancel out  $\|\Deltaxi{}^\nu\|$, and   using (\ref{eq:delta_upper_by_consensus}) to bound $\|\bdelta_{}^\nu\|^2$. 

\section{Proof of Theorem~\ref{thm:step_size_condition2}}\label{app:step_size_condition2}
Following the same steps as in the proof of Theorem~\ref{thm:step_size_condition}, we derive the optimal $\epsilon_{opt}$ appearing  in $\eta (\alpha)$ and $\sigma(\alpha)$:
\begin{align}\label{eq:opt_star2}
\epsilon_{opt}^\star =\left(1 - \frac{\alpha}{2}\right) \tmu_{\mn} + \alpha  {D_{\mn}^\ell}/{2},
\end{align} 	
where $\alpha$ must  satisfy  
\begin{align}\label{eq:step_size_12}
\alpha < {2 \tmu_{\mn}}/({\tmu_{\mn} - D_{\mn}^\ell}).
\end{align}

Setting   $\epsilon_{opt}=\epsilon_{opt}^\star$ and denoting the corresponding $\PP(\alpha,z)$ as $\PP^\star(\alpha,z)$,
the expression of $\PP^\star (\alpha,1)$ reads
\begin{align}
\begin{split}
\PP^\star (\alpha ,1)  \triangleq {}& G^\star_P(\alpha)  \cdot C_1 \cdot 8 \phi_{ub}  L_{\mx}^2 \cdot  \frac{2c_0^2 \rho_B^2}{(1-\rho_{{B}})^2} \alpha^2
\\
& + \left(G^\star_P(\alpha)  \cdot 2 \phi_{ub}  \cdot C_1 + C_2\right)  \cdot 2m\phi_{lb}^{-2} L_{\mx}^2 \cdot \frac{2c_0^2\rho_B^2 }{(1-\rho_{\bar{B}})^2} \alpha^2 
\\
& + \left(G^\star_P(\alpha)  \cdot 2 \phi_{ub} \cdot C_1 + C_2\right) \cdot 8m\phi_{lb}^{-2} L_{\mx}^2   \cdot \frac{4 c_0^4 \rho_B^4 }{(1-\rho_{\bar{B}})^4} \alpha^2,
\end{split}
\end{align}
where  \vspace{-0.3cm}
\begin{equation}\label{eq:r_star_def_TV}
G_P^\star (\alpha) \triangleq \frac{\frac{D_{\mx}^2}{\mu} + \frac{1}{\mu} \cdot \left(\left(1 - \frac{\alpha}{2}\right) \tmu_{\mn} + \frac{D_{\mn}^\ell}{2} \alpha \right)^2}{ \left(\left(1 - \frac{\alpha}{2}\right) \tmu_{\mn} + \frac{D_{\mn}^\ell}{2} \alpha\right)^2 }.
\end{equation}
Since $\PP^\star (\bullet ,1)$ is continuous and monotonically increasing on $(0, 2 \tmu_{\mn} / (\tmu_{\mn} - D_{\mn}^\ell)$, with $\PP^\star (0 ,1) = 0$. A upperbound of $\alpha$ can be found by setting

\begin{align}\label{eq:expression_alpha2}
\alpha < \alpha_2 \triangleq & \left( G^\star_P \left(\frac{\tmu_{\mn}}{\tmu_{\mn} - D_{\mn}^\ell} \right)  \cdot C_1 \cdot 8 \phi_{ub}  L_{\mx}^2 \cdot \frac{2c_0^2 \rho_B^2}{(1-\rho_{{B}})^2} \alpha^2 \right.
\\
& \quad + \left(G^\star_P \left(\frac{\tmu_{\mn}}{\tmu_{\mn} - D_{\mn}^\ell}\right)  \cdot 2 \phi_{ub}  \cdot C_1 + C_2\right)  \cdot  2m\phi_{lb}^{-2} L_{\mx}^2 \cdot \frac{2c_0^2\rho_B^2 }{(1-\rho_{\bar{B}})^2} \alpha^2 
\\
& \left. \quad + \left(G^\star_P \left(\frac{\tmu_{\mn}}{\tmu_{\mn} - D_{\mn}^\ell}\right)  \cdot 2 \phi_{ub}  \cdot C_1 + C_2\right) \cdot 8m\phi_{lb}^{-2} L_{\mx}^2   \cdot \frac{4 c_0^4 \rho_B^4 }{(1-\rho_{\bar{B}})^4}\right)^{-1/2}.
\end{align}
Therefore, a valid $\bar{\alpha}$ is 
$\bar{\alpha} = \min\{  \tmu_{\mn} / (\tmu_{\mn} - D_{\mn}^\ell), \alpha_2\}$.

\section{Explicit expression of the linear rate in the  time-varying directed network setting}
\label{app:pf_linear_rate2}

The following theorem provides an explicit expression of the convergence rate in Theorem \ref{thm:step_size_condition2}, in terms of the step-size $\alpha$; the  constants $J$ and $A_{\frac{1}{2}}$ therein are defined in~\eqref{eq:expression_J_TV} and \eqref{eq:rate_condition_2_TV} with $\theta = 1/2$, respectively.

\begin{theorem}\label{thm:linear_rate_TV}
	In the setting of Theorem \ref{thm:step_size_condition2}, suppose that  the step-size $\alpha$ satisfies $\alpha \in (0,\alpha_{\mx})$, with $\alpha_{\mx}\triangleq  \min\{(1-\rho_{{B}})/{A}_{\frac{1}{2}},,\tmu_{\mn} /(\tmu_{\mn} - D_{\mn}^\ell),1 \}$. Then $\{U(\bx_i^\nu)\}$ converges to $U^\star$ at the  R-linear rate $\mathcal{O}(z^\nu)$, for all $i =1,\ldots,m$, where 
	\begin{align}\label{rate_exp}
	z =  
	\begin{cases}
	1 - J \cdot \alpha, & \text{if }\alpha \in \left( 0,\min\{\alpha^*,\alpha_{\mx}\}\right), \\
	\rho_B + A_{\frac{1}{2}} \alpha, & \text{if } \alpha \in [\min\{\alpha^*,\alpha_{\mx}\}, \alpha_{\mx} ).
	\end{cases}
	\end{align}
\end{theorem}

\begin{proof}
The proof follows similar steps as the proof of Theorem~\ref{thm:linear_rate}. For sake of simplicity, we used the same notation as therein.
We find the smallest $z$ satisfying (\ref{eq:bound_z2}) such  that $\PP(\alpha,z) < 1$,  for   {$\alpha\in (0, \alpha_{\mx})$},  and $\alpha_{\mx}\in (0,1)$ to be determined[recall that   $\PP(\alpha,z)$ is defined in  \eqref{def_P_alpha_TV}].  

Using exactly the same argument as Theorem~\ref{thm:linear_rate} we have the following two conditions on $z$:
\begin{align}\label{eq:rate_condition_12}
z \geq   \sigma(\alpha) +   \frac{ (\theta \cdot \alpha) \cdot \left( \left(1 - \frac{\alpha}{2}\right)\tmu_{\mn} + \frac{D_{\mn}^\ell}{2} \alpha - \frac{1}{2}  \epsilon_{opt} \right) }{\frac{ D_{\mx}^2}{\mu} +  \left(1 - \frac{\alpha}{2}\right)\tmu_{\mn} + \frac{D_{\mn}^\ell}{2} \alpha - \frac{1}{2}  \epsilon_{opt} }
\end{align}
for some $\theta \in (0,1)$; and
\begin{align}\label{eq:step_size_bound2}
\begin{split}
& G_{P}^\star(\alpha)\cdot \theta^{-1} \cdot G_X(z) \cdot C_1 \cdot 8 \phi_{ub}  L_{\mx}^2 \cdot \rho_B^2 \cdot \alpha^2 
\\
& + \left( G_{P}^\star(\alpha)\cdot \theta^{-1} \cdot 2 \phi_{ub} \cdot C_1 + C_2\right) \cdot G_Y(z) \cdot 2m \phi_{lb}^{-2} L_{\mx}^2 \cdot \rho_B^2 \cdot  \alpha^2 
\\
& + \left( G_{P}^\star(\alpha)\cdot \theta^{-1} \cdot 2 \phi_{ub} \cdot C_1 + C_2\right) \cdot G_Y(z) \cdot 8m\phi_{lb}^{-2} L_{\mx}^2 \cdot  G_X(z)  \cdot \rho_B^4 \cdot \alpha^2 <1.
\end{split}
\end{align}

Using the fact that $G_P^\star(\alpha)$ is monotonically increasing on $\alpha\in (0, 2\tmu_{\mn}/(\tmu_{\mn} - D_{\mn}^\ell))$, and restricting  $\alpha\in (0, \tmu_{\mn}/(\tmu_{\mn} - D_{\mn}^\ell)]$, a sufficient condition for~\eqref{eq:step_size_bound2} is 
\begin{align}\label{eq:step_size_bound_12}
\alpha \leq  \alpha(z) \triangleq \left(A_{1,\theta}\frac{1 }{z-\rho_{{B}} } + A_{2 ,\theta}\frac{1 }{z-\rho_{{B}}} + A_{3 ,\theta} \frac{1 }{(z-\rho_{{B}})^2}\right)^{-1/2},
\end{align} 
where $A_{1,\theta}$, $A_{2,\theta}$ and $A_{3,\theta}$ are constants defined as
\begin{align*}
A_{1,\theta} & \triangleq  G_{P}^\star \left(\frac{\tmu_{\mn}}{\tmu_{\mn} - D_{\mn}^\ell}\right) \cdot \theta^{-1}  \cdot C_1 \cdot 8 \phi_{ub}  L_{\mx}^2 \cdot \frac{2 c_0^2 \rho_B^2}{1 - \rho_B}
\\
A_{2,\theta} & \triangleq \left( G_{P}^\star \left(\frac{\tmu_{\mn}}{\tmu_{\mn} - D_{\mn}^\ell}\right) \cdot \theta^{-1} \cdot 2 \phi_{ub}  \cdot C_1 + C_2\right)  \cdot 2m \phi_{lb}^{-2} L_{\mx}^2 \cdot \frac{2 c_0^2 \rho_B^2}{1 - \rho_B}
\\
A_{3,\theta} & \triangleq  \left( G_{P}^\star \left(\frac{\tmu_{\mn}}{\tmu_{\mn} - D_{\mn}^\ell}\right) \cdot \theta^{-1} \cdot 2 \phi_{ub} \cdot C_1 + C_2\right)  \cdot 8m\phi_{lb}^{-2} L_{\mx}^2 \cdot   \frac{4 c_0^4 \rho_B^4 }{(1 - \rho_B)^2}.
\end{align*}
Lower bounding $z - \rho_B$ by $(z - \rho_B)^2$ we obtain
\begin{align}\label{eq:rate_condition_2_TV}
z \geq \rho_B + A_\theta \alpha,\quad \text{with}\quad A_{\theta} \triangleq \sqrt{A_{1,\theta} + A_{2,\theta} + A_{3,\theta}}.
\end{align}
Letting $\epsilon_{opt} = \epsilon_{opt}^\star$ in~\eqref{eq:rate_condition_12}, the condition reduces to 
\begin{align}\label{eq:rate_condition_32}
z \geq 1 -  \frac{    \tmu_{\mn} - \frac{\alpha}{2} (\tmu_{\mn} - D_{\mn}^\ell)   }{\frac{2 D_{\mx}^2}{\mu} +  \tmu_{\mn} - \frac{\alpha}{2} (\tmu_{\mn} - D_{\mn}^\ell)  }\cdot (1- \theta) \alpha.
\end{align}
Therefore,  the overall convergence rate can be upper bounded by $\mathcal{O}(\bar{z}^\nu)$, where
\begin{align}\label{eq:rate_expression_complete2}
\bar{z} = \inf_{\theta \in (0,1)}\max \left\{ \rho_B + A_\theta \alpha,  1 -  \frac{    \tmu_{\mn} - \frac{\alpha}{2} (\tmu_{\mn} - D_{\mn}^\ell)   }{\frac{2 D_{\mx}^2}{\mu} +  \tmu_{\mn} - \frac{\alpha}{2} (\tmu_{\mn} - D_{\mn}^\ell)  }\cdot (1- \theta) \alpha \right\},
\end{align}
with  $A_\theta$  defined in~\eqref{eq:rate_condition_2_TV}.

Finally, we further simplify \eqref{eq:rate_expression_complete2}. Letting $\theta = 1/2$ and using  $\alpha \in (0,  \tmu_{\mn} /(\tmu_{\mn} - D_{\mn}^\ell)]$, the second term in the max of \eqref{eq:rate_expression_complete2}  can be upper bounded by  
\begin{align}\label{eq:expression_J_TV}
1- \underbrace{\frac{\tmu_{\mn} \mu}{ 4 D_{\mx}^2 + \tmu_{\mn} \mu}\cdot \frac{1}{2}}_{\triangleq J} \alpha.
\end{align}
The condition  $\bar{z} <1$  imposes the following upper bound on $\alpha$: $\alpha < \alpha_{\mx} = \min\{(1-\rho_{{B}})/{A}_{\frac{1}{2}},\tmu_{\mn} /(\tmu_{\mn} - D_{\mn}^\ell),1 \}$. 
Eq.~\eqref{eq:rate_expression_complete2} then simplifies to  \eqref{rate_exp} with $\alpha^* = (1 - \rho_B)/(A_{\frac{1}{2}} + J)$ that  equates $1 - J \alpha$ and $\rho_B + A_{\frac{1}{2}} \alpha$.
\end{proof}

\section{Rate estimate using linearization surrogate  \eqref{eq:linear_surrogate} (time-varying directed network case)}
\label{app:linearization_rate_TV}

	\begin{corollary}[Linearization surrogates]\label{cor:linearization_rate_TV} 
	In the setting of Theorem~\ref{thm:linear_rate_TV}, let $\{\bx^\nu\}$ be the sequence generated by SONATA (Algorithm~\ref{alg:SONATA_TV}), using the surrogates   \eqref{eq:linear_surrogate} and   step-size   $\alpha = c\cdot \alpha_{\mx}$,  $c \in (0,1)$,    where $\alpha_{\mx} =  \min\{1,(1 - \rho_B)^2/(C_M \cdot \kappa_g (1 + \beta/L)^2)\}$ and $C_M$ is a constant defined in \eqref{C_M_def_1}. The number of iterations (communications) needed for $U(\bx_i^\nu)-U^\star\leq \epsilon$, $i\in [m]$,   is\vspace{-0.2cm}
	\begin{align}
	& \mathcal{O} \left(\kappa_g \log (1/\epsilon)\right), &\text{if } \quad  \frac{\rho_{B}}{(1 - \rho_{B})^2} \leq \frac{1}{C_M \cdot  \kappa_g \left(1 + \frac{\beta}{L} \right)^2},\label{eq:Case_I_linear_TV}
	\\
	&\mathcal{O} \left(\frac{ \big(\kappa_g + \beta/\mu\big)^2 \rho_{B}}{(1 - \rho_{B})^2}\,\log (1/\epsilon)\right), &\text{otherwise}.
	\label{eq:Case_II_linear_TV}
	\end{align}
\end{corollary}
\begin{proof}

According to 
Theorem~\ref{thm:linear_rate_TV},  the rate $z$ can be bounded as \begin{equation}\label{eq_rate_summary_TV} z \leq \max \{z_1, z_2\},\quad \text{ with }\quad 
z_1 \triangleq 1 - \alpha\cdot J \,\,\text{ and }\,\,  z_2 \triangleq \rho_{B} + A_{\frac{1}{2}} \alpha,
\end{equation} 
where $J$ and $A_{\frac{1}{2}}$ are defined in (\ref{eq:expression_J_TV}) and (\ref{eq:rate_condition_2_TV}), respectively. 

The proof consists in bounding properly $z_1$ and $z_2$ based upon  the surrogate   \eqref{eq:linear_surrogate} postulated in the corollary.  
We begin
particularizing the expressions of $J$ and $A_{\frac{1}{2}}$.   
Since $\nabla^2 \tf_i (\bx_i; \bx_i^\nu) = L$, one can set $\tmu_{\mn} =  L$, and (\ref{eq:upper-lower-hessian}) holds with     $D_{\mn}^\ell = 0$ and  $D_{\mx} = L -\mu$. Furthermore,   by Assumption~\ref{assump:homogeneity}, it follows that $ \beta\geq \lambda_{\max}(\nabla^2 f_i (\bx)) -L$, for all $\mathbf{x}\in \mathcal{K}$; hence, one can set  $L_{\mx} = L + \beta$.  
Next, we will substitute the above values into the expressions of $J$ and $A_{\frac{1}{2}}$. 

To do so, we need to particularize first the  quantities $G_P^\star \left(\frac{\tmu_{\mn}}{\tmu_{\mn} - D_{\mn}^\ell}\right)$ [cf. (\ref{eq:r_star_def_TV})], $C_1$   and  $C_2$ [cf. (\ref{eq:C1_C2_TV})]: \vspace{-0.2cm}
\begin{align*}
&G_P^\star \left(\frac{\tmu_{\mn}}{\tmu_{\mn} - D_{\mn}^\ell}\right) = G_P^\star \left(1\right)  = \frac{4 (L - \mu)^2 + L^2}{ \mu L^2 }, 
\\
& C_1 = \frac{6}{\mu\phi_{lb} L^2}\left(   (2 L - \mu)^2  + 4 (L + \beta)^2 \right), \quad \text{and}\quad   C_2 =\frac{4}{L^2}.
\end{align*}
Accordingly, the expressions of $J$ and $A_{\frac{1}{2}}$ read: \begin{equation}\label{eq:J_custom_TV}
J  =  
{ \frac{1}{2}}\frac{\kappa_g }{4 (\kappa_g - 1)^2 +\kappa_g}\in \left[\frac{1}{8\kappa_g}, \frac{1}{2}\right],
\end{equation}  
and 
\begin{align}\label{eq:A_1/2_upper_TV}
\begin{split}
& (A_{\frac{1}{2}})^2 
\\
= & G_P^\star(1) \cdot 2\cdot C_1\cdot 8 \phi_{ub} \cdot  L_{\mx}^2\cdot \frac{2c_0^2\rho_B^2}{1-\rho_B}
\\
&+\left(G_P^\star(1) \cdot 2\cdot C_1\cdot 2 \phi_{ub}+C_2\right)  \cdot 2m\phi_{lb}^{-2} L_{\mx}^2\cdot \frac{2c_0^2\rho_B^2}{1-\rho_B}
\\
&+\left(G_P^\star(1) \cdot 2\cdot C_1\cdot 2 \phi_{ub}+C_2\right)  \cdot 8m\phi_{lb}^{-2} L_{\mx}^2\cdot \frac{4c_0^4\rho_B^4}{(1-\rho_B)^2}
\\
\leq & \left(G_P^\star(1) \cdot 2\cdot C_1\cdot 12 \phi_{ub}+C_2\right)  \cdot 8m\phi_{lb}^{-2} L_{\mx}^2\cdot \frac{4c_0^4\rho_B^2}{(1-\rho_B)^2}
\\
\leq & \left[\frac{4 (L - \mu)^2 + L^2}{ \mu L^2 }\cdot \frac{12}{\mu\phi_{lb} L^2}\left(   (2 L - \mu)^2  + 4 (L + \beta)^2 \right)\cdot 12\phi_{ub}+\frac{4}{L^2}\right]
\\
&\qquad \qquad \qquad \qquad\qquad \qquad \qquad \qquad\qquad  \cdot 8m\phi_{lb}^{-2} \left(L+\beta\right)^2\cdot \frac{4c_0^4\rho_B^2}{(1-\rho_B)^2}
\\
\leq & C_M^2 \cdot  \kappa_g^2 \left(1 + \frac{\beta}{L} \right)^4\cdot \frac{\rho_B^2}{(1-\rho_B)^2},
\end{split}
\end{align}
where 
\begin{equation}
\label{C_M_def_1}
C_M\triangleq 608\cdot\phi_{lb}^{-1}\cdot c_0~\sqrt{\frac{\phi_{ub}}{\phi_{lb}} \cdot m},
\end{equation}
and in the first inequality we have used the fact that  $\phi_{lb}<1$ and $c_0>1$, and the last inequality holds since $\kappa_g \geq 1$ and $\frac{\phi_{ub}}{\phi_{lb}}\geq 1$. Using the above expressions, in the sequel we upperbound  $z_1$ and $z_2$. 


By (\ref{eq:A_1/2_upper_TV}),   we have 	
\begin{equation} z_2\leq \bar{z}_2 \triangleq \rho_B + \alpha M\cdot \frac{\rho_B}{1-\rho_B},\quad \text{with}\quad 
M \triangleq C_M \cdot \kappa_g (1 + \beta/L)^2.
\end{equation}
Since   $\alpha \in (0,1]$ must be chosen so that $z \in (0,1]$, we impose $\max\{z_1, \bar{z}_2\}<1$, implying 
$\alpha \leq \min\{J^{-1}, (1-\rho_B)^2/(M\rho_B),1 \}$. 
Since $J^{-1} >1$ [cf. \eqref{eq:J_custom_TV}],    the  condition on $\alpha$ reduces to $\alpha \leq \alpha_{\mx} \triangleq  \min\{1,(1-\rho_B)^{2}/ (M\rho_B)\}<1$. Choose  $\alpha = c \cdot \alpha_{\mx}$, for some given  $c \in (0,1)$. Depending on the value of $\rho_B$, either $\alpha_{\mx}=1$ or $\alpha_{\mx} =(1-\rho_B)^{2}/ (M\rho_B)$.

\textbf{$\bullet$ Case I:}  $\alpha_{\mx}=1$. This corresponds to the case  $ M \rho_B \leq (1 - \rho_B)^2$. Note that, we also have  $\rho_B \leq 1/C_M$,  otherwise $M \rho_B \geq C_M \,\kappa_g \, \rho_B >1 > (1 - \rho_B)^2$. 
In this setting,   $\alpha = c\cdot \alpha_{\mx} =c$, and  \vspace{-0.1cm}\begin{align*}
z_1 &=   1 - c \cdot J,
\\
\bar{z}_2 & = \rho_B + c M \cdot \frac{\rho_B}{1-\rho_B}\overset{(a)}{\leq}
1 - (1-c)(1 - \rho_B) 
\\
&\overset{(b)}{\leq} 
1 - (1-c)\left(1 - \frac{1}{C_M}\right),
\end{align*}
where in (a) we used $ M \rho_B \leq (1 - \rho_B)^2$ and (b) follows from  $\rho_B \leq 1/C_M$. 

Therefore, $z$ can be bounded as
\begin{align}
\begin{split}
z  \leq \max \{z_1, \bar{z}_2\} & \leq 	 1 - c\cdot (1-c)\left(1 - \frac{1}{C_M}\right)\cdot J
\\
& \leq 	 1 - c\cdot (1-c)\left(1 - \frac{1}{C_M}\right)\cdot \frac{1}{8\kappa_g}.
\end{split}
\end{align} 

\textbf{$\bullet$ Case II:}	$ \alpha_{\mx}= (1-\rho_B)^2/(M\rho_B)$. This corresponds to $ M \rho_B> (1 - \rho_B)^2$. We have $\alpha=c\cdot  \alpha_{\mx}$,
\begin{equation*}
\begin{aligned}
z_1 = & 1 - \frac{J \, c}{{M \rho_B}} \cdot   {(1-\rho_B)^2},
\\
\bar{z}_2 = & 1 - (1-c)\left(1 -\rho_B\right).
\end{aligned}
\end{equation*}	
Now we can bound $z$. Since $Jc/(M\rho_B)<1$ (by the same reasoning as in proof of Proposition \ref{cor:linearization_rate}), 
\begin{equation}
\label{z_bound_Case2}
\begin{aligned}
z  & \leq \max\{z_1,\bar{z}_2\} \leq 1 - \frac{Jc}{{M \rho_B}} \cdot  \left(1 - c \right)    {(1-\rho_B)^2}
\\
& \overset{\eqref{eq:J_custom_TV}}{\leq } 1 - \frac{c\left(1 - c \right)}{{8C_M }} \cdot \frac{(1-\rho_B)^2}{\kappa_g^2(1+\beta/L)^2\rho_B}
\\
& = 1 - \frac{c\left(1 - c \right)}{{8C_M }} \cdot \frac{(1-\rho_B)^2}{(\kappa_g+\beta/\mu)^2\rho_B}.
\end{aligned}
\end{equation}
~	
\end{proof}

\section{Rate estimate using local $f_i$ \eqref{eq:f_surrogate} (time-varying directed network case)}
\label{app:f_rate_TV}

\begin{corollary}[Local $f_i$, $\beta\leq \mu$]\label{cor:f_rate_1_TV}  
	Instate assumptions of Theorem~\ref{thm:linear_rate_TV} and suppose $\beta \leq \mu$. Consider   SONATA (Algorithm~\ref{alg:SONATA_TV}) using the surrogates~\eqref{eq:f_surrogate} and step-size $\alpha = c\cdot \alpha_{\mx}$,  $c \in (0,1)$, with $\alpha_{\mx} = \min\{1,(1 - \rho_B)^2/(\tilde{M}_2\rho_B)\} $ where  $\tilde{M}_2 = 1087 \tilde{C}_M\left(1 + \frac{\beta}{\mu}\right)^2\left(\kappa_g + \frac{\beta}{\mu}\right)^2$ and the constant $\tilde{C}_M$ is defined in \eqref{C_M_def_2}. The number of iterations  (communications) needed for $U(\bx_i^\nu)-U^\star\leq \epsilon$, $i\in [m]$,   is\vspace{-0.2cm}
	\begin{align}
	& \mathcal{O} \left(1\cdot \log (1/\epsilon) \right), & \text{if} \quad \frac{\rho_B}{(1 - \rho_B)^2} \leq \frac{1}{1087\cdot \tilde{C}_M \cdot  \left(1 + \frac{\beta}{\mu}\right)^2\left(\kappa_g + \frac{\beta}{\mu}\right)^2},
	\\
	&\mathcal{O} \left( \frac{\kappa_g^2 \rho_B}{(1 - \rho_B)^2} \,\log (1/\epsilon)\right), &\text{otherwise}.
	\end{align}
\end{corollary}

\begin{corollary}[Local $f_i$, $\beta>\mu$]\label{cor:f_rate_2_TV}  
	Instate assumptions of Theorem~\ref{thm:linear_rate_TV} and suppose $\beta > \mu$.  Consider   SONATA  (Algorithm~\ref{alg:SONATA_TV}) using the surrogates~\eqref{eq:f_surrogate}  and step-size    $\alpha = c\cdot \alpha_{\mx}$, $c \in (0,1)$, where $\alpha_{\mx} =\min\{1,(1 - \rho_B)^2/(\tilde{M}_1\rho_B)\}$ with  $\tilde{M}_1 = 1428 \tilde{C}_M \left(1 + \frac{L}{\beta}\right) \left(\kappa_g + \frac{\beta}{\mu}\right)$ and the constant $\tilde{C}_M$ is defined in \eqref{C_M_def_2}.  The number of iterations (communications) needed for $U(\bx_i^\nu)-U^\star\leq \epsilon$, $i\in [m]$,   is\vspace{-0.2cm}
	\begin{align}
	& \mathcal{O} \left(\frac{\beta}{\mu}\log (1/\epsilon) \right), & \text{if} \quad \frac{\rho_B}{(1 - \rho_B)^2} \leq \frac{1}{1428\cdot \tilde{C}_M \cdot  \left(1 + \frac{L}{\beta}\right) \left(\kappa_g + \frac{\beta}{\mu}\right)},
	\\
	& \mathcal{O} \left( \frac{\left(\kappa_g + (\beta/\mu)\right)^2 \rho_B }{(1 - \rho_B)^2}\,\log (1/\epsilon) \right), &\text{otherwise.}
	\label{eq:Case_II_surrogate_beta_big_mu_TV}
	\end{align}
\end{corollary}

\begin{proof}
In the setting of the corollary, we have:  $\nabla^2 \tf_i (\bx; \by) = \nabla^2 f_i (\bx) + \beta \bI$,  for all $\mathbf{y}\in \KK$;  $\nabla^2 f_i(\mathbf{x}) \succeq \mathbf{0}$, for all   $\mathbf{x}\in \KK$; and,  by Assumption~\ref{assump:homogeneity}, 
$\mathbf{0}  \preceq \nabla^2 \tf_i (\bx ,\by) - \nabla^2 F(\bx) \preceq 2\beta\mathbf{I}$, for all $\bx, \by \in \KK$. Therefore,  
we can set $D_{\mn}^\ell = 0$, $D_{\mx} = 2 \beta$,   $\tmu_{\mn} = \beta + (\mu - \beta)_+=\max\{\beta, \mu\}$, and $L_{\mx} = L + \beta$.

Using these values, 
$G_P^\star \left(\frac{\tmu_{\mn}}{\tmu_{\mn} - D_{\mn}^\ell}\right)$, $C_1$, and $C_2$   can be simplified as follows:
\begin{align*}
&G_P^\star \left(\frac{\tmu_{\mn}}{\tmu_{\mn} - D_{\mn}^\ell}\right) = G_P^\star \left(1\right)  = \frac{ 16 \beta^2 +  \max\{\beta, \mu\}^2}{ \mu \max\{\beta, \mu\}^2 },  \\
& C_1 = \frac{6}{\mu\phi_{lb}}\left(  \left( \frac{2 \beta}{\max\{\beta, \mu\}} + 1\right)^2 + \frac{4 (L + \beta)^2}{\max\{\beta, \mu\}^2 } \right), \quad \text{and}\quad C_2 =\frac{4}{\max\{\beta, \mu\}^2}.
\end{align*}

Accordingly, the expressions of $J$ and $A_{\frac{1}{2}}$ read: \vspace{-0.2cm}
\begin{equation}\label{eq:J_II_TV}J = \frac{1}{2}\frac{1}{1 + 16 \left(\frac{\beta}{\mu}\right) \cdot \min \left\{1, \frac{\beta}{\mu}\right\}}, 
\end{equation} and 
\begin{align*}
\begin{split}
& (A_{\frac{1}{2}})^2 
\\
\leq & \left(G_P^\star(1) \cdot 2\cdot C_1\cdot 12 \phi_{ub}+C_2\right)  \cdot 8m\phi_{lb}^{-2} L_{\mx}^2\cdot \frac{4c_0^4\rho_B^2}{(1-\rho_B)^2}
\\
\leq  & \left(\frac{ 16 \beta^2 +  \max\{\beta, \mu\}^2}{ \mu \max\{\beta, \mu\}^2 } \cdot \frac{12}{\mu\phi_{lb}}\left(  \left( \frac{2 \beta}{\max\{\beta, \mu\}} + 1\right)^2 + \frac{4 (L + \beta)^2}{\max\{\beta, \mu\}^2 } \right)\cdot 12 \phi_{ub}+\frac{4}{\max\{\beta, \mu\}^2}\right) 
\\
&\qquad \qquad \qquad \qquad\qquad \qquad \qquad \qquad\qquad  \cdot 8m\phi_{lb}^{-2} (L+\beta)^2\cdot \frac{4c_0^4\rho_B^2}{(1-\rho_B)^2}
\\
\leq 
&
\begin{cases}
\left( 2448 \cdot \frac{\phi_{ub}}{\phi_{lb}} \cdot \left(9 + 4\left( 1+\frac{L}{\beta}\right)^2 \right) \cdot  \left(\kappa_g + \frac{\beta}{\mu}\right)^2 + 4 \left(1 + \frac{L}{\beta}\right)^2 \right) \cdot 8m\phi_{lb}^{-2} \cdot \frac{4c_0^4\rho_B^2}{(1-\rho_B)^2}, &  \beta>\mu,
\\
\left( 144 \cdot \frac{\phi_{ub}}{\phi_{lb}} \cdot  \left(\frac{16 \beta^2}{\mu^2} + 1\right) \left(\kappa_g + \frac{\beta}{\mu}\right)^2\left(  \left( \frac{2 \beta}{\mu} + 1\right)^2 + 4 \left(\kappa_g + \frac{\beta}{\mu}\right)^2 \right)  + 4 \left(\kappa_g + \frac{\beta}{\mu}\right)^2  \right)  & 
\\
\qquad\qquad\qquad\qquad\qquad \qquad\qquad\qquad\qquad\qquad \qquad\qquad\qquad\qquad\qquad \cdot 8m\phi_{lb}^{-2} \cdot \frac{4c_0^4\rho_B^2}{(1-\rho_B)^2}, & \beta\leq \mu;
\end{cases}\\
\leq & \tilde{M}^2\frac{ \rho_B^2}{1-\rho_B^2},
\end{split}
\end{align*}
where\vspace{-0.2cm}
\begin{align} 
\tilde{M} \triangleq  
\begin{cases}\label{eq:def_M_surrogate_TV}
1428\cdot \tilde{C}_{M}\left(1 + \frac{L}{\beta}\right) \left(\kappa_g + \frac{\beta}{\mu}\right), & \beta>\mu,
\\
1087\cdot \tilde{C}_{M}\left(1 + \frac{\beta}{\mu}\right)^2\left(\kappa_g + \frac{\beta}{\mu}\right)^2, &  \beta\leq \mu,
\end{cases}
\end{align}
and
\begin{equation}
\label{C_M_def_2}
\begin{aligned}
\tilde{C}_{M} & \triangleq c_0^2\phi_{lb}^{-1}\sqrt{\frac{\phi_{ub}}{\phi_{lb}} \cdot m},
\end{aligned}
\end{equation}
and the last inequality holds since $\kappa_g \geq 1$ and $\frac{\phi_{ub}}{\phi_{lb}}\geq 1$.

Similarly, we bound $z \leq \max \{z_1, z_2\}$ as
\begin{equation}\label{eq_rate_summary_2_TV} z \leq \max \{z_1, \bar{z}_2\},\quad \text{ with }\quad 
z_1 \triangleq 1 - \alpha\cdot J \,\,\text{ and }\,\,  \bar{z}_2   \triangleq \rho_B + \alpha\,\tilde{M}\cdot \frac{\rho_B}{1-\rho_B},
\end{equation} where $J$ and $\tilde{M}$ are now given by (\ref{eq:J_II_TV}) and (\ref{eq:def_M_surrogate_TV}),  respectively. 
For  $\max \{z_1, z_2\}<1$, we require $\alpha \leq \alpha_{\mx} \triangleq   \min\{1,(1-\rho_B)^2/(\tilde{M}\rho_B) \}$,   and choose  $\alpha = c \cdot \alpha_{\mx}$, with arbitrary $c \in (0,1)$.

\textbf{$\bullet$ Case I:} $\alpha_{\mx}=1$. This correspond to $ \tilde{M} \rho_B \leq (1 - \rho_B)^2$,  $\alpha = c$, hence,
\begin{equation*}
z_1  = 1 - c \cdot J \quad \text{and} \quad \bar{z}_2    \leq1 - \left(1 - c \right) (1 - \rho_B).
\end{equation*}
Since $\tilde{M} \geq  1087\cdot \tilde{C}_M$ and $(1 - \rho_B)^2 \leq 1$, it must be $\rho_B \leq 1/(1087\cdot \tilde{C}_M)$.
Therefore, the rate $z$ can be bounded as
\begin{align*}
\begin{split}
z  \leq \max \{z_1, \bar{z}_2\} & \leq1 - c \cdot  \left(1 - c \right)\cdot  J \cdot (1 - \rho_B)
\\
& \leq  1 - c \cdot  \left(1 - c\right)\cdot \left(1 - \frac{1}{1087\cdot \tilde{C}_M}\right) \cdot \frac{1}{34} \cdot \frac{\mu}{\beta},
\end{split}
\end{align*} 
when $\beta>\mu$, and 
\begin{align*}
\begin{split}
z  
& \leq   1 - c \cdot  \left(1 - c \right)\cdot \left(1 - \frac{1}{1087\cdot \tilde{C}_M}\right) \cdot\frac{1}{2+32\left(\frac{\beta}{\mu}\right)^2},
\end{split}
\end{align*} 
when $\beta\leq \mu$.  


\textbf{$\bullet$ Case II:} 	$ \alpha_{\mx}= (1-\rho_B)^2/(\tilde{M}\rho_B)$. This corresponds to $\tilde{ M} \rho_B> (1 - \rho_B)^2$. We have $\alpha=c\cdot  \alpha_{\mx}$. Similarly to inequality \eqref{z_bound_Case2} in proof of Corollary \ref{cor:linearization_rate_TV}, we have
\begin{align*}
\begin{split}
z  & \leq \max\{z_1,\bar{z}_2\} \leq 1 - \frac{c\,J}{{\tilde{M}\rho_B }} \cdot  \left(1 - c \right) (1-\rho_B)^2,
\end{split}
\end{align*}
which yields 
\begin{align*}
\begin{split}
z  & \leq 1 - \frac{c\left(1 - c\right)  }{{1428\cdot 34 \cdot   \tilde{C}_{M} }} \cdot \frac{(1-\rho_B)^2}{(\kappa_g+\beta/\mu)^2\rho_B}     ,
\end{split}
\end{align*}	
when $\beta>\mu$, and 
\begin{align*}
\begin{split}
z  & \leq 1 - \frac{c\left(1 - c \right) }{{34\cdot 1087\cdot \tilde{C}_{M} }} \cdot \frac{(1-\rho_B)^2}{(1+\beta/\mu)^2(\kappa_g+\beta/\mu)^2\rho_B}      
\\
& \leq 1 - \frac{c\left(1 - c\right) }{{34\cdot 16\cdot 1087\cdot \tilde{C}_{M} }} \cdot \frac{(1-\rho_B)^2}{\kappa_g^2\rho_B},
\end{split}
\end{align*}
when $\beta\leq \mu$; the last inequality holds due to $\kappa_g\geq 1$.

\end{proof}	

\end{document}